\documentclass[a4paper, 10pt, twoside, notitlepage]{amsart}

\usepackage[utf8]{inputenc}
\usepackage{xcolor}
\usepackage{amsmath} 
\usepackage{amssymb} 
\usepackage{amsthm}
\usepackage{geometry}
\usepackage{graphicx}
\usepackage{mathtools}
\usepackage{esint}
\usepackage{enumitem}
\usepackage{comment}
\usepackage[colorlinks=true,linkcolor=blue]{hyperref}

\newtheorem{thm}{Theorem}[section]
\newtheorem{cor}[thm]{Corollary}
\newtheorem{lem}[thm]{Lemma}
\newtheorem{prop}[thm]{Proposition}

\newtheorem{defn}[thm]{Definition}
\newtheorem{rem}[thm]{Remark}

\newtheorem{assump}[thm]{Assumption}

\theoremstyle{definition}

\numberwithin{equation}{section}

\newtheorem{remark}[thm]{Remark}

\newcommand{\C}{\mathbb{C}}

\newcommand{\N}{\mathbb{N}}

\newcommand{\R}{\mathbb{R}}

\newcommand{\supp}{\operatorname{supp}}

\newcommand{\Z}{\mathbb{Z}}

\newcommand {\D} {\Delta}

\DeclareMathOperator {\dist} {dist}

\DeclareMathOperator{\F} {\mathcal{F}}

\def\hat{\widehat}
\def\tilde{\widetilde}
\def \bfo {\begin {eqnarray*} }
\def \efo {\end {eqnarray*} }
\def \ba {\begin {eqnarray*} }
\def \ea {\end {eqnarray*} }
\def \beq {\begin {eqnarray}}
\def \eeq {\end {eqnarray}}
\def \supp {\hbox{supp }}

\def \dist {\hbox{dist}}

\def \det {\hbox{det}}

\def \p {\partial}

\def\F{{\mathcal F}}

\def\hat{\widehat}
\def\tilde{\widetilde}
\def \bfo {\begin {eqnarray*} }
\def \efo {\end {eqnarray*} }
\def \ba {\begin {eqnarray*} }
\def \ea {\end {eqnarray*} }
\def \beq {\begin {eqnarray}}
\def \eeq {\end {eqnarray}}

\def \dist {\hbox{dist}}

\def \det {\hbox{det}}

\def \p {\partial}

\allowdisplaybreaks

\begin{document}

\title[Fractional anisotropic Calder\'on problem with external data]{Fractional anisotropic Calder\'on problem with external data}

\author[Feizmohammadi]{Ali Feizmohammadi}

\address{A. Feizmohammadi, Department of Mathematics\\
University of Toronto\\ 
Road Deerfield Hall\\
 3008K Mississauga ON L5L 1C6}
\email{ali.feizmohammadi@utoronto.ca}

\author[Ghosh]{Tuhin Ghosh}
\address
        {T. Ghosh, Harish-Chandra Research Institute, Homi Bhabha National Institute\\
Chhatnag Road, Jhunsi\\
Prayagraj (Allahabad) 211 019, 
India}

\email{tuhinghosh@hri.res.in}

\author[Krupchyk]{Katya Krupchyk}
\address
        {K. Krupchyk, Department of Mathematics\\
University of California, Irvine\\
CA 92697-3875, USA }

\email{katya.krupchyk@uci.edu}

\author[R\"uland]{Angkana R\"uland}
\address
        {A. R\"uland, Institute for Applied Mathematics and Hausdorff Center for Mathematics\\
         University of Bonn\\
          Endenicher Allee 60, 53115 Bonn, Germany}

\email{rueland@uni-bonn.de}

\author[Sj\"ostrand]{Johannes Sj\"ostrand}
\address
        {J. Sj\"ostrand, Universit\'e Bourgogne Europe\\
         CNRS, IMB UMR 5584\\
          F-21000 Dijon, France}

\email{johannes.sjostrand@u-bourgogne.fr}

\author[Uhlmann]{Gunther Uhlmann}

\address
       {G. Uhlmann, Department of Mathematics\\
       University of Washington\\
       Seattle, WA  98195-4350\\
       USA}
\email{gunther@math.washington.edu}

\maketitle

\begin{abstract}
In this paper, we solve the fractional anisotropic Calder\'on problem with external data in the Euclidean space, in dimensions two and higher, for smooth Riemannian metrics that agree with the Euclidean metric outside a compact set. Specifically, we prove that the knowledge of the partial exterior Dirichlet--to--Neumann map for the fractional Laplace-Beltrami operator, given on arbitrary open nonempty sets in the exterior of the domain in the Euclidean space, determines the Riemannian metric up to diffeomorphism, fixing the exterior. We provide two proofs of this result: one relies on the heat semigroup representation of the fractional Laplacian and a pseudodifferential approach, while the other is based on a variable-coefficient elliptic extension interpretation of the fractional Laplacian.

\end{abstract}

\section{Introduction and statement of results}

The classical anisotropic Calder\'on problem asks whether the electrical conductivity matrix of a medium can be determined, up to a change of variables, from current and voltage measurements taken along the boundary of the medium (see \cite{Calderon_1980, Uhlmann_2014}). This problem has the following geometric formulation (see \cite{Uhlmann_2014}): Let $(M,g)$ be a smooth compact Riemannian manifold of dimension $n \geq 2$ with boundary, and consider the Dirichlet problem for the Laplace-Beltrami operator:
\[
-\Delta_g u = 0 \quad \text{in} \quad M, \quad 
u|_{\partial M} = f.
\]
The geometric version of the classical anisotropic Calder\'on problem asks whether the Riemannian metric $g$, and more generally the manifold $(M, g)$, can be recovered from the Dirichlet--to--Neumann map $\Lambda_g: f \mapsto \partial_\nu u|_{\partial M}$. A natural obstruction to uniqueness arises from the following gauge invariance: if $\psi: M \to M$ is a smooth diffeomorphism such that $\psi|_{\partial M} = \text{Id}$, then $\Lambda_{\psi^* g} = \Lambda_g$, where $\psi^* g$ is the pullback of $g$ (see \cite{Kohn_Vogelius_1983, Lee_Uhlmann_1989}).

In dimension $n = 2$, the geometric version of the classical anisotropic Calder\'on problem, with an additional obstruction arising from the conformal invariance of the Laplacian, was solved in \cite{Lassas_Uhlmann_2001}; see also \cite{Nachman_1996}. However, in dimensions $n \geq 3$, this remains one of the major open problems in the field of inverse problems. The central conjecture posits that gauge invariance is the sole obstruction to uniqueness. This conjecture has been resolved in the real-analytic case (see \cite{Lassas_Taylor_Uhlmann_2003, Lassas_Uhlmann_2001, Lee_Uhlmann_1989}), but it remains widely open for smooth metrics. For state-of-the-art results in the smooth case, we refer to \cite{DKSaloU_2009, DKurylevLS_2016} and references therein.

Beginning with the work \cite{Ghosh_Salo_Uhlmann_2020}, there has been a surge of activity in the study of the fractional Calder\'on problem. In particular, the fractional anisotropic Calder\'on problem on a closed smooth connected Riemannian manifold with source--to--solution measurements has been solved in complete geometric generality in \cite{FGKU_2021}; see also \cite{Feizmohammadi_2021, Choulli_Ouhabaz_2023, Ruland_2023}.

The purpose of this paper is to provide a solution to the more challenging fractional anisotropic Calder\'on problem with external data in noncompact settings for smooth Riemannian metrics, as introduced in \cite{Ghosh_Salo_Uhlmann_2020}.
To state the problem, let $g = (g_{jk})_{j,k=1}^n$ be a $C^\infty$ Riemannian metric on $\mathbb{R}^n$, for $n \geq 2$, which agrees with the Euclidean metric outside a compact set. Then $(\mathbb{R}^n, g)$ is a complete Riemannian manifold (see \cite[Example 1.2]{Hassell_Tao_Wunsch_2006}). Associated with the Riemannian metric $g$ is the Laplace-Beltrami operator
\[
-\Delta_g = -|g|^{-1/2} \sum_{j,k=1}^n \partial_{x_j} \left( |g|^{1/2} g^{jk} \partial_{x_k} \right),
\]
where $(g^{jk}) = (g_{jk})^{-1}$ and $|g| = \det(g_{jk})$. When viewed as an unbounded operator on $L^2(\mathbb{R}^n; \allowbreak dV_g)$, the Laplace-Beltrami operator $-\Delta_g$ is nonnegative essentially self-adjoint from $C^\infty_0(\mathbb{R}^n)$, with the domain
\[
\mathcal{D}(-\Delta_g) = \{ u \in L^2(\mathbb{R}^n; dV_g) : -\Delta_g u \in L^2(\mathbb{R}^n; dV_g) \} = H^2(\mathbb{R}^n),
\]
(see \cite[Theorem 5.2.3]{Davies_book}, \cite[Theorem 2.4]{Strichartz_1983}). Here, $dV_g = |g|^{1/2} dx$ denotes the Riemannian volume element, and $H^2(\mathbb{R}^n)$ denotes the standard Sobolev space on $\mathbb{R}^n$. The spectrum of $-\Delta_g$ is purely absolutely continuous and is given by $[0, \infty)$ (see \cite[Section 1.3]{Guillarmou_Hassell_Krup_2020} and the references therein).

Let $\alpha \in (0,1)$. Using the functional calculus for self-adjoint operators, we define the fractional Laplacian $(-\Delta_g)^\alpha$ as an unbounded self-adjoint operator on $L^2(\mathbb{R}^n; dV_g)$ with the domain $\mathcal{D}((-\Delta_g)^\alpha) = H^{2\alpha}(\mathbb{R}^n)$ (see \cite[Theorem 4.4]{Strichartz_1983}, and \cite[Theorem 4.2, page 22]{Taylor_book_pseudo}).

Let $\Omega \subset \mathbb{R}^n$ be a bounded open set with $C^\infty$ boundary, and let $W_1, W_2 \subset \Omega_e := \mathbb{R}^n \setminus \overline{\Omega}$ be open and nonempty. Let $f \in C^\infty_0(W_1)$ and consider the following exterior Dirichlet problem for the fractional Laplacian:
\begin{equation}
\label{eq_int_1}
\begin{cases}
(-\Delta_g)^\alpha u = 0 & \text{in} \quad \Omega,\\
u = f & \text{in} \quad \Omega_e.
\end{cases}
\end{equation}
It is established in Proposition \ref{prop_well-posedness} below that the problem \eqref{eq_int_1} has a unique energy solution $u = u^f \in H^\alpha(\mathbb{R}^n)$. Furthermore, Proposition \ref{prop_eq_2_6} below shows that the operator $(-\Delta_g)^\alpha$ extends to a linear continuous map: $H^\alpha(\mathbb{R}^n) \to H^{-\alpha}(\mathbb{R}^n)$. Associated with \eqref{eq_int_1}, we introduce the partial exterior Dirichlet--to--Neumann map,
\begin{equation}
\label{eq_int_1_exterior_DN}
\Lambda_g^{W_1, W_2}: C^\infty_0(W_1) \to H^{-\alpha}(W_2), \quad f \mapsto ((-\Delta_g)^\alpha u^f)|_{W_2}.
\end{equation}
Here,
\[
H^{-\alpha}(W_2) = \{ u|_{W_2} : u \in H^{-\alpha}(\mathbb{R}^n) \} \subset \mathcal{D}'(W_2)
\]
is the standard Sobolev space on the open set $W_2$.

The fractional anisotropic Calder\'on problem with external data asks whether the metric $g$ in $\mathbb{R}^n$ can uniquely be determined from the knowledge of the partial exterior Dirichlet--to--Neumann map $\Lambda_g^{W_1,W_2}$, given on arbitrary open nonempty sets $W_1, W_2 \subset \Omega_e$. There is an obstruction to uniqueness in this problem. Indeed, if $\Phi: \mathbb{R}^n \to \mathbb{R}^n$ is a $C^\infty$ diffeomorphism such that $g_1 = \Phi^* g_2$ on $\mathbb{R}^n$ and $\Phi(x) = x$ for all $x \in \Omega_e$, then $\Lambda_{g_1}^{W_1, W_2} = \Lambda_{g_2}^{W_1, W_2}$; see Proposition \ref{prop_app_obstruction} below.

The following is our main result, which shows that the diffeomorphism invariance mentioned above is the only obstruction to uniqueness.
\begin{thm}
\label{thm_main}
Let $\alpha \in (0,1)$. Let $g_1$ and $g_2$ be $C^\infty$ Riemannian metrics on $\mathbb{R}^n$, $n \geq 2$, which agree with the Euclidean metric outside of a compact set. Let $\Omega \subset \mathbb{R}^n$ be a bounded open set with $C^\infty$ boundary such that $\Omega_e$ is connected, and let $W_1, W_2 \subset \Omega_e$ be open and nonempty. Assume that 
\begin{equation}
\label{eq_int_2}
g_1|_{\Omega_e} = g_2|_{\Omega_e}.
\end{equation}
Assume furthermore that 
\begin{equation}
\label{eq_int_3}
\Lambda_{g_1}^{W_1,W_2}(f) = \Lambda_{g_2}^{W_1,W_2}(f),
\end{equation}
for all $f \in C^\infty_0(W_1)$. 
Then there exists a $C^\infty$ diffeomorphism $\Phi: \mathbb{R}^n \to \mathbb{R}^n$ such that $g_1=\Phi^*g_2$ on $\mathbb{R}^n$ and $\Phi(x) = x$ for all $x \in \Omega_e$.
\end{thm}

We have the following immediate consequence of Theorem \ref{thm_main} in the case when the metrics $g_1$ and $g_2$ are assumed to be in the same conformal class. 
\begin{cor}
\label{cor_main}
Let $\alpha \in (0,1)$. Let $g_1$ and $g_2$ be $C^\infty$ Riemannian metrics on $\mathbb{R}^n$, $n \geq 2$, which agree with the Euclidean metric outside of a compact set, and such that $g_1 = c g_2$ for some $0 < c \in C^\infty(\mathbb{R}^n)$. Let $\Omega \subset \mathbb{R}^n$ be a bounded open set with $C^\infty$ boundary such that $\Omega_e$ is connected, and let $W_1, W_2 \subset \Omega_e$ be open and nonempty. Assume that $c = 1$ on $\Omega_e$, and that 
\[
\Lambda_{g_1}^{W_1, W_2}(f) = \Lambda_{g_2}^{W_1, W_2}(f),
\]
for all $f \in C^\infty_0(W_1)$. Then $c = 1$ in $\mathbb{R}^n$.
\end{cor}
Corollary \ref{cor_main} follows from Theorem \ref{thm_main}, combined with the fact that if a $C^\infty$ diffeomorphism $\Phi: \Omega \to \Omega$ satisfies $\Phi^* g_2 = c g_2$ and $\Phi|_{\partial \Omega} = \text{Id}$, then $\Phi = \text{Id}$ in $\Omega$; see \cite[Proposition 3.3]{Lionheart_1997}.

Let us now outline the key ideas in the proof of our main result, Theorem \ref{thm_main}. Similar to \cite{FGKU_2021}, we aim to recover the smooth Riemannian metric by reducing the inverse problem to one for a wave equation, allowing the metric to be determined via the boundary control method \cite{Belishev_1987, Belishev_Kurylev_1992, HLOS_2018}. However, in the current unbounded setting with exterior measurements, this reduction presents significant challenges that do not arise in the case of closed Riemannian manifolds with source--to--solution data. Specifically, working with exterior measurements \eqref{eq_int_3} and following the approach of \cite{FGKU_2021} only enables us to establish the equality of the heat semigroups $e^{t\Delta_{g_j}}$ for the Laplacians $-\Delta_{g_j}$, acting on energy solutions $u_j^f$ to the exterior Dirichlet problem \eqref{eq_int_1}, with $g = g_j$ for $j = 1, 2$, in the exterior of the domain $\Omega$:
\[
(e^{t\Delta_{g_1}}u_1^f)(x) = (e^{t\Delta_{g_2}}u_2^f)(x), \quad t > 0, \quad x \in \Omega_e,
\]
for all $f \in C^\infty_0(W_1)$. To complete the reduction to an inverse problem for the wave equation, we must recover the heat kernels:
\[
e^{t\Delta_{g_1}}(x,y) = e^{t\Delta_{g_2}}(x,y),
\]
for all times $t > 0$ and for all $x, y \in \Omega_e$. The novel and significant challenge, in contrast to \cite{FGKU_2021}, is thus the recovery of these heat kernels in our current unbounded setting with exterior measurements.

To overcome this challenge, we develop an approach to pass from the knowledge of the partial exterior Dirichlet-to-Neumann maps \eqref{eq_int_3} to the knowledge of the corresponding exterior source-to-solution maps in the present unbounded setting, under the assumption \eqref{eq_int_2} that the metrics agree in the exterior, and then recover the heat kernels from this data. Let us proceed to define the exterior source--to--solution maps. To that end, let $F \in C_0^\infty(\Omega_e)$, and consider the Poisson equation 
\begin{equation}
\label{eq_801_1_introduction}
(-\Delta_{g_j})^\alpha w_j = (-\Delta_{g_j}) F \quad \text{in} \quad \mathcal{D}'(\mathbb{R}^n).
\end{equation}
Associated with \eqref{eq_801_1_introduction}, we define the exterior source--to--solution map 
\[
L_{g_j}^{\Omega_e,\Omega_e}: C_0^\infty(\Omega_e) \to H^\alpha(\Omega_e), \quad F \mapsto w_j^F|_{\Omega_e},
\]
where $w_j = w_j^F \in H^\alpha(\mathbb{R}^n)$ is the unique solution to \eqref{eq_801_1_introduction}, $j=1,2$, see Lemma \ref{lem_solution_poisson_laplacian} below. The key step of our proof is to establish that under the assumption \eqref{eq_int_2}, the following holds, 
\begin{equation}
\label{eq_int_DN-source}
\Lambda_{g_1}^{W_1, W_2}= \Lambda_{g_2}^{W_1, W_2}\Longrightarrow L_{g_1}^{\Omega_e,\Omega_e}=L_{g_2}^{\Omega_e,\Omega_e}. 
\end{equation}
To prove \eqref{eq_int_DN-source}, we must address various technical difficulties not present in the case of closed Riemannian manifolds with source-to-solution data. In particular, we need to utilize delicate regularity properties of solutions to the exterior Dirichlet problem \eqref{eq_int_1} for the fractional Laplace-Beltrami operator, known as Vishik--Eskin estimates \cite{Hormander_1965-66, Grubb_2015, Eskin_book_1981, Vishik_Eskin_1965}, to rule out the concentration of solutions on the boundary $\partial \Omega$.

To derive the desired Vishik--Eskin estimates for the fractional Laplace-Beltrami operator, we employ two approaches. The first approach leverages the fact that $(-\Delta_g)^\alpha \in \Psi^{2\alpha}_{1,0}(\mathbb{R}^n)$ is a classical elliptic pseudodifferential operator on $\mathbb{R}^n$ (see Theorem \ref{thm_pseudodiff_op} below). A complete proof of Theorem \ref{thm_pseudodiff_op} is provided in Appendix~\ref{sec_pseudodiff_op}, and this proof is instrumental in deriving the Vishik--Eskin estimates from the theory developed in \cite{Hormander_1965-66, Grubb_2015, Eskin_book_1981, Vishik_Eskin_1965}. This proof may also be of independent interest. The second approach offers an elliptic extension perspective on the Vishik--Eskin estimates. Specifically, we establish the regularity of solutions to the mixed Dirichlet--Neumann problem for the variable-coefficient fractional Laplacian that arises from the extension. This approach relies on the method of difference quotients developed in \cite{KM07}; see also \cite{Savare} and \cite{KM07} for the case $\alpha=1/2$.

\begin{rem}
\label{prop_equiv}
Although in the proof of Theorem \ref{thm_main} it suffices to establish the implication \eqref{eq_int_DN-source}, we expect that, similar to the local case of the classical Calder\'on problem, the partial exterior Dirichlet--to--Neumann data and exterior source--to--solution data are equivalent. Specifically, in the setting of Theorem \ref{thm_main}, we anticipate that
\[
\Lambda_{g_1}^{W_1, W_2} = \Lambda_{g_2}^{W_1, W_2} \Leftrightarrow L^{\Omega_e, \Omega_e}_{g_1} = L^{\Omega_e, \Omega_e}_{g_2}. 
\]
 A back-of-the-envelope argument supports this expectation, yet a rigorous justification would require detailed considerations of well-posedness and mapping properties for 
 the whole-space problem. While we believe these properties hold, we postpone a thorough examination to future work due to the associated technical challenges and the length of this article.
\end{rem}

\begin{rem}[On a PDE version of Theorem \ref{thm_main}]
\label{rmk:PDE_op}
The classical anisotropic Calder\'on problem in dimension $n \ge 3$ has two equivalent formulations (see \cite{Uhlmann_2014}). The first, a geometric version discussed above, involves recovering the Riemannian metric $g$ for the Laplacian $-\Delta_g$ on a compact Riemannian manifold $(M, g)$ with boundary, based on the Dirichlet--to--Neumann map on the boundary of $M$. The second is a PDE formulation, which concerns the recovery of the smooth,  uniformly elliptic, symmetric conductivity matrix $a$ in the operator $\nabla \cdot a \nabla$ from the associated Dirichlet--to--Neumann map. In parallel, also in the nonlocal context two natural operators arise in the anisotropic inverse problem on $\mathbb{R}^n$: one is the fractional Laplace-Beltrami operator $(-\Delta_g)^{\alpha}$, $\alpha \in (0,1)$, introduced above; the other is the operator $(-\nabla \cdot a \nabla)^{\alpha}$, $\alpha \in (0,1)$, where the matrix $a$ is assumed to coincide with the identity matrix outside a compact set. The latter operator is defined by functional calculus for self-adjoint operators and is an unbounded self-adjoint operator on $L^2(\mathbb{R}^n; dx)$ with domain $H^{2\alpha}(\mathbb{R}^n)$; see \cite[Theorem 1.5]{KPPV_2024}.

Using an approach analogous to Section~\ref{sec_degenerate_elliptic_extension_approach} below, one could obtain a PDE version of Theorem \ref{thm_main} for the operators $(-\nabla \cdot a \nabla)^{\alpha}$. Specifically, one could show that knowledge of the partial exterior Dirichlet-to-Neumann map, defined as in \eqref{eq_int_1_exterior_DN} with $(-\Delta_g)^{\alpha}$ replaced by $(-\nabla \cdot a \nabla)^{\alpha}$, along with knowledge of the matrix $a$ in the exterior, determines the matrix $a$ in $\mathbb{R}^n$ up to a natural gauge.
\end{rem}

\subsection{Previous literature} The study of the fractional Calder\'on problem was first introduced in \cite{Ghosh_Salo_Uhlmann_2020}, where the unknown potential in the fractional Schr\"odinger equation on a bounded domain in Euclidean space was recovered from exterior measurements. Building on this foundational work, inverse problems involving the recovery of lower-order terms, as well as isotropic leading-order coefficients in fractional elliptic equations, have been explored extensively. We refer to 
\cite{Ghosh_Ruland_Salo_Uhlmann_2020, Ghosh_Lin_Xiao_2017, Ruland_Salo_2018, Ruland_Salo_2020, Ruland_2021, Li_Li_2021, Li_Li_2020, Bhattacharyya_Ghosh_Uhlmann_2021, Cekic_Lin_Ruland_2020, Covi_2020, Covi_2020_2, Covi_quasilocal, Covi_Garcia-Ferrero_Ruland_preprint, Covi_M_R_Uhlmann_preprint, Baers_Covi_Ruland_2024, Covi_de_Hoop_Salo_2024, Covi_Railo_Tyni_Zimmermann_2024, Harrach_Lin_Weth_2024, Lai_Zhou_2023, Lin_Zimmermann_2024, Railo_Zimmermann_2023, Railo_Zimmermann_2024}  and the references therein for a non-exhaustive list of contributions.

Recently, some progress has been made in the study of inverse problems for fractional elliptic operators with unknown anisotropic leading-order coefficients. Two distinct approaches have emerged in the literature. The first, introduced in the works \cite{Ghosh_Uhlmann_2021, Covi_Ghosh_Ruland_Uhlmann_2023}, involves reducing inverse problems for fractional (nonlocal) operators to their corresponding local counterparts, thereby enabling the transfer of uniqueness results from the local to the nonlocal setting. A similar approach is employed in \cite{Lin_Lin_Uhlmann_2023} to study inverse problems for fractional parabolic operators.

A second approach, developed in \cite{Feizmohammadi_2021, FGKU_2021}, achieves the complete recovery of a smooth Riemannian metric on a smooth closed connected Riemannian manifold using local source--to--solution data. This method has been extended to noncompact complete Riemannian manifolds with the local source--to--solution data in \cite{Choulli_Ouhabaz_2023}, and \cite{Ruland_2023} offers a novel perspective on the approach in \cite{Feizmohammadi_2021, FGKU_2021} by utilizing a variable coefficient Caffarelli--Silvestre-type extension for nonlocal operators, building on the interpretations of the fractional Laplacian from \cite{CS07} for constant and from \cite{ST10} for variable coefficients. Further generalizations to other fractional anisotropic geometric operators are explored in \cite{Chien_2023, Quan_Uhlmann}, see also 
\cite{Feizmohammadi_Lin_2024}. Additionally, in \cite{FKU_2024}, both the smooth Riemannian manifold, up to isometry, and the potential, up to a corresponding gauge transformation, are recovered on a closed Riemannian manifold under certain geometric conditions on the manifold and observation set.

The problem of recovering a smooth Riemannian metric in the fractional anisotropic Calder\'on problem with exterior data in noncompact settings, as introduced in \cite{Ghosh_Salo_Uhlmann_2020}, has remained a notable open problem. In this paper, we present a complete solution to this problem. We propose two closely related approaches: the first uses the heat semigroup representation of the fractional Laplacian combined with a pseudodifferential framework, while the second is based on a variable-coefficient Caffarelli--Silvestre type extension interpretation of the fractional Laplacian (see the seminal works \cite{CS07, ST10}).

\subsection{Organization of the paper} The remainder of the paper is organized as follows. In Section~\ref{sec_well_posedness_dirichlet}, we collect essential facts and results concerning the fractional Laplace-Beltrami operator and the exterior Dirichlet problem, which form the foundation for proving Theorem~\ref{thm_main} in Section~\ref{sec_proof_using_heat_semigroup_approach}. Section~\ref{sec_proof_using_heat_semigroup_approach} presents a proof of Theorem~\ref{thm_main}, utilizing the heat semigroup representation of the fractional Laplacian in conjunction with a pseudodifferential framework. In Section~\ref{sec_degenerate_elliptic_extension_approach}, we offer an equivalent perspective on the formulation and proof of Theorem~\ref{thm_main} via a variable-coefficient Caffarelli--Silvestre type extension interpretation of the fractional Laplacian. Appendix~\ref{app_obstruction} discusses an obstruction to uniqueness in the anisotropic fractional Calder\'on problem with external data. Appendix~\ref{sec_pseudodiff_op} contains the proof of Theorem~\ref{thm_pseudodiff_op}, demonstrating that $(-\Delta_g)^\alpha \in \Psi^{2\alpha}_{1,0}(\mathbb{R}^n)$ is a classical elliptic pseudodifferential operator on $\mathbb{R}^n$, which is crucial for the proof of Theorem~\ref{thm_main} in Section~\ref{sec_proof_using_heat_semigroup_approach}. Finally, Appendix~\ref{sec:reg} provides an extension perspective on Vishik–Eskin type estimates, which are essential for the proof of Theorem~\ref{thm_main} in Section~\ref{sec_degenerate_elliptic_extension_approach}.

\section{The fractional Laplace--Beltrami operator and the exterior Dirichlet problem}

\label{sec_well_posedness_dirichlet}

In this section, we gather some facts and results about the fractional Laplace--Beltrami operator and its associated exterior Dirichlet problem, which are needed to prove Theorem \ref{thm_main} using the heat semigroup approach in Section~\ref{sec_proof_using_heat_semigroup_approach}.

\subsection{The fractional Laplace--Beltrami operator on the Euclidean space as a pseudodifferential operator}

\label{sec_pseudodiff_op_1}

Let $m \in \mathbb{R}$, and let us introduce the following standard symbol space: $a \in S^m_{1,0}(\mathbb{R}^n \times \mathbb{R}^n \times \mathbb{R}^n)$ if and only if $a \in C^\infty(\mathbb{R}^n \times \mathbb{R}^n \times \mathbb{R}^n)$, and for all compact subsets $K \Subset  \mathbb{R}^n \times \mathbb{R}^n$ and all multi-indices $\beta, \gamma, \delta \in \mathbb{N}^n$, there exists a constant $C = C_{K, \beta, \gamma, \delta} > 0$ such that 
\[
|\partial_x^\beta \partial_y^\gamma \partial_\theta^\delta a(x, y, \theta)| \leq C(1 + |\theta|)^{m - |\delta|}, \quad (x, y) \in K, \quad \theta \in \mathbb{R}^n.
\] 
Associated with the symbol class $S^m_{1,0}(\mathbb{R}^n \times \mathbb{R}^n \times \mathbb{R}^n)$ is the space $\Psi^m_{1,0}(\mathbb{R}^n)$ of pseudodifferential operators on $\mathbb{R}^n$ of the form
\[
A u(x) = \frac{1}{(2\pi)^n} \int_{\mathbb{R}^n}\!\!\int_{\mathbb{R}^n} e^{i(x-y)\cdot \theta} a(x, y, \theta) u(y) \, dy \, d\theta, \quad u \in C^\infty_0(\mathbb{R}^n),
\] 
with $a \in S^m_{1,0}(\mathbb{R}^n \times \mathbb{R}^n \times \mathbb{R}^n)$; see \cite[Section 3]{Grigis_Sjostrand_book}.

We have the following result; see also \cite[Theorem~4.7]{Strichartz_1983}.
\begin{thm}
\label{thm_pseudodiff_op}
Let $g$ be a $C^\infty$ Riemannian metric on $\mathbb{R}^n$, $n \geq 2$, which agrees with the Euclidean metric outside of a compact set, and let $0 < \alpha < 1$. Then the operator $(-\Delta_g)^\alpha$, defined via functional calculus, satisfies $(-\Delta_g)^\alpha \in \Psi^{2\alpha}_{1,0}(\mathbb{R}^n)$ and is a classical elliptic pseudodifferential operator on $\mathbb{R}^n$, with principal symbol 
\[
\bigg(\sum_{j,k=1}^n g^{jk}(x)\xi_j\xi_k\bigg)^\alpha, \quad (x,\xi) \in T^*\mathbb{R}^n \setminus 0.
\]
\end{thm}

We refer to Appendix \ref{sec_pseudodiff_op} for the proof of Theorem \ref{thm_pseudodiff_op}. Let $\beta > 0$ and express $\beta$ as $\beta = m + \alpha$, where $m \in \mathbb{N}$ is the integer part of $\beta$, and $0 \leq \alpha < 1$. The representation $(-\Delta_g)^\beta = (-\Delta_g)^{m}(-\Delta_g)^\alpha$ leads to the following consequence of Theorem \ref{thm_pseudodiff_op}. 
\begin{cor}
Let $g$ be a $C^\infty$ Riemannian metric on $\mathbb{R}^n$, $n \geq 2$, which agrees with the Euclidean metric outside of a compact set, and let $\beta > 0$. Then the operator $(-\Delta_g)^\beta$, defined via functional calculus, satisfies $(-\Delta_g)^\beta \in \Psi^{2\beta}_{1,0}(\mathbb{R}^n)$ and is a classical elliptic pseudodifferential operator on $\mathbb{R}^n$, with principal symbol 
\[
\left(\sum_{j,k=1}^n g^{jk}(x)\xi_j\xi_k\right)^\beta, \quad (x,\xi) \in T^*\mathbb{R}^n \setminus 0.
\] 
\end{cor}

\subsection{Some facts about Sobolev spaces}
\label{sec_Sobolev_spaces_direct_problem}

Let $g$ be a $C^\infty$ Riemannian metric on $\mathbb{R}^n$, $n\ge 2$, that agrees with the Euclidean metric outside of a compact set. 
Let $H^s(\mathbb{R}^n)$, $s \in \mathbb{R}$, be the standard $L^2$-based Sobolev space on $\mathbb{R}^n$. We have the following alternative characterization of the Sobolev space $H^s(\mathbb{R}^n)$:
\[
H^s(\mathbb{R}^n) = H^s(\mathbb{R}^n; dV_g) = \{u \in \mathcal{S}'(\mathbb{R}^n) : (1 - \Delta_g)^{s/2} u \in L^2(\mathbb{R}^n; dV_g)\},
\]
with the norms
\[
\|u\|_{H^s(\mathbb{R}^n)} = \|(1 - \Delta)^{s/2} u\|_{L^2(\mathbb{R}^n)} \quad \text{and} \quad \|u\|_{H^s(\mathbb{R}^n; dV_g)} = \|(1 - \Delta_g)^{s/2} u\|_{L^2(\mathbb{R}^n; dV_g)},
\]
being equivalent. Here, the Bessel potential $(1 - \Delta_g)^{s/2}$ is defined by the self-adjoint functional calculus, and the equivalence of the norms follows from the fact that $(1 - \Delta)^{s/2}, \, (1 - \Delta_g)^{s/2} \in \Psi^s_{1,0}(\mathbb{R}^n)$ (with global estimates in $x$), see \cite[Section 3]{Grubb_2020}, \cite[Theorem 4.7]{Strichartz_1983}, and that pseudodifferential operators of class $\Psi^0_{1,0}(\mathbb{R}^n)$ are bounded on $L^2(\mathbb{R}^n)$. 

For future reference, we state the following result; see \cite[Theorem 4.4]{Strichartz_1983}. 
\begin{prop}
\label{prop_regularity}
Let $s > 0$. Then $u \in H^s(\mathbb{R}^n)$ if and only if $u \in \mathcal{D}((-\Delta_g)^{s/2})$, and we have the equivalence of the norms:
\begin{equation}
\label{eq_Strichartz_domain}
\|u\|_{H^s(\mathbb{R}^n)} \asymp \|u\|_{L^2(\mathbb{R}^n; dV_g)} + \|(-\Delta_g)^{s/2} u\|_{L^2(\mathbb{R}^n; dV_g)}.
\end{equation}
\end{prop}

For the duality between $H^{-s}(\R^n)$ and $H^s(\R^n)$, we use the notation
\[
\langle u, \overline{v} \rangle_{H^{-s}(\R^n), H^s(\R^n)}, \quad u \in H^{-s}(\R^n), \ v \in H^s(\R^n),
\]
and we note that it coincides with the scalar product in $L^2(\R^n)$ and with the distribution duality where these are defined;  see \cite[page 133]{Grubb_book}.

Let $\Omega \subset \mathbb{R}^n$ be an open set (not necessarily bounded). We define the Sobolev space on $\Omega$ as
\[
H^s(\Omega) = \{u|_{\Omega} : u \in H^s(\mathbb{R}^n)\} \subset \mathcal{D}'(\Omega), \quad s \in \mathbb{R},
\]
and equip it with the quotient norm
\[
\|u\|_{H^s(\Omega)} = \inf \{\|w\|_{H^s(\mathbb{R}^n)} : w \in H^s(\mathbb{R}^n), \, w|_{\Omega} = u\}.
\]

From now on, assume that $\Omega \subset \mathbb{R}^n$ is an open bounded set with a $C^\infty$ boundary. We also introduce the Sobolev space
\begin{equation}
\label{eq_duality_on_domains_new_0}
H^s_0(\Omega) = \{u \in H^s(\mathbb{R}^n) : \supp(u) \subset \overline{\Omega}\}, \quad s \in \mathbb{R},
\end{equation}
and recall from \cite[Theorem 22.4]{Eskin_book} that the space $C_0^\infty(\Omega)$ is dense in $H^s_0(\Omega)$ with respect to the $H^s(\mathbb{R}^n)$ topology for all $s \in \mathbb{R}$.

We also have the following natural duality statements:
\begin{equation}
\label{eq_duality_on_domains_new_1}
(H_0^s(\Omega))^* = H^{-s}(\Omega), \quad (H^s(\Omega))^* = H_0^{-s}(\Omega), \quad s \in \mathbb{R},
\end{equation}
see \cite[Theorem 22.8]{Eskin_book}. The duality pairing is defined as follows: for $v \in H^{s}(\Omega)$ and $w \in H^{-s}_0(\Omega)$, we set
\begin{equation}
\label{eq_duality_on_domains_new}
\langle v, \overline{w}\rangle_{ H^{s}(\Omega), H^{-s}_0(\Omega)} = \langle \text{Ext}(v), \overline{w}\rangle_{H^{s}(\mathbb{R}^n), H^{-s}(\mathbb{R}^n)},
\end{equation}
where $\text{Ext}(v) \in H^s(\mathbb{R}^n)$ is an arbitrary extension of $v$. Note that the definition in \eqref{eq_duality_on_domains_new} is independent of the choice of the extension $\text{Ext}(v)$; see \cite[Lemma 22.7]{Eskin_book}. 

We also note that \eqref{eq_duality_on_domains_new_0}, \eqref{eq_duality_on_domains_new_1}, and \eqref{eq_duality_on_domains_new} remain valid when $\Omega$ is replaced by $\Omega_e = \mathbb{R}^n \setminus \overline{\Omega}$; see \cite[Section 2A]{Ghosh_Salo_Uhlmann_2020} and \cite[Chapter 3]{McLean_book_2000}. Furthermore, the space $C_0^\infty(\Omega_e)$ is dense in $H^s_0(\Omega_e)$ with respect to the $H^s(\mathbb{R}^n)$ topology for all $s \in \mathbb{R}$; see \cite[Section 2A]{Ghosh_Salo_Uhlmann_2020} and \cite[Chapter 3]{McLean_book_2000}.

For future reference, we state the following result; see \cite[Theorems 5.1.10, 5.1.11, 5.1.13, 5.1.14]{Agranovich_book_2015}.
\begin{thm}
\label{thm_Agranovich}
Let $s \in \mathbb{R}$. We have:
\begin{itemize}
\item[(i)] For $|s| < 1/2$, the operator of extension by zero is bounded from $H^s(\Omega)$ to $H^s(\mathbb{R}^n)$.
\item[(ii)] For $|s| < 1/2$, $H^s(\Omega) = H^s_0(\Omega)$ with the natural identification.
\item[(iii)] For $-1/2\leq s $, the subspace of $H^s(\mathbb{R}^n)$ consisting of elements supported on $\partial \Omega$ is trivial.
\item[(iv)] For $s \leq 1/2$, the space $C_0^\infty(\Omega)$ is dense in $H^s(\Omega)$.
\end{itemize}
\end{thm}

Let $|s| < \frac{1}{2}$. Then for $u \in H^s(\mathbb{R}^n)$ and $v \in H^{-s}(\mathbb{R}^n)$, we have
\begin{equation}
\label{eq_Sobolev_duality_mikko}
\langle u, v \rangle_{\mathbb{R}^n} = \langle r_\Omega u, r_\Omega v \rangle_{\Omega} + \langle r_{\Omega_e} u, r_{\Omega_e} v \rangle_{\Omega_e},
\end{equation}
see \cite[(A-1), page 469]{Ghosh_Salo_Uhlmann_2020}. Here, $r_\Omega u = u|_{\Omega}$ and $r_{\Omega_e} u = u|_{\Omega_e}$ denote the restrictions of $u$ to $\Omega$ and $\Omega_e$, respectively, and $\langle \cdot, \cdot \rangle$ denotes the distributional duality. The equality \eqref{eq_Sobolev_duality_mikko} follows from the fact that the characteristic functions $\chi_\Omega$ and $1 - \chi_\Omega$ are pointwise multipliers in $H^s(\mathbb{R}^n)$.

When $s \in (0,1)$, the following alternative characterization of $H^s(\Omega)$ holds,
\begin{equation}
\label{eq_sobolev_new_1}
H^s(\Omega) = \bigg\{ u \in L^2(\Omega) : \frac{|u(x) - u(y)|}{|x - y|^{\frac{n}{2} + s}} \in L^2(\Omega \times \Omega) \bigg\},
\end{equation}
with the norms
\begin{equation}
\label{eq_2_Sobolev_frac}
\|u\|_{H^s(\Omega)} \asymp \bigg( \|u\|^2_{L^2(\Omega)} + \int_{\Omega} \int_{\Omega} \frac{|u(x) - u(y)|^2}{|x - y|^{n + 2s}} \, dx \, dy \bigg)^{\frac{1}{2}}
\end{equation}
being equivalent, see \cite[page 314]{Brezis_2011}. Note that the same characterization as in \eqref{eq_sobolev_new_1} is valid for the Sobolev space $H^s(\mathbb{R}^n)$, with the equivalent norms as in \eqref{eq_2_Sobolev_frac}, replacing $\Omega$ with $\mathbb{R}^n$.

For future reference, we state the following fractional Sobolev inequality (see \cite[Theorem 6.5]{Nezza_Palatucci_Valdinoci_2012}, \cite[Section 3]{Brasco_Gomez-Castro_Vazquez_2021}, and \cite[Section 10.2]{Mazya_book_2012}).
\begin{prop}
\label{prop_fractional_Sobolev_inequality}
Let $n \geq 1$, $s \in (0,1)$, and $p \in [1, +\infty)$ be such that $sp < n$. Then there exists a constant $C = C(n, p, s) > 0$ such that for all $u \in C^\infty_0(\mathbb{R}^n)$, we have
\[
\|u\|_{L^{p^*}(\mathbb{R}^n)}^p \leq C \int_{\mathbb{R}^n} \int_{\mathbb{R}^n} \frac{|u(x) - u(y)|^p}{|x - y|^{n + sp}} \, dx \, dy.
\]
Here, $p^* = \frac{np}{n - sp} \geq p$ is the fractional critical exponent.
\end{prop}

\subsection{The fractional Laplacian and the heat semigroup}

Let $g$ be a $C^\infty$ Riemannian metric on $\mathbb{R}^n$, $n \geq 2$, which agrees with the Euclidean metric outside of a compact set. Consider the heat semigroup 
\[ e^{t\Delta_g}: L^2(\mathbb{R}^n; dV_g) \to L^2(\mathbb{R}^n; dV_g), \quad t \ge 0, \]
associated with the Laplacian $-\Delta_g$. The heat kernel $e^{t\Delta_g}(x, y)$ satisfies:
\begin{equation}
\label{eq_2_4_heat_sym}
e^{t\Delta_g}(x, y) = e^{t\Delta_g}(y, x) \in C^\infty((0, \infty) \times \mathbb{R}^n \times \mathbb{R}^n),
\end{equation}
and
\begin{equation}
\label{eq_2_4}
(e^{t\Delta_g} u)(x) = \int_{\mathbb{R}^n} e^{t\Delta_g}(x, y) u(y) \, dV_g(y)
\end{equation}
for $u \in L^2(\mathbb{R}^n)$ and for all $x \in \mathbb{R}^n$ and $t > 0$; see \cite[Theorem 3.5]{Strichartz_1983}.
Additionally, we have:
\begin{equation}
\label{eq_2_4_L_2_boundedness}
\|e^{t\Delta_g} u\|_{L^2(\mathbb{R}^n; dV_g)} \le \|u\|_{L^2(\mathbb{R}^n; dV_g)},
\end{equation}
see \cite[Theorem 3.5]{Strichartz_1983}.

Let $\alpha \in (0,1)$. In what follows below, it will be useful to express the self-adjoint operator $(-\Delta_g)^\alpha$ in terms of the heat semigroup $e^{t\Delta_g}$, $t \geq 0$. To that end, combining the identity
\[
a^{\alpha} = \frac{1}{\Gamma(-\alpha)} \int_0^\infty (e^{-at} - 1) \frac{dt}{t^{1+\alpha}},
\]
which is valid for all $a \geq 0$, together with the functional calculus, we get
\begin{equation}
\label{eq_2_6}
(-\Delta_g)^\alpha u = \frac{1}{\Gamma(-\alpha)} \int_0^\infty (e^{t\Delta_g} - 1) u \frac{dt}{t^{1+\alpha}},
\end{equation}
for $u \in \mathcal{D}(-\Delta_g) = H^2(\mathbb{R}^n)$; see \cite{ST10}, and also  \cite{Caffarelli_Stinga_2016}. The integral on the right-hand side of \eqref{eq_2_6} converges in $L^2(\mathbb{R}^n)$.

For future reference, we state the following pointwise upper and lower Gaussian estimates on the heat kernel $e^{t\Delta_g}(x, y)$, established in \cite{Aronson_1967} (see also \cite[Theorem 3.4 and Theorem 3.6]{Porper_Eidelman_1984}):
\begin{equation}
\label{eq_2_5_1}
C_1 t^{-\frac{n}{2}} e^{-c_1 \frac{|x - y|^2}{t}} \leq e^{t\Delta_g}(x, y) \leq C_2 t^{-\frac{n}{2}} e^{-c_2 \frac{|x - y|^2}{t}}, \quad t > 0, \quad x, y \in \mathbb{R}^n,
\end{equation}
where $C_1, C_2, c_1, c_2 > 0$. We also record the following pointwise upper Gaussian estimates for the derivatives of the heat kernel:
\begin{equation}
\label{eq_2_5_1_derivatives}
| \partial_t^m e^{t\Delta_g}(x, y) | \leq C t^{-\frac{n}{2} - m} e^{-c \frac{|x - y|^2}{t}}, \quad t > 0, \quad x, y \in \mathbb{R}^n, \quad m = 1, 2, \dots,
\end{equation}
where $C = C(m) > 0$ and $c > 0$ is independent of $m$, see \cite[Theorem 4.1]{Porper_Eidelman_1984}.

We note that the stochastic completeness property holds for the heat kernel:
\begin{equation}
\label{eq_2_5}
\int_{\mathbb{R}^n} e^{t\Delta_g}(x, y) \, dV_g(y) = 1,
\end{equation}
for all $t > 0$ and $x \in \mathbb{R}^n$; see \cite[Chapter VIII, Theorem 5, p. 191]{Chavel_book_1984} and \cite[Theorem 5.2.6]{Davies_book}. Using \eqref{eq_2_4} and \eqref{eq_2_5}, we have
\begin{equation}
\label{eq_2_5_L_infty}
\|e^{t\Delta_g}\|_{L^\infty(\R^n)\to L^\infty(\R^n)}\le 1. 
\end{equation}

In view of \eqref{eq_2_6}, letting $u, v \in C^\infty_0(\mathbb{R}^n)$, we consider the sesquilinear forms:
\begin{equation}
\label{eq_2_6_seq_forms_new}
((-\Delta_g)^\alpha u, v)_{L^2(\mathbb{R}^n; dV_g)} = \mathcal{E}(u, v),
\end{equation}
where
\begin{equation}
\label{eq_2_6_A}
\mathcal{E}(u, v) := (A u, v)_{L^2(\mathbb{R}^n; dV_g)}, \quad 
A := \frac{1}{\Gamma(-\alpha)} \int_0^\infty (e^{t\Delta_g} - 1) \frac{dt}{t^{1+\alpha}}.
\end{equation}
Here, we recall that the Riemannian volume form $dV_g$ is given by $dV_g = \sqrt{|g|}dx$.  

\begin{prop}
	\label{prop_K_bounds}
We have
\begin{equation}
\label{eq_2_11}
\mathcal{E}(u, v) = \int_{\mathbb{R}^n} \int_{\mathbb{R}^n} K_g(z, x) (u(x) - u(z)) \overline{(v(x) - v(z))} \, dx \, dz,
\end{equation}
for $u, v \in C^\infty_0(\mathbb{R}^n)$. Here,
\begin{equation}
\label{eq_2_12}
K_g(z, x) = \frac{1}{2|\Gamma(-\alpha)|} \sqrt{|g(x)|} \sqrt{|g(z)|} \int_0^\infty e^{t\Delta_g}(z, x) \frac{dt}{t^{1+\alpha}}.
\end{equation}
Furthermore, the kernel $K_g: \mathbb{R}^n \times \mathbb{R}^n \to \mathbb{R}$ is measurable and satisfies the estimates,
\begin{equation}
\label{eq_2_13}
\frac{1}{C |z - x|^{n + 2\alpha}} \leq K_g(z, x) \leq \frac{C}{|z - x|^{n + 2\alpha}}, \quad x, z \in \mathbb{R}^n, \quad x \neq z,
\end{equation}
where $C > 1$.
\end{prop}

\begin{proof}
Letting $u, v \in C^\infty_0(\mathbb{R}^n)$, we shall compute the expression for $\mathcal{E}(u, v)$ following \cite[Theorem 2.3]{Caffarelli_Stinga_2016}. To that end, we first observe that
\begin{equation}
\label{eq_2_7_new_1}
\mathcal{E}(u,v) = \frac{1}{\Gamma(-\alpha)} \int_0^\infty (e^{t\Delta_g}u - u, v)_{L^2(\mathbb{R}^n; dV_g)} \frac{dt}{t^{1+\alpha}}.
\end{equation}
Indeed, this follows by Fubini's theorem owing to the fact that the function
\begin{equation}
\label{eq_100_1}
(t,z) \mapsto (e^{t\Delta_g} - 1)u(z) \overline{v(z)} \sqrt{|g(z)|} t^{-(1+\alpha)}
\end{equation}
is in $L^1((0,\infty) \times \mathbb{R}^n)$. To see this, we first note that \eqref{eq_2_5_L_infty} implies that
\begin{equation}
\label{eq_100_2}
|(e^{t\Delta_g} - 1)u(z)| \le 2 \|u\|_{L^\infty(\mathbb{R}^n)},
\end{equation}
for all $t>0$ and $z\in \R^n$. 
We also have $e^{t\Delta_g} u \in C^\infty((0,\infty) \times \mathbb{R}^n)$; see \cite[Theorem 7.4]{Grigoryan_book_2009}, and $e^{t\Delta_g} \Delta_g = \Delta_g e^{t\Delta_g}$ for $t \ge 0$ on $\mathcal{D}(-\Delta_g)$. By the fundamental theorem of calculus, we get
\[
e^{t\Delta_g} u(z) - u(z) = \int_0^t e^{s\Delta_g} \Delta_g u(z) \, ds,
\]
and therefore,
\begin{equation}
\label{eq_100_3}
|e^{t\Delta_g} u(z) - u(z)| \le t \|\Delta_g u\|_{L^\infty(\mathbb{R}^n)},
\end{equation}
for all $t>0$ and $z\in \R^n$. 
Using Tonelli’s theorem, together with \eqref{eq_100_2} for $t \ge 1$ and \eqref{eq_100_3} for $0 < t < 1$, we obtain the claim \eqref{eq_100_1}.

Now, using \eqref{eq_2_4} and \eqref{eq_2_5}, we get from \eqref{eq_2_7_new_1} that 
\begin{equation}
\label{eq_2_7}
\begin{aligned}
\mathcal{E}(u,v)&=\frac{1}{\Gamma(-\alpha)}\int_0^\infty \int_{\R^n} \big((e^{t\Delta_g}u)(z)-u(z)\big)\overline{v(z)}dV_g(z) \frac{dt}{t^{1+\alpha}}\\
&=\frac{1}{\Gamma(-\alpha)}\int_0^\infty \int_{\R^n} \bigg[ \int_{\R^n} e^{t\Delta_g}(z,x)u(x)dV_g(x) -u(z) \bigg]\overline{v(z)}dV_g(z) \frac{dt}{t^{1+\alpha}}\\
&=\frac{1}{\Gamma(-\alpha)}\int_0^\infty \int_{\R^n}\int_{\R^n} e^{t\Delta_g}(z,x)(u(x)-u(z))\overline{v(z)}dV_g(x)dV_g(z)\frac{dt}{t^{1+\alpha}}.
\end{aligned}
\end{equation}

As in \cite[Theorem 2.3]{Caffarelli_Stinga_2016}, by exchanging $x$ and $z$ in the last equality in \eqref{eq_2_7}, and using \eqref{eq_2_4_heat_sym} together with Fubini's theorem in $x$ and $z$, we get 
\begin{equation}
\label{eq_2_10}
\begin{aligned}
\mathcal{E}(u,v)=-\frac{1}{\Gamma(-\alpha)}\int_0^\infty  \int_{\R^n}\int_{\R^n}e^{t\Delta_g}(z,x)(u(x)-u(z))\overline{v(x)}dV_g(x)dV_g(z)\frac{dt}{t^{1+\alpha}}.
\end{aligned}
\end{equation}
Here, the application of Fubini's theorem is justified thanks to \eqref{eq_2_5}.

Adding the last equality in \eqref{eq_2_7} and \eqref{eq_2_10}, and using the fact that $\Gamma(-\alpha)<0$, we get
\begin{equation}
\label{eq_2_11_new_1}
\begin{aligned}
\mathcal{E}(u,v)&=\frac{1}{2|\Gamma(-\alpha)|}\\
& \int_0^\infty\int_{\mathbb{R}^n}\int_{\mathbb{R}^n} e^{t\Delta_g}(z,x)(u(x)-u(z))\overline{(v(x)-v(z))}dV_g(x)dV_g(z)\frac{dt}{t^{1+\alpha}}.
\end{aligned}
\end{equation}
Now we claim that the function
\[
(t, x, z) \mapsto e^{t\Delta_g}(z, x)(u(x) - u(z))\overline{(v(x) - v(z))} \frac{1}{t^{1+\alpha}}\sqrt{|g(x)|}\sqrt{|g(z)|}
\]
is in $L^1((0,\infty) \times \mathbb{R}^n \times \mathbb{R}^n)$. Indeed, it follows from \eqref{eq_2_5_1}, Tonelli's theorem, the Cauchy-Schwarz inequality, and \eqref{eq_2_Sobolev_frac} that
\begin{equation}
\label{eq_2_11_new_1_20}
\begin{aligned}
&\int_0^\infty\int_{\R^n}\int_{\R^n} \big|  e^{t\Delta_g}(z,x)(u(x)-u(z))\overline{(v(x)-v(z))} \big| dV_g(x)dV_g(z)\frac{dt}{t^{1+\alpha}}\\
&\le C\int_0^\infty\int_{\R^n}\int_{\R^n}  \frac{e^{-\frac{|x-z|^2}{ct}}}{t^{n/2+1+\alpha}}|u(x)-u(z)||v(x)-v(z)|dxdzdt\\
&\le C\int_{\R^n}\int_{\R^n} 
\frac{1}{|x-z|^{n+2\alpha}}|u(x)-u(z)||v(x)-v(z)|dxdz\\
&
\le C\|u\|_{H^\alpha(\R^n)}\|v\|_{H^\alpha(\R^n)} <\infty,
\end{aligned}
\end{equation}
showing the claim. An application of Fubini's theorem to \eqref{eq_2_11_new_1} shows that 
\[
\mathcal{E}(u,v)=\int_{\R^n}\int_{\R^n}
K_g(z,x)(u(x)-u(z))\overline{(v(x)-v(z))}d xd z,
\]
where 
\begin{equation}
\label{eq_2_12_inside}
K_g(z,x)= \frac{1}{2|\Gamma(-\alpha)|} \sqrt{|g(x)|}\sqrt{|g(z)|}\int_0^\infty e^{t\Delta_g}(z,x)\frac{dt}{t^{1+\alpha}}.
\end{equation}
Using \eqref{eq_2_5_1} and the fact that for $c > 0$,
\[
\int_0^\infty \frac{e^{-\frac{c}{t}}}{t^{\frac{n}{2}+1+\alpha}} \, dt = c^{-\left(\frac{n}{2} + \alpha\right)} \Gamma\left(\frac{n}{2} + \alpha\right) > 0,
\]
and that $\sqrt{|g(x)|} \ge c > 0$ for all $x \in \mathbb{R}^n$, we obtain from \eqref{eq_2_12_inside} that $K_g$ satisfies the estimates
\[
\frac{1}{C |z-x|^{n+2\alpha}} \le K_g(z,x) \le \frac{C}{|z-x|^{n+2\alpha}}, \quad x, z \in \mathbb{R}^n, \quad x \ne z,
\]
where $C > 0$.
\end{proof}

\begin{prop}
\label{prop_eq_2_6}
The operator $(-\Delta_g)^\alpha$ extends to a linear continuous map $(-\Delta_g)^\alpha: H^\alpha(\mathbb{R}^n) \to H^{-\alpha}(\mathbb{R}^n)$, and the operator $A$, given in \eqref{eq_2_6_A}, also extends to a linear continuous map $A: H^\alpha(\mathbb{R}^n) \to H^{-\alpha}(\mathbb{R}^n)$. The representation \eqref{eq_2_6} extends to $u \in H^\alpha(\mathbb{R}^n)$. Furthermore, the equality \eqref{eq_2_6_seq_forms_new} extends to $u, v \in H^\alpha(\mathbb{R}^n)$, and the equality 
\[
((-\Delta_g)^\alpha u, v)_{L^2(\mathbb{R}^n; dV_g)} = (u, (-\Delta_g)^\alpha v)_{L^2(\mathbb{R}^n; dV_g)}
\]
extends to $u, v \in H^\alpha(\mathbb{R}^n)$.
\end{prop}

\begin{proof}
It follows from \eqref{eq_2_11_new_1} and \eqref{eq_2_11_new_1_20} that 
for $u,v \in C^\infty_0(\R^n)$,
\begin{equation}
\label{eq_2_15}
|\mathcal{E}(u,v)| \leq C\|u\|_{H^\alpha(\R^n)}\|v\|_{H^{\alpha}(\R^n)}.
\end{equation}
Therefore, the form $(u,v) \mapsto \mathcal{E}(u,v)$ uniquely extends to a continuous sesquilinear form on $H^\alpha(\R^n) \times H^\alpha(\R^n)$. Consequently, the operator $A$ defined in \eqref{eq_2_6_A} extends as a continuous linear map:
\[
A: H^{\alpha}(\R^n) \to H^{-\alpha}(\R^n).
\]
Furthermore, for $u, v \in C^\infty_0(\R^n)$, by functional calculus, we have:
\begin{equation}
\label{eq_ext_domain_100_-100}
((-\Delta_g)^\alpha u,v)_{L^2(\R^n; dV_g)} = ((-\Delta_g)^{\alpha/2} u, (-\Delta_g)^{\alpha/2} v)_{L^2(\R^n; dV_g)}.
\end{equation}
By the Cauchy-Schwarz inequality and \eqref{eq_Strichartz_domain}, it follows that:
\begin{equation}
\label{eq_ext_domain_100}
\begin{aligned}
|((-\Delta_g)^\alpha u,v)_{L^2(\R^n; dV_g)}| &\leq \|(-\Delta_g)^{\alpha/2} u\|_{L^2(\R^n; dV_g)} \|(-\Delta_g)^{\alpha/2} v\|_{L^2(\R^n; dV_g)} \\
&\leq C \|u\|_{H^\alpha(\R^n)} \|v\|_{H^{\alpha}(\R^n)}.
\end{aligned}
\end{equation}
Hence, the form $(u,v) \mapsto ((-\Delta_g)^\alpha u,v)_{L^2(\R^n; dV_g)}$ extends to $H^\alpha(\R^n) \times H^\alpha(\R^n)$, giving us a continuous linear operator:
\[
(-\Delta_g)^\alpha: H^{\alpha}(\R^n) \to H^{-\alpha}(\R^n).
\]
The result follows.
\end{proof}

\subsection{Solvability of the exterior Dirichlet problem for fractional Laplacians}
\label{sec:energy_sol}

Let $g$ be a $C^\infty$ Riemannian metric on $\mathbb{R}^n$,  $n \geq 2$, which agrees with the Euclidean metric outside of a compact set, and let $\Omega \subset \R^n$ be a bounded open set with a $C^\infty$ boundary. Let $\alpha \in (0,1)$ and $F \in (H_0^\alpha(\Omega))^* = H^{-\alpha}(\Omega)$. Consider the following exterior Dirichlet problem:
\begin{equation}
\label{eq_9_1}
\begin{cases}
(-\Delta_g)^\alpha u = F & \text{in } \Omega, \\
u = 0 & \text{in } \Omega_e.
\end{cases}
\end{equation}
Recall that $\Omega_e := \R^n \setminus \overline{\Omega}$. We say that $u \in H^\alpha_0(\Omega)$ is an energy solution to \eqref{eq_9_1} if
\begin{equation}
\label{eq_9_3}
\mathcal{E}(u, \varphi) = \langle F, \overline{\varphi}\rangle_{H^{-\alpha} (\Omega), H_0^\alpha (\Omega)},
\end{equation}
for all $\varphi \in C^\infty_0(\Omega)$. 

We need the following existence and uniqueness result for energy solutions given in \cite[Theorem 12]{LPPS_2015}. We present the proof for the completeness and convenience of the reader.
\begin{prop}
\label{prop_finite_energy_sol}
Let $F \in (H_0^\alpha(\Omega))^*$. Then there exists a unique energy solution $u \in H_0^\alpha(\Omega)$ to \eqref{eq_9_1}. Moreover, we have 
\begin{equation}
\label{eq_9_2}
\|u\|_{H^\alpha(\R^n)} \le C \|F\|_{H^{-\alpha}(\Omega)}.
\end{equation}
\end{prop}

\begin{proof}
We shall use the Lax--Milgram lemma to prove this result. First, it follows from  \eqref{eq_2_15} that 
\[
|\mathcal{E}(u,\varphi)|\le C\|u\|_{H^\alpha(\R^n)}\|\varphi\|_{H^\alpha(\R^n)},
\]
for $u,\varphi\in  H_0^\alpha(\Omega)$, and we only need to check the coercivity of the form $\mathcal{E}$ on $H^\alpha_0(\Omega)$. To this end, using \eqref{eq_2_11} and \eqref{eq_2_13}, 
we get 
\begin{equation}
\label{eq_9_4}
\mathcal{E}(u,u)=\int_{\R^n}\int_{\R^n} K_g(z,x)|u(x)-u(z)|^2dxdz\ge \frac{1}{C}\int_{\R^n}\int_{\R^n}\frac{|u(x)-u(z)|^2}{|z-x|^{n+2\alpha}}dxdz,
\end{equation}
for $u\in H^\alpha_0(\Omega)$. Using the Poincar\'e inequality 
\[
\|u\|^2_{L^2(\R^n)}\le C\mathcal{E}(u,u), \quad u\in H^\alpha_0(\Omega), 
\]
see \cite[Lemma 2.9]{Felsinger_Kassmann_Voigt_2015} and \cite[Appendix A]{BWZ_2017},
and \eqref{eq_2_Sobolev_frac}, we conclude from \eqref{eq_9_4} that 
\begin{equation}
\label{eq_9_5}
\mathcal{E}(u,u)\ge c\|u\|^2_{H^\alpha(\R^n)},  \quad c>0, 
\end{equation}
for all $u\in H^\alpha_0(\Omega)$.  An application of the Lax--Milgram lemma shows that there exists a unique $u\in  H^\alpha_0(\Omega)$ satisfying \eqref{eq_9_3} for all $\varphi\in H^\alpha_0(\Omega)$. The bound \eqref{eq_9_2} follows from the coercivity estimate \eqref{eq_9_5}. 
\end{proof}

Let $F \in H^{-\alpha}(\Omega)$ and $f \in H^\alpha(\mathbb{R}^n)$. Consider the following exterior Dirichlet problem with inhomogeneous boundary conditions:
\begin{equation}
\label{eq_2_6_1_1}
\begin{cases}
(-\Delta_g)^\alpha u = F & \text{in } \Omega, \\
u = f & \text{in } \Omega_e.
\end{cases}
\end{equation}
We say that $u \in H^\alpha(\mathbb{R}^n)$ is an energy solution to \eqref{eq_2_6_1_1} if 
\[
\mathcal{E}(u, \varphi) = \langle F, \overline{\varphi} \rangle_{H^{-\alpha}(\Omega), H_0^\alpha(\Omega)},
\]
for all $\varphi \in C^\infty_0(\Omega)$, and $u - f \in H^\alpha_0(\Omega)$. We have the following result.
\begin{prop}
\label{prop_well-posedness}
Let $F \in H^{-\alpha}(\Omega)$ and $f \in H^\alpha(\mathbb{R}^n)$. Then \eqref{eq_2_6_1_1} has a unique energy solution $u \in H^\alpha(\mathbb{R}^n)$. Moreover, we have
\begin{equation}
\label{eq_2_6_1_2}
\|u\|_{H^\alpha(\mathbb{R}^n)} \le C(\|F\|_{H^{-\alpha}(\Omega)} + \|f\|_{H^\alpha(\mathbb{R}^n)}).
\end{equation}
\end{prop}

\begin{proof}
To demonstrate the existence of a solution $u$ to \eqref{eq_2_6_1_1}, we seek it in the form $u = v + f$, where $v \in H^\alpha_0(\Omega)$. We then have
\begin{equation}
\label{eq_2_6_1_2_new_form_1}
\mathcal{E}(v, \varphi) = \langle F, \overline{\varphi} \rangle_{H^{-\alpha} (\Omega), H_0^\alpha (\Omega)} - \mathcal{E}(f, \varphi),
\end{equation}
for all $\varphi \in C^\infty_0(\Omega)$. Here,
\[
\big| \langle F, \varphi \rangle_{H^{-\alpha} (\Omega), H_0^\alpha (\Omega)} - \mathcal{E}(f, \varphi) \big| \le C (\| F \|_{H^{-\alpha}(\Omega)} + \| f \|_{H^\alpha(\R^n)}) \| \varphi \|_{H^\alpha(\R^n)}.
\]
By Proposition \ref{prop_finite_energy_sol}, there exists a unique $v \in H_0^\alpha(\Omega)$ satisfying \eqref{eq_2_6_1_2_new_form_1} such that
\[
\| v \|_{H^\alpha(\R^n)} \le C (\| F \|_{H^{-\alpha}(\Omega)} + \| f \|_{H^\alpha(\R^n)}).
\]
Thus, the bound \eqref{eq_2_6_1_2} follows. The uniqueness of the solution to \eqref{eq_2_6_1_1} is a consequence of Proposition \ref{prop_finite_energy_sol}.
\end{proof}

Recalling from Proposition \ref{prop_eq_2_6} that the map $(-\Delta_g)^\alpha: H^\alpha(\R^n) \to H^{-\alpha}(\R^n)$ is continuous, we have the following observation.
\begin{lem}
\label{lem_distrubutional_solution}
Let $F \in H^{-\alpha}(\Omega)$ and $u \in H^\alpha_0(\Omega)$. Then $u$ satisfies
\begin{equation}
\label{eq_2_distribution_10}
 (-\Delta_g)^\alpha u = F \quad \text{in} \quad \mathcal{D}'(\Omega),
\end{equation}
if and only if $u$ is an energy solution to \eqref{eq_9_1} with $F$ replaced by $\sqrt{|g|}F$. Furthermore, the equation \eqref{eq_2_distribution_10} has a unique solution  $u \in H^\alpha_0(\Omega)$ in the sense of distributions on $\Omega$. 
\end{lem}

\begin{proof}
First, note that $\sqrt{|g|}F\in H^{-\alpha}(\Omega)$. Assume that $u \in H^\alpha_0(\Omega)$ is an energy solution to \eqref{eq_9_1} with $F$ replaced by $\sqrt{|g|}F$. Then, using Proposition \ref{prop_eq_2_6}, for any $\varphi \in C^\infty_0(\Omega)$, we obtain
\[
\langle F,  \sqrt{|g|} \overline{\varphi} \rangle_{H^{-\alpha} (\Omega), H_0^\alpha (\Omega)} = \mathcal{E}(u, \varphi) = \langle (-\Delta_g)^\alpha u,  \sqrt{|g|} \overline{\varphi} \rangle_{H^{-\alpha}(\R^n) \times H^\alpha(\R^n)}.
\]
Therefore, $u$ satisfies \eqref{eq_2_distribution_10} in the sense of $\mathcal{D}'(\Omega)$. The converse follows in the same manner. The uniqueness of the distributional solution hence follows from Proposition \ref{prop_finite_energy_sol}. 
\end{proof}

\subsection{Solvability of  the fractional Poisson equation}

Our starting point is the following fractional Sobolev inequality, which is a consequence of Proposition \ref{prop_fractional_Sobolev_inequality}; see also \cite[Theorems 4.2, 4.3]{Saloff-Coste_2009} and \cite[Theorem 2.1]{Roidos_Shao}.
\begin{lem}
Let $\alpha \in (0,1)$ and $n \ge 2$. Then there exists a constant $C > 0$ such that for all $u \in H^\alpha(\R^n)$, we have
\begin{equation}
\label{eq_2000_1}
\|u\|_{L^{\frac{2n}{n-2\alpha}}(\R^n)} \le C \|(-\Delta_g)^{\frac{\alpha}{2}} u\|_{L^2(\R^n; dV_g)}.
\end{equation}
\end{lem}
\begin{proof}
First, using \eqref{eq_ext_domain_100_-100}, \eqref{eq_2_6_seq_forms_new}, \eqref{eq_2_11}, \eqref{eq_2_13}, and Proposition \ref{prop_fractional_Sobolev_inequality}, we obtain for all $u \in C^\infty_0(\R^n)$,
\begin{align*}
\|(-\Delta_g)^{\frac{\alpha}{2}} u\|_{L^2(\R^n; dV_g)}^2 &= \mathcal{E}(u,u) = \int_{\R^n} \int_{\R^n} K_g(z,x) |u(x) - u(z)|^2 \, dx \, dz \\
&\ge \frac{1}{C} \int_{\R^n} \int_{\R^n} \frac{|u(x) - u(z)|^2}{|z - x|^{n + 2\alpha}} \, dx \, dz \ge C \|u\|_{L^{\frac{2n}{n - 2\alpha}}(\R^n)}^2,
\end{align*}
which implies \eqref{eq_2000_1} for $u \in C^\infty_0(\R^n)$.
This, together with the density of $C^\infty_0(\R^n)$ in $H^\alpha(\R^n)$ and the Sobolev embedding $H^\alpha(\R^n) \subset L^{\frac{2n}{n - 2\alpha}}(\R^n)$, implies that \eqref{eq_2000_1} holds for all $u \in H^\alpha(\R^n)$.
\end{proof}

We have the following result; see also \cite[page 2]{Choulli_Ouhabaz_2023}.
\begin{lem}
\label{lem_solution_poisson_laplacian}
Let $\alpha \in (0,1)$, $n \ge 2$, and let $F \in C^{\infty}_0(\R^n)$. Then the fractional Poisson equation
\begin{equation}
\label{eq_500_1}
(-\Delta_g)^\alpha w = -\Delta_g F \quad \text{in} \quad \R^n
\end{equation}
has a unique solution $w \in H^\alpha(\R^n)$ in the sense of $\mathcal{D}'(\R^n)$. Furthermore, $w \in H^k(\R^n)$ for all $k \in \N$, and by the Sobolev embedding, $w \in C^\infty(\R^n)$.
\end{lem}
\begin{proof}
Let us start with the uniqueness. To that end, assume that $(-\Delta_g)^\alpha w = 0$ in $\mathcal{D}'(\R^n)$ for some $w \in H^\alpha(\R^n)$. Hence, for all $\varphi \in C^\infty_0(\R^n)$, we get
\begin{equation}
\label{eq_2000_10}
0 = \langle (-\Delta_g)^\alpha w, \sqrt{|g|} \overline{\varphi} \rangle_{H^{-\alpha}(\R^n), H^{\alpha}(\R^n)} = ((-\Delta_g)^{\frac{\alpha}{2}} w, (-\Delta_g)^{\frac{\alpha}{2}} \varphi)_{L^2(\R^n; dV_g)}.
\end{equation}
Note that the second equality in \eqref{eq_2000_10} follows by functional calculus when $w \in C^\infty_0(\R^n)$, and then by density and continuity for $w \in H^\alpha(\R^n)$. Again, by density and continuity, we conclude from \eqref{eq_2000_10} that $\|(-\Delta_g)^{\frac{\alpha}{2}} w\|_{L^2(\R^n; dV_g)} = 0$, and therefore, by \eqref{eq_2000_1}, we get $w = 0$, establishing the uniqueness.

Let us now demonstrate that $w := (-\Delta_g)^{1-\alpha} F$ belongs to $H^\alpha(\mathbb{R}^n)$ and is a solution to \eqref{eq_500_1}. Indeed, $F \in C^\infty_0(\mathbb{R}^n) \subset \mathcal{D}((-\Delta_g)^k) = H^{2k}(\mathbb{R}^n)$ for all $k \in \mathbb{N}$, and by the functional calculus of self-adjoint operators, $(-\Delta_g)^{1-\alpha}: \mathcal{D}((-\Delta_g)^k) \to \mathcal{D}((-\Delta_g)^{k-(1-\alpha)})$ is continuous. Thus, $(-\Delta_g)^{1-\alpha}: C^\infty_0(\mathbb{R}^n) \to H^k(\mathbb{R}^n)$ is continuous for all $k \in \mathbb{N}$. By the functional calculus of self-adjoint operators (see \cite[Theorem 4.15]{Dimassi_Sjostrand_book}), we have
\[
(-\Delta_g)^\alpha w = (-\Delta_g)^\alpha (-\Delta_g)^{1-\alpha} F = -\Delta_g F.
\]
Thus, the result follows.
\end{proof}

\subsection{Mapping properties of the fractional Laplace--Beltrami operator}
Let $g$ be a $C^\infty$ Riemannian metric on $\mathbb{R}^n$, $n \geq 2$, which agrees with the Euclidean metric outside of a compact set, and let $0 < \alpha < 1$. By Theorem \ref{thm_pseudodiff_op}, we have $(-\Delta_g)^\alpha \in \Psi^{2\alpha}_{1,0}(\mathbb{R}^n)$. Thus, the general theory of such operators implies that $(-\Delta_g)^\alpha$ defines continuous linear maps:
\begin{align*}
&(-\Delta_g)^\alpha: C_0^\infty(\mathbb{R}^n) \to C^\infty(\mathbb{R}^n),\\
&(-\Delta_g)^\alpha: \mathcal{E}'(\mathbb{R}^n) \to \mathcal{D}'(\mathbb{R}^n),
\end{align*}
see \cite[Proposition 2.4]{Shubin_book_2001}. Furthermore, it follows from \cite[Theorem 7.1]{Shubin_book_2001} that for all $\tau \in \mathbb{R}$,
\[
(-\Delta_g)^\alpha: H^\tau_{comp}(\mathbb{R}^n) \to H^{\tau-2\alpha}_{loc}(\mathbb{R}^n).
\]
Here, $H^\tau_{loc}(\mathbb{R}^n) = \{ u \in \mathcal{D}'(\mathbb{R}^n) : \varphi u \in H^\tau(\mathbb{R}^n) \text{ for all } \varphi \in C_0^\infty(\mathbb{R}^n) \}$ and $H^\tau_{comp}(\mathbb{R}^n) = H^\tau_{loc}(\mathbb{R}^n) \cap \mathcal{E}'(\mathbb{R}^n)$.

However, we have the following global mapping properties of the fractional Laplace--Beltrami operator $(-\Delta_g)^\alpha$. 
\begin{lem}
\label{lem_global_bounds_fractional_laplacian}
Let $g$ be a $C^\infty$ Riemannian metric on $\mathbb{R}^n$, $n \geq 2$, which agrees with the Euclidean metric outside of a compact set, and let $0 < \alpha < 1$. Then, for all $\tau \in \mathbb{R}$, the map
\begin{equation}
\label{eq_702_0_general}
(-\Delta_g)^\alpha: H^{\tau}(\mathbb{R}^n) \to H^{\tau-2\alpha}(\mathbb{R}^n)
\end{equation}
is continuous.
\end{lem}
\begin{proof}
First, recall from Proposition \ref{prop_eq_2_6} that the map
\begin{equation}
\label{eq_702_1}
(-\Delta_g)^\alpha: H^\alpha(\mathbb{R}^n) \to H^{-\alpha}(\mathbb{R}^n)
\end{equation}
is continuous. We next claim that
\begin{equation}
\label{eq_702_0}
(-\Delta_g)^\alpha: H^{\alpha+r}(\mathbb{R}^n) \to H^{-\alpha+r}(\mathbb{R}^n)
\end{equation}
is continuous for all $r >0$. Indeed, by the functional calculus of self-adjoint operators, we have that $(-\Delta_g)^\alpha$ is continuous from $\mathcal{D}((-\Delta_g)^k)$ to $\mathcal{D}((-\Delta_g)^{k-\alpha})$ for all $k = 1, 2, \dots$, and therefore, the map
\begin{equation}
\label{eq_702_2}  
(-\Delta_g)^\alpha: H^{2k}(\mathbb{R}^n) \to H^{2k-2\alpha}(\mathbb{R}^n)
\end{equation}
is continuous. Interpolating between the bounds \eqref{eq_702_1} and \eqref{eq_702_2}, we obtain that 
the map 
\begin{equation}
\label{eq_702_3}
(-\Delta_g)^\alpha: H^{(1-\theta)\alpha+\theta 2k}(\mathbb{R}^n) \to H^{(1-\theta)(-\alpha)+\theta(2k-2\alpha)}(\mathbb{R}^n)
\end{equation}
is continuous for all $0 \leq \theta \leq 1$ and $k = 1, 2, \dots$; see \cite[Theorem 7.22]{Grubb_book}. For any $r > 0$, there is $k$ such that $0 < r \leq 2k-\alpha$, and thus, with $\theta = \frac{r}{2k-\alpha}$, the bound \eqref{eq_702_3} becomes \eqref{eq_702_0}. 

Furthermore, by duality, we also obtain that
\begin{equation}
\label{eq_702_0_negative_r}
(-\Delta_g)^\alpha: H^{\alpha-r}(\mathbb{R}^n) \to H^{-\alpha-r}(\mathbb{R}^n)
\end{equation}
is continuous for all $r > 0$. Indeed, letting $u, v \in C^\infty_0(\mathbb{R}^n)$, we have
\[
((-\Delta_g)^\alpha u, v)_{L^2(\mathbb{R}^n; dV_g)} = (u, (-\Delta_g)^\alpha v)_{L^2(\mathbb{R}^n; dV_g)},
\]
and therefore, using \eqref{eq_702_0}, we get
\begin{align*}
|((-\Delta_g)^\alpha u, v)_{L^2(\mathbb{R}^n; dV_g)}| &\leq \|u\|_{H^{\alpha-r}(\mathbb{R}^n; dV_g)} \|(-\Delta_g)^\alpha v\|_{H^{-\alpha+r}(\mathbb{R}^n; dV_g)} \\
&\leq C \|u\|_{H^{\alpha-r}(\mathbb{R}^n)} \|v\|_{H^{\alpha+r}(\mathbb{R}^n)}.
\end{align*}
Hence, \eqref{eq_702_0_negative_r} follows. Combining \eqref{eq_702_1},  \eqref{eq_702_0}, and \eqref{eq_702_0_negative_r}, we conclude \eqref{eq_702_0_general}.
\end{proof}

\subsection{The Vishik--Eskin estimates for the fractional Laplace--Beltrami operator}
\label{sec:VE_p1}

Let $g$ be a $C^\infty$ Riemannian metric on $\mathbb{R}^n$,  $n \geq 2$, which agrees with the Euclidean metric outside of a compact set, and let $0 < \alpha < 1$. According to Theorem \ref{thm_pseudodiff_op}, we have $(-\Delta_g)^\alpha \in \Psi^{2\alpha}_{1,0}(\mathbb{R}^n)$, which is a classical elliptic pseudodifferential operator on $\mathbb{R}^n$. Let $\Omega \subset \mathbb{R}^n$ be an open bounded set with $C^\infty$ boundary.  

\begin{lem}
\label{lem_mu-transmission}
The operator $(-\Delta_g)^\alpha$ on $\mathbb{R}^n$ is of type $\alpha$ relative to $\Omega$, i.e., it satisfies the $\mu=\alpha$ transmission condition in the sense of \cite[Definition 2.5]{Grubb_2015}, and has a factorization index $\alpha$ in the sense of \cite[Section 3]{Grubb_2015}. If $n=2$, $(-\Delta_g)^\alpha$ satisfies the root condition in the sense of \cite[Section 3]{Grubb_2015}. 
\end{lem}
\begin{proof}
First, we observe that both the type and the factorization index conditions are local conditions on the full and principal symbols of $(-\Delta_g)^\alpha$ on the boundary of $\Omega$, respectively. Furthermore, it is proven in \cite[Lemma 2.9]{Grubb_2015} that the fractional Laplacian $(-\Delta_{g_2})^\alpha$ on a $C^\infty$ compact Riemannian manifold $(M_2, g_2)$ without boundary is of type $\mu=\alpha$ for any open set $\omega \subset M_2$ with $C^\infty$ boundary. Additionally, \cite[Example 3.2]{Grubb_2015} shows that $(-\Delta_{g_2})^\alpha$ has a factorization index $\alpha$. This, together with the proof of Theorem \ref{thm_pseudodiff_op}, establishes the first two claims. The root condition follows from the form of the principal symbol of $(-\Delta_g)^\alpha$ given in Theorem~\ref{thm_pseudodiff_op}.
\end{proof}

As a consequence of the theory developed in \cite{Hormander_1965-66, Grubb_2015, Eskin_book_1981, Vishik_Eskin_1965}, in conjunction with Theorem \ref{thm_pseudodiff_op} and Lemma \ref{lem_mu-transmission}, we have the existence theory for the Dirichlet problem for the fractional Laplace–Beltrami operator, see \cite[Theorem 2.4.1]{Hormander_1965-66} and \cite[Theorem 3.1]{Grubb_2015}. 
\begin{thm}
\label{thm_Vishik_Eskin}
Let $\alpha\in (0,1)$ and $\delta\in (-1/2, 1/2)$. Then 
the mapping 
\[
H^{\alpha+\delta}_0(\Omega)\ni u\mapsto r_\Omega (-\Delta_g)^\alpha u\in H^{-\alpha+\delta}(\Omega)
\]
is a Fredholm operator. 
\end{thm}

\begin{rem} It can be inferred from the proof provided in \cite[Theorem 2.4.1]{Hormander_1965-66} that the local uniform estimates in $x$ for the pseudodifferential operator (as stated in Section \ref{sec_pseudodiff_op_1}) are sufficient for Theorem \ref{thm_Vishik_Eskin}. 
\end{rem}

Thanks to Lemma \ref{lem_global_bounds_fractional_laplacian}, we have the following consequence of Theorem \ref{thm_Vishik_Eskin}.
\begin{cor}
\label{cor_Vishin_Eskin}
Let $\alpha \in (0,1)$ and $\delta \in [0, 1/2)$. Then for any $f \in H^{\alpha+\delta}(\Omega_e)$, there is a unique solution $u \in H^{\alpha+\delta}(\mathbb{R}^n)$ to the problem
\begin{equation}
\label{eq_3_distribution_4}
\begin{cases}
(-\Delta_g)^\alpha u=0 & \quad \text{in} \quad \mathcal{D}'(\Omega), \\
u|_{\Omega_e}=f. 
\end{cases}
\end{equation}
\end{cor}
\begin{proof}
To show the existence of a solution $u \in H^{\alpha+\delta}(\mathbb{R}^n)$ to \eqref{eq_3_distribution_4}, we let $F \in H^{\alpha+\delta}(\mathbb{R}^n)$ be an extension of $f$ to $\mathbb{R}^n$ such that $F|_{\Omega_e} = f$. We look for $u$ in the form $u = v + F$, where $v \in H^{\alpha+\delta}(\mathbb{R}^n)$ satisfies
\begin{equation}
\label{eq_3_distribution_5}
\begin{cases}
 (-\Delta_g)^\alpha v = G & \quad \text{in} \quad \mathcal{D}'(\Omega), \\
v = 0 & \quad \text{in} \quad \Omega_e.
\end{cases}
\end{equation}
Here $G := -r_\Omega (-\Delta_g)^\alpha F \in H^{-\alpha+\delta}(\Omega)$ thanks to Lemma \ref{lem_global_bounds_fractional_laplacian}.

Now, Theorem \ref{thm_Vishik_Eskin} gives Fredholm solvability in $H^{\alpha+\delta}_0(\Omega)$ of the problem \eqref{eq_3_distribution_5}. By \cite[Theorem 3.5]{Grubb_2014}, the finite-dimensional kernel and range complement are independent of $\delta$ and they are trivial when $\delta=0$ by Lemma \ref{lem_distrubutional_solution}.  Hence, there is a unique $v \in H^{\alpha+\delta}(\mathbb{R}^n)$ satisfying \eqref{eq_3_distribution_5}, and therefore, the existence of a solution to \eqref{eq_3_distribution_4} follows. The uniqueness of a solution to \eqref{eq_3_distribution_4} follows from Lemma \ref{lem_distrubutional_solution}. We refer to \cite[Lemma A.1]{Ghosh_Salo_Uhlmann_2020} for similar arguments.
\end{proof}

For future reference, we state the following result, which is similar to Lemma~\ref{lem_distrubutional_solution}. The regularity improvement is an immediate consequence of Corollary \ref{cor_Vishin_Eskin}.

\begin{lem} 
\label{lem_distrubutional_solution_exterior_data}
Let $\alpha\in (0,1)$.  Let $f \in C^\infty_0(\Omega_e)$ and let $u \in H^\alpha(\mathbb{R}^n)$ be the unique energy solution to
\begin{equation}
\label{eq_distr_10_new}
\begin{cases}
(-\Delta_{g})^\alpha u = 0 & \text{in} \quad \Omega, \\
u = f & \text{in} \quad \Omega_e.
\end{cases}
\end{equation}
Then $u$ satisfies the first equation in \eqref{eq_distr_10_new} in $\mathcal{D}'(\Omega)$ and $u|_{\Omega_e} = f$. Furthermore, we have $u \in H^{\alpha+\delta}(\mathbb{R}^n)$ for all $\delta\in [0,1/2)$.
\end{lem}

\section{Proof of Theorem \ref{thm_main}. The heat semigroup and pseudodifferential approach}

\label{sec_proof_using_heat_semigroup_approach}

Let $\alpha \in (0,1)$ be fixed. Let $\Omega \subset \mathbb{R}^n$, $n \ge 2$, be a bounded open set with a $C^\infty$ boundary such that $\Omega_e$ is connected, and let $W_1, W_2 \subset \Omega_e$ be open and nonempty. Let $g_1, g_2$ be $C^\infty$ Riemannian metrics on $\mathbb{R}^n$ that agree with the Euclidean metric outside a compact set, such that
\begin{equation}
\label{eq_9001_1}
g_1|_{\Omega_e} = g_2|_{\Omega_e} =: g.
\end{equation}

\subsection{An auxiliary unique continuation type result}

Let $f \in C^\infty_0(W_1)$ and let $u_1^f, u_2^f \in H^\alpha(\mathbb{R}^n)$ be the unique energy solutions to
\begin{equation}
\label{eq_int_1_thm_new}
\begin{cases}
(-\Delta_{g_j})^\alpha u_j^f = 0 & \text{in} \quad \Omega, \\
u_j^f = f & \text{in} \quad \Omega_e,
\end{cases}
\end{equation}
for $j = 1, 2$, respectively. Our starting point is the following unique continuation type result, which shows that knowledge of the partial exterior Dirichlet--to--Neumann map and equation \eqref{eq_9001_1} determine the fractional Laplacian on the energy solutions to the exterior Dirichlet problem \eqref{eq_int_1_thm_new} in the exterior of $\Omega$.

\begin{lem}
\label{lem_exterior_determination}
The equality 
\begin{equation}
\label{eq_9001_2}
\Lambda_{g_1}^{W_1,W_2} = \Lambda_{g_2}^{W_1,W_2}, 
\end{equation}
together with \eqref{eq_9001_1}, implies that
\begin{equation}
\label{eq_4_prop_1_laplacian}
((-\Delta_{g_1})^\alpha u_1^f)|_{\Omega_e} = ((-\Delta_{g_2})^\alpha u_2^f)|_{\Omega_e},
\end{equation}
for all $f \in C^\infty_0(W_1)$. 
\end{lem}

\begin{proof}
First, note that \eqref{eq_9001_2} implies
\begin{equation}
\label{eq_4_prop_1_laplacian_direct_cons}
((-\Delta_{g_1})^\alpha u_1^f)|_{W_2} = ((-\Delta_{g_2})^\alpha u_2^f)|_{W_2},
\end{equation}
for all $f \in C^\infty_0(W_1)$. Hence, if $W_2 = \Omega_e$, we are done. Therefore, assume that $W_2 \neq \Omega_e$. We shall proceed to prove \eqref{eq_4_prop_1_laplacian} by considering two cases.

\noindent\textbf{Case I.} Assume that $W_1 \cap W_2 = \emptyset$. To establish \eqref{eq_4_prop_1_laplacian} in this case, we shall first show that the equalities \eqref{eq_4_prop_1_laplacian_direct_cons} and \eqref{eq_9001_1} imply the equality of the heat semigroups on the energy solutions $u_1^f$ and $u_2^f$ to the exterior Dirichlet problem \eqref{eq_int_1_thm_new},
\begin{equation}
\label{eq_4_prop_1}
(e^{t\Delta_{g_1}}u_1^f)(x) = (e^{t\Delta_{g_2}}u_2^f)(x), \quad t > 0, \quad x \in \Omega_e,
\end{equation}
for all $f \in C^\infty_0(W_1)$.

Here, we note that $e^{t\Delta_{g_j}}u_j^f$ satisfies the heat equation $(\partial_t - \Delta_{g_j})e^{t\Delta_{g_j}}u_j^f = 0$ on $(0, \infty) \times \mathbb{R}^n$, and hence, by parabolic hypoellipticity, we have $e^{t\Delta_{g_j}}u_j^f \in C^\infty((0, \infty) \times \mathbb{R}^n)$ for $j = 1, 2$ (see \cite[Theorem 7.4]{Grigoryan_book_2009}).

In order to show \eqref{eq_4_prop_1}, let $f \in C^\infty_0(W_1)$ and set $K := \supp(f) \subset W_1$. Let $\omega_2 \Subset W_2$ be an open bounded nonempty set. Then $K \cap \overline{\omega_2} = \emptyset$ and $\overline{\omega_2} \cap \overline{\Omega} = \emptyset$. Let $u_1^f, u_2^f \in H^\alpha(\mathbb{R}^n)$ be the energy solutions to \eqref{eq_int_1_thm_new}, corresponding to the metrics $g_1$ and $g_2$, respectively. By \eqref{eq_4_prop_1_laplacian_direct_cons}, we get 
\begin{equation}
\label{eq_4_1}
((-\Delta_{g_1})^{\alpha} u_1^f)|_{\omega_2} = ((-\Delta_{g_2})^{\alpha} u_2^f)|_{\omega_2} \in H^{-\alpha}(\omega_2).
\end{equation}

We now claim that, in view of \eqref{eq_2_6}, the equality \eqref{eq_4_1} implies that for any $\varphi \in C^\infty_0(\omega_2)$, the following holds:
\begin{equation}
\label{eq_4_2_new}
\begin{aligned}
\frac{1}{\Gamma(-\alpha)} &\int_0^\infty \left( (e^{t\Delta_{g_1}}u_1^f, \varphi)_{L^2(\mathbb{R}^n;dV_{g_1})} - (u_1^f, \varphi)_{L^2(\mathbb{R}^n;dV_{g_1})} \right) \frac{dt}{t^{1+\alpha}} \\
&= \frac{1}{\Gamma(-\alpha)} \int_0^\infty \left( (e^{t\Delta_{g_2}}u_2^f, \varphi)_{L^2(\mathbb{R}^n;dV_{g_2})} - (u_2^f, \varphi)_{L^2(\mathbb{R}^n;dV_{g_2})} \right) \frac{dt}{t^{1+\alpha}}.
\end{aligned}
\end{equation}

Indeed, it follows from \eqref{eq_4_1} that
\begin{equation}
\label{eq_4_2_new_0}
\langle (-\Delta_{g_1})^{\alpha} u_1^f, \overline{\varphi}\sqrt{|g|}\rangle_{\mathcal{D}'(\R^n), C^\infty_0(\R^n)}
= \langle (-\Delta_{g_2})^{\alpha} u_2^f, \overline{\varphi}\sqrt{|g|}\rangle_{\mathcal{D}'(\R^n), C^\infty_0(\R^n)},
\end{equation}
for all $\varphi \in C^\infty_0(\omega_2)$.
Now, let $\varphi \in C^\infty_0(\omega_2)$ be arbitrary. Using Proposition \ref{prop_eq_2_6} and \eqref{eq_2_6}, we obtain 
\begin{equation}
\label{eq_4_2_new_1}
\begin{aligned}
\langle (-\Delta_{g_1})^{\alpha} u_1^f, \overline{\varphi}\sqrt{|g|}\rangle_{\mathcal{D}'(\R^n), C^\infty_0(\R^n)}=\langle (-\Delta_{g_1})^{\alpha} u_1^f, \overline{\varphi}\sqrt{|g_1|}\rangle_{H^{-\alpha}(\R^n), H^\alpha(\R^n)}\\
=(u_1^f, (-\Delta_{g_1})^{\alpha}\varphi)_{L^2(\R^n;dV_{g_1})}=\frac{1}{\Gamma(-\alpha)} \int_0^\infty (u_1^f, e^{t\Delta_{g_1}}\varphi-\varphi)_{L^2(\R^n; dV_{g_1})}\frac{dt}{t^{1+\alpha}}\\
=\frac{1}{\Gamma(-\alpha)} \int_0^\infty \big((e^{t\Delta_{g_1}} u_1^f, \varphi)_{L^2(\R^n; dV_{g_1})}-(u_1^f, \varphi)_{L^2(\R^n; dV_{g_1})}\big)\frac{dt}{t^{1+\alpha}}.
\end{aligned}
\end{equation}

In the penultimate equality, we applied Fubini's theorem, and in the last equality, we used the fact that $e^{t\Delta_{g_1}}$ is symmetric on $L^2(\mathbb{R}^n; dV_{g_1})$. We shall now proceed to justify our use of Fubini's theorem. To that end, we will show that the function  
\begin{equation}
\label{eq_4_2_new_2}
(t, x) \mapsto u_1^f(x) \overline{(e^{t\Delta_{g_1}}\varphi(x)-\varphi(x))}\frac{1}{t^{1+\alpha}}\sqrt{|g_1(x)|}
\end{equation}
is in $L^1((0,\infty)\times \mathbb{R}^n)$. Indeed, using Tonelli's theorem, the Cauchy--Schwarz inequality, and Minkowski's inequality, we first obtain that 
\begin{equation}
\label{eq_4_2_new_3}
\begin{aligned}
\int_0^\infty\int_{\mathbb{R}^n} |u_1^f(x)| &|e^{t\Delta_{g_1}}\varphi(x)-\varphi(x)|\frac{dt}{t^{1+\alpha}} dV_{g_1}(x) \\
&\le \|u_1^f\|_{L^2(\mathbb{R}^n;dV_{g_1})} \int_0^\infty \|e^{t\Delta_{g_1}}\varphi-\varphi\|_{L^2(\mathbb{R}^n;dV_{g_1})} \frac{dt}{t^{1+\alpha}}.
\end{aligned}
\end{equation}

Using \eqref{eq_2_4_L_2_boundedness}, for $t>0$, we get 
\begin{equation}
\label{eq_4_2_new_4}
\|e^{t\Delta_{g_1}}\varphi-\varphi\|_{L^2(\mathbb{R}^n;dV_{g_1})} \le 2\|\varphi\|_{L^2(\mathbb{R}^n;dV_{g_1})}.
\end{equation}

Since $\varphi \in C^\infty_0(\omega_2)$, by the fundamental theorem of calculus, \eqref{eq_2_4_L_2_boundedness}, and Minkowski's inequality, we obtain that for $t > 0$,
\begin{equation}
\label{eq_4_2_new_5}
\|e^{t\Delta_{g_1}}\varphi-\varphi\|_{L^2(\mathbb{R}^n;dV_{g_1})} \le \int_0^t \|e^{s\Delta_{g_1}}\Delta_{g_1}\varphi\|_{L^2(\mathbb{R}^n;dV_{g_1})} \, ds \le t\|\Delta_{g_1}\varphi\|_{L^2(\mathbb{R}^n;dV_{g_1})},
\end{equation}
see also \cite[page 119, Exercise 4.39]{Grigoryan_book_2009}.

Relying on \eqref{eq_4_2_new_4} for $t \ge 1$ and on \eqref{eq_4_2_new_5} for $0 < t \le 1$, we see that the integral in \eqref{eq_4_2_new_3} is finite, and hence, \eqref{eq_4_2_new_2} follows. This completes the justification of the use of Fubini's theorem in \eqref{eq_4_2_new_1}. The claim \eqref{eq_4_2_new} now follows from \eqref{eq_4_2_new_0}, \eqref{eq_4_2_new_1}, and a similar computation with $g_2$ and $u_2^f$ in place of $g_1$ and $u_1^f$.

Using the conditions $u_1^f|_{\Omega_e} = u_2^f|_{\Omega_e} = f$ and \eqref{eq_9001_1}, we conclude from \eqref{eq_4_2_new} that
\begin{equation}
\label{eq_4_2_new_6}
\begin{aligned}
0 &= \int_0^\infty \left( e^{t\Delta_{g_1}}u_1^f - e^{t\Delta_{g_2}}u_2^f, \varphi \right)_{L^2(\omega_2; dV_{g})} \frac{dt}{t^{1+\alpha}} \\
&= \left( \int_0^\infty \left( e^{t\Delta_{g_1}}u_1^f - e^{t\Delta_{g_2}}u_2^f \right) \frac{dt}{t^{1+\alpha}}, \varphi \right)_{L^2(\omega_2; dV_{g})},
\end{aligned}
\end{equation}
for all $\varphi \in C^\infty_0(\omega_2)$. In the last equality in \eqref{eq_4_2_new_6}, we have applied Fubini's theorem, and we shall now proceed to justify its application. To do so, it suffices to verify that the function
\begin{equation}
\label{eq_4_7_new_just_1}
(t,x) \mapsto \frac{U^{(j)}(t, x)}{t^{1+\alpha}} \text{ is in } L^1((0, \infty) \times \omega_2),
\end{equation}
where 
\begin{equation}
\label{eq_4_7_new_just_2}
U^{(j)}(t, \cdot) = e^{t\Delta_{g_j}}u_j^f,
\end{equation}
with $j = 1, 2$. To establish \eqref{eq_4_7_new_just_1}, we note that $U^{(j)} \in C^\infty((0, \infty) \times \mathbb{R}^n)$, and using \eqref{eq_2_4}, we have for $x \in \omega_2$ and $t > 0$,
\begin{equation}
\label{eq_4_8}
\begin{aligned}
U^{(j)}(t,x) &= \int_{\mathbb{R}^n} e^{t\Delta_{g_j}}(x,y)u^f_j(y) \, dV_{g_j}(y) \\
&= \int_{\Omega} e^{t\Delta_{g_j}}(x,y)u^f_j(y) \, dV_{g_j}(y) + \int_{K} e^{t\Delta_{g_j}}(x,y)f(y) \, dV_{g_j}(y),
\end{aligned}
\end{equation}
where we recall that $K= \supp(f)$.
It follows from \eqref{eq_4_8} that for $x \in \omega_2$ and $t > 0$,
\begin{equation}
\label{eq_4_9}
|U^{(j)}(t,x)| \leq C \|e^{t\Delta_{g_j}}(\cdot,\cdot)\|_{L^\infty(\omega_2 \times \Omega)} \|u^f_j\|_{L^2(\Omega)} + \|e^{t\Delta_{g_j}}(\cdot,\cdot)\|_{L^\infty(\omega_2 \times K)} \|f\|_{L^1(W_1)}.
\end{equation}
Using \eqref{eq_2_5_1}, we deduce from \eqref{eq_4_9} that for $x \in \omega_2$ and $0 < t < 1$,
\begin{equation}
\label{eq_4_10}
|U^{(j)}(t,x)| \leq C e^{-\frac{\tilde{c}}{t}} \left( \|u^f_j\|_{L^2(\Omega)} + \|f\|_{L^1(W_1)} \right),
\end{equation}
where $\tilde{c} > 0$ depends on the Euclidean distances $d(\overline{\omega_2}, \overline{\Omega}) > 0$ and $d(\overline{\omega_2}, K) > 0$. For $x \in \omega_2$ and $t > 1$,
\begin{equation}
\label{eq_4_11}
|U^{(j)}(t,x)| \leq C t^{-\frac{n}{2}} \left( \|u^f_j\|_{L^2(\Omega)} + \|f\|_{L^1(W_1)} \right).
\end{equation}
Thus, \eqref{eq_4_10} and \eqref{eq_4_11}, together with the fact that $\omega_2$ is bounded, imply the claim \eqref{eq_4_7_new_just_1}, and therefore, the application of Fubini's theorem in the last equality in \eqref{eq_4_2_new_6} is justified.

Thus, it follows from \eqref{eq_4_2_new_6}, in view of the notation \eqref{eq_4_7_new_just_2}, that
\begin{equation}
\label{eq_4_6}
\int_0^\infty \big( U^{(1)}(t, x) - U^{(2)}(t, x) \big) \frac{dt}{t^{1+\alpha}} = 0, \quad x \in \omega_2.
\end{equation}
Here, we used the fact that the integral in \eqref{eq_4_6} is a continuous function of $x \in \omega_2$, which follows by the dominated convergence theorem, thanks to the bounds \eqref{eq_4_10} and \eqref{eq_4_11}.

Recalling that $g_1 = g_2 = g$ on $\Omega_e$, we next apply $\Delta_g^m$ to \eqref{eq_4_6}, where $m = 1, 2, \dots$. To that end, we shall show that 
\begin{equation}
\label{eq_4_7}
(t,x) \mapsto \frac{\Delta_g^m U^{(j)}(t, x)}{t^{1+\alpha}} \text{ is in } L^1((0, \infty) \times \omega_2),
\end{equation}
where $j = 1, 2$, for all $m = 1, 2, \dots$. In doing so, we let $m \ge 1$ and, using the fact that 
\begin{equation}
\label{eq_4_4}
\partial_t U^{(j)} = \Delta_{g_j} U^{(j)}, \quad (t, x) \in (0, \infty) \times \mathbb{R}^n,
\end{equation}
we obtain for $x \in \omega_2$ and $t > 0$,
\begin{equation}
\label{eq_4_12}
\begin{aligned}
\Delta_{g_j}^m U^{(j)}(t, x) &= \int_{\mathbb{R}^n} \partial_t^m e^{t \Delta_{g_j}}(x, y) u^f_j(y) \, dV_{g_j}(y) \\
&= \int_{\Omega} \partial_t^m e^{t \Delta_{g_j}}(x, y) u^f_j(y) \, dV_{g_j}(y) + \int_{K} \partial_t^m e^{t \Delta_{g_j}}(x, y) f(y) \, dV_{g_j}(y).
\end{aligned}
\end{equation}
Note that the differentiation under the integral sign in \eqref{eq_4_12} is justified by the bound \eqref{eq_2_5_1_derivatives}; see also \eqref{eq_4_13}, \eqref{eq_4_14}, and \eqref{eq_4_15} below.

It follows from \eqref{eq_4_12} that for $x \in \omega_2$ and $t>0$,
\begin{equation}
\label{eq_4_13}
\begin{aligned}
|\Delta_{g_j}^m U^{(j)}(t,x)|
&\le C\|\partial_t^m e^{t\Delta_{g_j}}(\cdot, \cdot)\|_{L^\infty(\omega_2 \times \Omega)} \|u^f_j\|_{L^2(\Omega)} \\
&\quad + \|\partial_t^m e^{t\Delta_{g_j}}(\cdot, \cdot)\|_{L^\infty(\omega_2 \times K)} \|f\|_{L^1(W_1)}.
\end{aligned}
\end{equation}
Now, \eqref{eq_4_13} together with \eqref{eq_2_5_1_derivatives} implies that for $x \in \omega_2$ and $0 < t < 1$,
\begin{equation}
\label{eq_4_14}
|\Delta_{g_j}^m U^{(j)}(t, x)| \le Ce^{-\frac{\tilde{c}}{t}}(\|u^f_j\|_{L^2(\Omega)} + \|f\|_{L^1(W_1)}),
\end{equation}
where $\tilde{c} > 0$ depends on the Euclidean distances $d(\overline{\omega_2}, \overline{\Omega}) > 0$ and $d(\overline{\omega_2},K) > 0$. For $x \in \omega_2$ and $t > 1$,
\begin{equation}
\label{eq_4_15}
|\Delta_{g_j}^m U^{(j)}(t, x)| \le Ct^{-\frac{n}{2} - m}(\|u^f_j\|_{L^2(\Omega)} + \|f\|_{L^1(W_1)}).
\end{equation}
Thus, the bounds \eqref{eq_4_14} and \eqref{eq_4_15} establish \eqref{eq_4_7}.

Now, letting $\varphi \in C^\infty_0(\omega_2)$, we conclude from \eqref{eq_4_6} in view of \eqref{eq_4_7} and \eqref{eq_4_7_new_just_1}, Fubini's theorem, and integration by parts, that
\begin{align*}
0 &= \bigg\langle \int_0^\infty \big(U^{(1)}(t,x) - U^{(2)}(t,x)\big) \frac{dt}{t^{1+\alpha}}, (\Delta_g^m \overline{\varphi}) \sqrt{|g|} \bigg\rangle_{\mathcal{D'}(\omega_2), C^\infty_0(\omega_2)} \\
&= \int_{\omega_2} \int_0^\infty \big(U^{(1)}(t,x) - U^{(2)}(t,x)\big) \frac{1}{t^{1+\alpha}} \Delta_g^m \overline{\varphi} \, dt \, dV_g(x) \\
&= \int_{\omega_2} \bigg( \int_0^\infty \Delta_g^m \big(U^{(1)}(t,x) - U^{(2)}(t,x)\big) \frac{1}{t^{1+\alpha}} \, dt \bigg) \overline{\varphi} \, dV_g(x).
\end{align*}
Therefore,
\begin{equation}
\label{eq_4_16}
\int_0^\infty \Delta_g^m \big(U^{(1)}(t,x) - U^{(2)}(t,x)\big) \frac{1}{t^{1+\alpha}} \, dt = 0, \quad x \in \omega_2,
\end{equation}
for all $m = 1, 2, \dots$. Here, we used the fact that the integrals in \eqref{eq_4_16} yield continuous functions of $x \in \omega_2$ for $m = 1, 2, \dots$. This follows from the dominated convergence theorem, thanks to the bounds \eqref{eq_4_14} and \eqref{eq_4_15}.

Using the fact that $g_1 = g_2 = g$ on $\Omega_e$, we obtain from \eqref{eq_4_4} that 
\begin{equation}
\label{eq_4_17}
(\partial_t - \Delta_g)(U^{(1)} - U^{(2)}) = 0, \quad \text{on} \quad (0, \infty) \times \Omega_e,
\end{equation}
and hence, for $m = 1, 2, \dots$,
\begin{equation}
\label{eq_4_18}
\Delta_g^m (U^{(1)} - U^{(2)}) = \partial_t^m (U^{(1)} - U^{(2)}), \quad \text{on} \quad (0, \infty) \times \Omega_e.
\end{equation}
Using \eqref{eq_4_18}, we have from \eqref{eq_4_16} that 
\begin{equation}
\label{eq_4_19}
\int_0^\infty \partial_t^m \big(U^{(1)}(t, x) - U^{(2)}(t, x)\big) \frac{1}{t^{1+\alpha}} \, dt = 0, \quad x \in \omega_2,
\end{equation}
for $m = 1, 2, \dots$. Integrating by parts $m$ times in \eqref{eq_4_19} and in view of \eqref{eq_4_6}, we get 
\begin{equation}
\label{eq_4_20}
\int_0^\infty \big(U^{(1)}(t, x) - U^{(2)}(t, x)\big) \frac{1}{t^{1+\alpha+m}} \, dt = 0, \quad x \in \omega_2,
\end{equation}
for $m = 0, 1, 2, \dots$. Note that when integrating by parts, there are no contributions from the endpoints thanks to the bounds \eqref{eq_4_10}, \eqref{eq_4_11}, \eqref{eq_4_14}, and \eqref{eq_4_15}.

To proceed, we follow \cite{FGKU_2021}. Making the change of variables $\tau = 1/t$ in \eqref{eq_4_20} and letting
\[
\varphi(\tau) = \big(U^{(1)}(\tau^{-1}, x) - U^{(2)}(\tau^{-1}, x)\big) \tau^{\alpha - 1}, \quad x \in \omega_2,
\]
we obtain
\begin{equation}
\label{eq_4_21}
\int_0^\infty \varphi(\tau) \tau^m \, d\tau = 0,
\end{equation}
for $m = 0, 1, 2, \dots$. Using \eqref{eq_4_10} and \eqref{eq_4_11}, we get 
\begin{equation}
\label{eq_4_21_for_holom}
|\varphi(\tau)| \le \mathcal{O}(1) e^{-c\tau} \tau^{\alpha - 1},
\end{equation}
for some $c > 0$. Thus, thanks to \eqref{eq_4_21_for_holom}, an application of \cite[Theorem 3.3.7]{Lerner_book} shows that the Fourier transform of $1_{[0,\infty)} \varphi$,
\[
\mathcal{F}(1_{[0,\infty)} \varphi)(\xi) = \int_0^\infty \varphi(\tau) e^{-i\xi \tau} \, d\tau
\]
is holomorphic for $\text{Im} \, \xi < c$. Now \eqref{eq_4_21} implies that $\mathcal{F}(1_{[0,\infty)} \varphi)$ vanishes at $0$ with all derivatives. Hence, $\varphi(\tau) = 0$ for $\tau > 0$, and therefore,
\[
U^{(1)}(t, x) - U^{(2)}(t, x) = 0, \quad t > 0, \quad x \in \omega_2.
\]

Given that $U^{(j)} \in C^\infty((0, \infty) \times \mathbb{R}^n)$ for $j = 1, 2$, and since $U^{(1)} - U^{(2)}$ satisfies the heat equation \eqref{eq_4_17}, combined with the fact that $\Omega_e$ is connected, the unique continuation property for the heat equation (see \cite[Sections 1 and 4]{Lin_1990}) implies that 
\[
U^{(1)}(t, x) = U^{(2)}(t, x) = 0, \quad t > 0, \quad x \in \Omega_e,
\]
thus proving \eqref{eq_4_prop_1}.

Now, let $\varphi \in C^\infty_0(\Omega_e)$ be arbitrary. Using \eqref{eq_4_prop_1} and the fact that $u_1^{f}|_{\Omega_e} = u_2^{f}|_{\Omega_e}$, we conclude from \eqref{eq_4_2_new_1} with $\omega_2$ replaced by $\Omega_e$, as well as from a similar computation with $g_2$ and $u_2^f$ in place of $g_1$ and $u_1^f$, that 
\[
\langle (-\Delta_{g_1})^{\alpha} u_1^f, \overline{\varphi} \sqrt{|g|} \rangle_{\mathcal{D}'(\mathbb{R}^n), C^\infty_0(\mathbb{R}^n)}
= \langle (-\Delta_{g_2})^{\alpha} u_2^f, \overline{\varphi} \sqrt{|g|} \rangle_{\mathcal{D}'(\mathbb{R}^n), C^\infty_0(\mathbb{R}^n)}.
\]
Hence, \eqref{eq_4_prop_1_laplacian} follows, completing the proof of the lemma in the case when $W_1 \cap W_2 =\emptyset$.

\noindent\textbf{Case II.} 
Assume that $W_1 \cap W_2 \neq \emptyset$. We shall show that this case can be reduced to the previous case. Indeed, let $f \in C^\infty_0(W_1)$ and set $K := \supp(f)$. 
Consider a finite open cover $\{W_{1,k}\}_{k=1}^N$, $N \in \mathbb{N}$, of the compact set $K$,
\begin{equation}
\label{eq_22_cover_new}
K \subset \bigcup_{k=1}^N W_{1,k},
\end{equation}
such that for each $k = 1, \dots, N$, we have $W_{1,k} \subset W_1$ and there exists an open set $W_{2,k} \subset W_2$ with $W_{1,k} \cap W_{2,k} = \emptyset$.
To justify the existence of such a cover, let $x \in K$ be arbitrary and consider two cases. If $x \in W_1 \cap W_2$, then since $W_1 \cap W_2$ is open, there exist open sets $W_{1,x}, W_{2,x} \subset W_1 \cap W_2$ such that $x \in W_{1,x}$ and $W_{1,x} \cap W_{2,x} = \emptyset$.  On the other hand, if $x \notin W_1 \cap W_2$, then in particular $x \notin W_2$. Let $\omega_2$ be an open set such that $\overline{\omega_2} \subset W_2$. Since $\text{dist}(x, \overline{\omega_2}) > 0$, there exists an open set $W_{1,x} \subset W_1$ such that $x \in W_{1,x}$ and $W_{1,x} \cap \omega_2 = \emptyset$. Thus, the collection $\{W_{1,x}\}_{x \in K}$ forms an open cover of $K$. By the compactness of $K$, we can extract a finite subcover, yielding the required collection $\{W_{1,k}\}_{k=1}^N$.

Let $\varphi_k \in C^\infty_0(W_{1,k})$, $k=1, \dots, N$, be a partition of unity subordinate to the cover \eqref{eq_22_cover_new}:
\[
\sum_{k=1}^N \varphi_k = 1 \quad \text{near} \quad K.
\]
We then have
\[
f = \sum_{k=1}^N f_k, \quad f_k = f \varphi_k \in C^\infty_0(W_{1,k}).
\]
Let $u_j^{f_k} \in H^\alpha(\mathbb{R}^n)$ be the energy solutions to
\[
\begin{cases}
(-\Delta_{g_j})^\alpha u_j^{f_k} = 0 & \text{in} \quad \Omega, \\
u_j^{f_k} = f_k & \text{in} \quad \Omega_e,
\end{cases}
\]
with $j = 1, 2$, and $k = 1, \dots, N$. Letting $u_j^f \in H^\alpha(\mathbb{R}^n)$ denote the energy solutions to \eqref{eq_int_1_thm_new} for $j = 1, 2$, we observe that
\begin{equation}
\label{eq_22_linearity}
u_j^f = \sum_{k=1}^N u_j^{f_k}.
\end{equation}
By \eqref{eq_9001_2}, we know
\[
((-\Delta_{g_1})^\alpha u_1^{f_k})|_{W_{2,k}} = ((-\Delta_{g_2})^\alpha u_2^{f_k})|_{W_{2,k}},
\]
for $k = 1, \dots, N$. Since $W_{1,k} \cap W_{2,k} = \emptyset$, this situation falls under Case I, and therefore, we conclude that
\[
((-\Delta_{g_1})^\alpha u_1^{f_k})|_{\Omega_e} = ((-\Delta_{g_2})^\alpha u_2^{f_k})|_{\Omega_e},
\]
for $k = 1, \dots, N$. Combining this with \eqref{eq_22_linearity}, we deduce that
\[
((-\Delta_{g_1})^\alpha u_1^f)|_{\Omega_e} = ((-\Delta_{g_2})^\alpha u_2^f)|_{\Omega_e}.
\]
Thus, we have established \eqref{eq_4_prop_1_laplacian} for all $f \in C^\infty_0(W_1)$ in the case when $W_1 \cap W_2 \neq \emptyset$. This completes the proof.
\end{proof}

\subsection{From the partial exterior Dirichlet--to--Neumann map to the full exterior Dirich\-let--to--Neumann map}

Let $\delta \ge 0$ be fixed such that $\delta \in (\alpha - \frac{1}{2}, \frac{1}{2})$, and let $h \in H^{\alpha + \delta}(\Omega_e)$. By Corollary \ref{cor_Vishin_Eskin}, there exists a unique solution $u_j = u_j^h \in H^{\alpha + \delta}(\mathbb{R}^n)$ to the problem
\begin{equation}
\label{eq_800_1}
\begin{cases}
(-\Delta_{g_j})^\alpha u_j = 0 & \quad \text{in} \quad \mathcal{D}'(\Omega), \\
u_j|_{\Omega_e} = h,
\end{cases}
\end{equation}
for $j = 1, 2$. Associated with \eqref{eq_800_1}, we define the full exterior Dirichlet-to-Neumann map
\[
\Lambda_{g_j}^{\Omega_e, \Omega_e}: H^{\alpha + \delta}(\Omega_e) \to H^{-\alpha + \delta}(\Omega_e), \quad h \mapsto ((-\Delta_{g_j})^\alpha u_j^h)|_{\Omega_e},
\]
where $u_j = u_j^h \in H^{\alpha + \delta}(\mathbb{R}^n)$ is the unique solution to \eqref{eq_800_1}, for $j = 1, 2$. 

The following result allows us to transition from knowing the exterior partial Dirichlet--to--Neumann map, with Dirichlet data compactly supported in $W_1$ and Neumann data measured on $W_2$, to knowing the full exterior Dirichlet--to--Neumann map, where Dirichlet data is provided over the entire $\Omega_e$, not necessarily compactly supported, and Neumann data is measured on $\Omega_e$.

\begin{lem}
The equality $\Lambda_{g_1}^{W_1,W_2} = \Lambda_{g_2}^{W_1,W_2}$ and \eqref{eq_9001_1} imply that $\Lambda_{g_1}^{\Omega_e,\Omega_e} = \Lambda_{g_2}^{\Omega_e,\Omega_e}$.
\end{lem}

\begin{proof}
First, let $q \in C^\infty_0(\Omega_e) \subset H^{\alpha + \delta}(\Omega_e)$. Then, by Corollary \ref{cor_Vishin_Eskin}, there exists a unique solution $u_j = u_j^{q} \in H^{\alpha + \delta}(\mathbb{R}^n)$ to problem \eqref{eq_800_1} with $h = q$, for $j = 1, 2$. We shall proceed to show that
\begin{equation}
\label{eq_810_1}
((-\Delta_{g_1})^\alpha u_1^{q})|_{W_1} = ((-\Delta_{g_2})^\alpha u_2^{q})|_{W_1}.
\end{equation}

To that end, let $f \in C^\infty_0(W_1)$ and let $u_j^f \in H^\alpha(\mathbb{R}^n)$ be the unique energy solution to the exterior Dirichlet problem \eqref{eq_int_1_thm_new}, for $j = 1, 2$. Then, by Lemma \ref{lem_distrubutional_solution_exterior_data}, $u_j^f$ satisfies the first equation in \eqref{eq_int_1_thm_new} in $\mathcal{D}'(\Omega)$ and $u_j^f|_{\Omega_e} = f$, for $j = 1, 2$. Furthermore, by Lemma \ref{lem_distrubutional_solution_exterior_data}, we conclude that $u_j^f \in H^{\alpha + \delta}(\mathbb{R}^n)$.

Using Proposition \ref{prop_eq_2_6}, we obtain
\begin{equation}
\label{eq_810_2}
\langle (-\Delta_{g_j})^\alpha u_j^f, \sqrt{|g_j|} \overline{u_j^{q}} \rangle_{H^{-\alpha}(\mathbb{R}^n), H^{\alpha}(\mathbb{R}^n)} = \langle \sqrt{|g_j|} u_j^f, \overline{(-\Delta_{g_j})^\alpha u_j^{q}} \rangle_{H^{\alpha}(\mathbb{R}^n), H^{-\alpha}(\mathbb{R}^n)},
\end{equation}
for $j = 1, 2$. By Lemma \ref{lem_global_bounds_fractional_laplacian}, we have $(- \Delta_{g_j})^\alpha u_j^f, (- \Delta_{g_j})^\alpha u_j^{q} \in H^{-\alpha + \delta}(\mathbb{R}^n)$ with $-\alpha + \delta \in (-1/2, 1/2)$. We also note that $\sqrt{|g_j|} \, u_j^f, \sqrt{|g_j|} \, u_j^{q} \in H^{\alpha + \delta}(\mathbb{R}^n) \subset H^{\alpha - \delta}(\mathbb{R}^n)$.

Now, since $-\alpha + \delta \in (-1/2, 1/2)$, using \eqref{eq_Sobolev_duality_mikko} together with the facts that $(-\Delta_{g_j})^\alpha u_j^f = 0$ in $\mathcal{D}'(\Omega)$ and $(-\Delta_{g_j})^\alpha u_j^{q} = 0$ in $\mathcal{D}'(\Omega)$, as well as $u_j^q|_{\Omega_e}=q$ and  $u_j^f|_{\Omega_e}=f$, we obtain from \eqref{eq_810_2} that
\begin{equation}
\label{eq_810_3}
\langle (-\Delta_{g_j})^\alpha u_j^f, \overline{q} \sqrt{|g_j|} \rangle_{\Omega_e} = \langle \sqrt{|g_j|} f, \overline{(-\Delta_{g_j})^\alpha u_j^{q}} \rangle_{\Omega_e},
\end{equation}
for $j = 1, 2$. Indeed, since $\alpha - \delta \in (-1/2, 1/2)$, we have $H^{\alpha - \delta}(\Omega) = H^{\alpha - \delta}_0(\Omega) = (H^{-\alpha + \delta}(\Omega))^*$; see Theorem \ref{thm_Agranovich}. This fact, together with $(-\Delta_{g_j})^\alpha u_j^f = 0$ in $\mathcal{D}'(\Omega)$ and the density of $C_0^\infty(\Omega)$ in $H_0^{\alpha - \delta}(\Omega)$, allows us to conclude that
\begin{equation}
\label{eq_810_3_extra_1}
\begin{aligned}
\langle (-\Delta_{g_j})^\alpha u_j^f, \sqrt{|g_j|} \overline{u_j^{q}} \rangle_\Omega = \langle (-\Delta_{g_j})^\alpha u_j^f, \sqrt{|g_j|} \overline{u_j^{q}} \rangle_{H^{-\alpha + \delta}(\Omega), H_0^{\alpha - \delta}(\Omega)} = 0.
\end{aligned}
\end{equation}
Similarly, we have
\begin{equation}
\label{eq_810_3_extra_2}
\langle \sqrt{|g_j|} u_j^f, \overline{(-\Delta_{g_j})^\alpha u_j^{q}} \rangle_{\Omega} = 0.
\end{equation}
The claim \eqref{eq_810_3} follows from \eqref{eq_810_2}, \eqref{eq_Sobolev_duality_mikko}, \eqref{eq_810_3_extra_1}, and \eqref{eq_810_3_extra_2}.

Now, in view of $\Lambda_{g_1}^{W_1,W_2} = \Lambda_{g_2}^{W_1,W_2}$ and \eqref{eq_9001_1}, we conclude from Lemma \ref{lem_exterior_determination} that 
\begin{equation}
\label{eq_810_4}
((-\Delta_{g_1})^\alpha u_1^f)|_{\Omega_e} = ((-\Delta_{g_2})^\alpha u_2^f)|_{\Omega_e}.
\end{equation}

Using \eqref{eq_810_4} together with the fact that $q\sqrt{|g|} \in C^\infty_0(\Omega_e)$, we obtain from \eqref{eq_810_3} that
\begin{equation}
\label{eq_810_4_new_1}
\langle \sqrt{|g|}\, f , \overline{(-\Delta_{g_1})^\alpha u_1^q} - \overline{(-\Delta_{g_2})^\alpha u_2^q}\rangle_{\Omega_e}=0.
\end{equation}
Here, we have also used $ \sqrt{|g_1|}= \sqrt{|g_2|}=\sqrt{|g|} $ on $ \Omega_e $, which is a consequence of \eqref{eq_9001_1}. Recalling that $ f \in C^{\infty}_0(W_1) $ is arbitrary, the equality \eqref{eq_810_4_new_1} implies the claim \eqref{eq_810_1}. 

Now, using Lemma \ref{lem_exterior_determination} with the equality \eqref{eq_810_1} instead of \eqref{eq_9001_2}, and with $ W_1 = \Omega_e $ and $ W_2 = W_1 $, we conclude that 
\begin{equation}
\label{eq_810_1_all_Omega_e}
((-\Delta_{g_1})^\alpha u_1^q)|_{\Omega_e} = ((-\Delta_{g_2})^\alpha u_2^q)|_{\Omega_e},
\end{equation}
where $ q \in C^\infty_0(\Omega_e) $ is arbitrary. 

Finally, let $h \in H^{\alpha+\delta}(\Omega_e)$. By Corollary \ref{cor_Vishin_Eskin}, there is a unique solution $u_j = u_j^{h} \in H^{\alpha+\delta}(\R^n)$ to the problem \eqref{eq_800_1}, for $j = 1, 2$. Writing \eqref{eq_810_2} with the solutions $u_j^q$ and $u_j^h$ instead of $u_j^f$ and $u_j^q$, and following the same reasoning as above but relying on \eqref{eq_810_1_all_Omega_e} instead of \eqref{eq_810_4}, we obtain that
\begin{equation}
\label{eq_810_1_all_Omega_e_new_100}
((-\Delta_{g_1})^\alpha u_1^h)|_{\Omega_e} = ((-\Delta_{g_2})^\alpha u_2^h)|_{\Omega_e},
\end{equation}
where $h \in H^{\alpha+\delta}(\Omega_e)$ is arbitrary. Note that here we used:
\[
\langle (-\Delta_{g_1})^\alpha u_1^q, h\sqrt{|g|}\rangle_{\Omega_e} = \langle (-\Delta_{g_2})^\alpha u_2^q, h\sqrt{|g|}\rangle_{\Omega_e},
\]
which follows from \eqref{eq_810_1_all_Omega_e} due to the fact that $h\sqrt{|g|}\in H^{\alpha+\delta}(\Omega_e) \subset H^{\alpha-\delta}(\Omega_e) = H^{\alpha-\delta}_0(\Omega_e)$, $\alpha-\delta \in (-1/2,1/2)$, and the density of $C_0^\infty(\Omega_e)$ in $H^{\alpha-\delta}_0(\Omega_e)$. Equation \eqref{eq_810_1_all_Omega_e_new_100} shows that $\Lambda_{g_1}^{\Omega_e,\Omega_e} = \Lambda_{g_2}^{\Omega_e,\Omega_e}$.
 \end{proof}

\subsection{From the full exterior Dirichlet--to--Neumann map to the exterior source--to--solution map}

Let $F \in C_0^\infty(\Omega_e)$. By Lemma \ref{lem_solution_poisson_laplacian}, the Poisson equation 
\begin{equation}
\label{eq_801_1}
(-\Delta_{g_j})^\alpha w_j = (-\Delta_g) F \quad \text{in} \quad \mathcal{D}'(\R^n)
\end{equation}
has a unique solution $w_j = w_j^F \in H^\alpha(\R^n)$, $j=1,2$. Here, $g$ is given by \eqref{eq_9001_1}. Associated with \eqref{eq_801_1}, we define the exterior source--to--solution map 
\[
L_{g_j}^{\Omega_e,\Omega_e}: C_0^\infty(\Omega_e) \to H^\alpha(\Omega_e), \quad F \mapsto w_j^F|_{\Omega_e},
\]
where $w_j = w_j^F \in H^\alpha(\R^n)$ is the unique solution to \eqref{eq_801_1}, $j=1,2$. 

The following result enables us to pass from the knowledge of the full exterior Dirichlet--to--Neumann map to the exterior source--to--solution map.
\begin{lem}
The equality $\Lambda_{g_1}^{\Omega_e,\Omega_e} = \Lambda_{g_2}^{\Omega_e,\Omega_e}$ implies $L_{g_1}^{\Omega_e,\Omega_e} = L_{g_2}^{\Omega_e,\Omega_e}$.
\end{lem}

\begin{proof}
Let $F \in C^\infty_0(\Omega_e)$ and let $w_1 = w_1^F \in H^\alpha(\R^n)$ be the unique solution to the equation 
\begin{equation}
\label{eq_801_2}
(-\Delta_{g_1})^\alpha w_1 = (-\Delta_g) F \quad \text{in} \quad \mathcal{D}'(\R^n).
\end{equation}
By Lemma \ref{lem_solution_poisson_laplacian}, $w_1^F \in H^k(\R^n)$ for all $k \in \N$ and, therefore, in particular, $w_1^F \in H^{\alpha+\delta}(\mathbb{R}^n)$, where $\delta \geq 0$, $\delta \in (\alpha - 1/2, 1/2)$, is given in the definition of the exterior Dirichlet--to--Neumann maps $\Lambda_{g_j}^{\Omega_e,\Omega_e}$ for $j = 1, 2$. 

Letting $h := w_1^F|_{\Omega_e} \in H^{\alpha+\delta}(\Omega_e)$, we see that $w_1^F$ is a solution to the problem
\begin{equation}
\label{eq_801_3}
\begin{cases}
(-\Delta_{g_1})^\alpha u_1 = 0 & \quad \text{in} \quad \mathcal{D}'(\Omega), \\
u_1|_{\Omega_e} = h.
\end{cases}
\end{equation}
Therefore, thanks to Corollary \ref{cor_Vishin_Eskin}, by the uniqueness of the solution to \eqref{eq_801_3} we can write $w_1^F = u_1^h$.

Let $u_2 = u_2^h \in H^{\alpha+\delta}(\mathbb{R}^n)$ be the unique solution to the problem
\begin{equation}
\label{eq_801_4}
\begin{cases}
(-\Delta_{g_2})^\alpha u_2 = 0 & \quad \text{in} \quad \mathcal{D}'(\Omega), \\
u_2|_{\Omega_e} = h,
\end{cases}
\end{equation}
as shown in Corollary \ref{cor_Vishin_Eskin}. Given that $\Lambda_{g_1}^{\Omega_e,\Omega_e} = \Lambda_{g_2}^{\Omega_e,\Omega_e}$, it follows that
\[
((-\Delta_{g_1})^\alpha u_1^h)|_{\Omega_e} = ((-\Delta_{g_2})^\alpha u_2^h)|_{\Omega_e}.
\]
Thanks to this equality, and the first equations in \eqref{eq_801_3} and \eqref{eq_801_4}, we deduce that 
\[
\supp((-\Delta_{g_1})^\alpha u_1^h - (-\Delta_{g_2})^\alpha u_2^h) \subset \partial \Omega.
\]
 By Lemma \ref{lem_global_bounds_fractional_laplacian}, it follows that $((-\Delta_{g_1})^\alpha u_1^h - (-\Delta_{g_2})^\alpha u_2^h) \in H^{-\alpha+\delta}(\mathbb{R}^n)$ with $-\alpha+\delta \in (-1/2, 1/2)$. Therefore, Theorem \ref{thm_Agranovich} allows us to conclude that
\begin{equation}
\label{eq_801_7}
(-\Delta_{g_1})^\alpha u_1^h = (-\Delta_{g_2})^\alpha u_2^h \in H^{-\alpha+\delta}(\mathbb{R}^n).
\end{equation}

Recalling that $u_1^h = w_1^F$ is the solution to \eqref{eq_801_2}, and using \eqref{eq_801_7}, we see that $u_2^h$ satisfies
\begin{equation}
\label{eq_801_8}
(-\Delta_{g_2})^\alpha u_2^h = (-\Delta_g)F \quad \text{in} \quad \mathcal{D}'(\mathbb{R}^n).
\end{equation}
Hence, we conclude that $u_1^h, u_2^h \in H^\alpha(\mathbb{R}^n)$ satisfy \eqref{eq_801_2} and \eqref{eq_801_8}, respectively, with $u_1^h|_{\Omega_e} = u_2^h|_{\Omega_e}$. Therefore, $L_{g_1}^{\Omega_e,\Omega_e}(F) = L_{g_2}^{\Omega_e,\Omega_e}(F)$ for $F \in C_0^\infty(\Omega_e)$.
\end{proof}

\subsection{From the exterior source--to--solution map to the heat kernel of the Laplacian in the exterior for all positive times}
The following result enables us to transition from the knowledge of the exterior source--to--solution map to the knowledge of the heat kernels $e^{t\Delta_{g_j}}(x,y)$ of the Laplacian $-\Delta_{g_j}$ in the exterior for all $t > 0$.

\begin{lem}
\label{lem_pass_heat_kernel}
The equality $L_{g_1}^{\Omega_e, \Omega_e} = L_{g_2}^{\Omega_e, \Omega_e}$ implies that for all $t > 0$ and $x, y \in \Omega_e$, we have
\[
e^{t\Delta_{g_1}}(x,y) = e^{t\Delta_{g_2}}(x,y).
\]
\end{lem}
\begin{proof}
Let $F \in C_0^\infty(\Omega_e)$ and define $F^{(m)} := \Delta_g^m F \in C_0^\infty(\Omega_e)$ for $m = 0, 1, 2, \ldots$. By Lemma \ref{lem_solution_poisson_laplacian}, the Poisson equation 
\[
(-\Delta_{g_j})^\alpha w_j = (-\Delta_g) F^{(m)} \quad \text{in} \quad \mathcal{D}'(\mathbb{R}^n)
\]
has a unique solution $w_j = w_j^{F^{(m)}} \in H^\alpha(\mathbb{R}^n)$ for $j = 1, 2$. Furthermore, Lemma~\ref{lem_solution_poisson_laplacian} gives that $w_j^{F^{(m)}} = (-\Delta_{g_j})^{1-\alpha} F^{(m)}$.

The equality $L_{g_1}^{\Omega_e, \Omega_e} = L_{g_2}^{\Omega_e, \Omega_e}$ implies that 
\begin{equation}
\label{eq_801_9}
((-\Delta_{g_1})^{1-\alpha} F^{(m)})|_{\Omega_e} = ((-\Delta_{g_2})^{1-\alpha} F^{(m)})|_{\Omega_e}
\end{equation}
for $m = 0, 1, 2, \ldots$.

Letting $\tilde{\alpha} := 1 - \alpha \in (0, 1)$ and taking an open non-empty set $\omega_2 \Subset \Omega_e$ such that $\overline{\omega_2} \cap \supp(F) = \emptyset$, and using \eqref{eq_2_6}, we obtain from \eqref{eq_801_9} that 
\begin{equation} 
\label{eq_801_10} 
\int_0^\infty \big((e^{t \Delta_{g_1}} - e^{t \Delta_{g_2}}) \Delta_g^m F\big)|_{\omega_2} \frac{dt}{t^{1+\tilde{\alpha}}} = 0 \end{equation} 
for $m = 0, 1, 2, \ldots$.
 Following the argument in the proof of Lemma \ref{lem_exterior_determination} (see also \cite{FGKU_2021}), we complete the proof.
\end{proof}

\subsection{From the heat kernel of the Laplacian to the recovery of the Riemannian metric} 
\label{sec:hk}

Having Lemma \ref{lem_pass_heat_kernel} at hand, Theorem \ref{thm_main} follows by applying the following result with $(N_j, g_j) = (\mathbb{R}^n, g_j)$ for $j = 1, 2$ and $\mathcal{O} = \Omega_e$. This result was established in \cite[Theorem 1.5]{FGKU_2021}, as a consequence of \cite[Theorem 2]{HLOS_2018}.

\begin{thm}
\label{thm_from_heat_semigroup_to_wave}
Let $(N_1, g_1)$ and $(N_2, g_2)$ be smooth, connected, complete Riemannian manifolds of dimension $n \ge 2$ without boundary. Let $\mathcal{O}_j \subset N_j$, $j = 1, 2$, be open nonempty sets, and assume that $\mathcal{O}_1 = \mathcal{O}_2 := \mathcal{O}$. Furthermore, assume that 
\begin{equation}
\label{eq_5_2}
e^{t\Delta_{g_1}}(x, y) = e^{t\Delta_{g_2}}(x, y), \quad x, y \in \mathcal{O}, \quad t > 0.
\end{equation}
Then there exists a diffeomorphism $\Phi: N_1 \to N_2$ such that $\Phi^* g_2 = g_1$ on $N_1$ and $\Phi(x) = x$ for all $x \in \mathcal{O}$.
\end{thm}

\begin{rem}
\label{rem_Teemu}
In \cite[Theorem 1.5]{FGKU_2021}, there was no claim that $\Phi(x) = x$ for all $x \in \mathcal{O}$. However, this claim follows from a small modification of the proof of \cite[Theorem 2]{HLOS_2018}, which was used in the proof of \cite[Theorem 1.5]{FGKU_2021}. This was explained to us by Teemu Saksala. 
\end{rem}

\section{Proof of Theorem \ref{thm_main}. A (degenerate) elliptic extension approach}

\label{sec_degenerate_elliptic_extension_approach}

In this section, we present a second, equivalent perspective on the formulation and proof of Theorem \ref{thm_main}. This is based on a degenerate elliptic extension perspective and allows one to use $L^2$-based elliptic arguments.
To this end, in what follows below, we view the fractional Laplacian in terms of an elliptic extension perspective, see \cite{CS07} for the constant and \cite{ST10} for the variable coefficient setting. 
As above, we work in the smooth set-up and consider the following assumptions on the domain and metric.

\begin{assump}[Metric and domain regularity]
\label{assump:main}
In our discussion below, we suppose that the following conditions are satisfied by the metric and domain.
\begin{itemize}
\item[(a)] We will suppose that the sets $\Omega, W_1, W_2 \subset \R^n$ are open, non-empty with $\Omega $ bounded $C^{\infty}$ regular and such that $\Omega_e:= \R^n \setminus \overline{\Omega}$ is connected. Moreover,  $W_1, W_2 \subset \Omega_e$.
\item[(b)] For the metric $g$ we assume that for some open, bounded set  $\tilde{\Omega} \subset \R^n$ such that $\Omega \subset \tilde{\Omega}$ we have
\begin{align*}
g\in C^{\infty}(\R^n, \R^{n \times n}_{sym}), \ g \mbox{ is uniformly elliptic}, \ g = id \mbox{ in } \R^n \setminus \overline{\tilde{\Omega}}.
\end{align*}
\end{itemize}
\end{assump}

We remark that the assumptions in (b) above, in particular, imply that $(\R^n, g)$ is a complete Riemannian manifold.

\begin{rem}[Regularity assumptions]
\label{rmk:reg}
Let us comment on the regularity assumption from above:
\begin{itemize}
\item[(i)] We remark that for the argument presented in this section it would be possible to weaken the regularity assumptions on the metric and domain. Indeed, the steps below remain valid with only slight modifications for $C^{k,\epsilon}$ regular metrics $g$ for a finite value of $k \in \N$ large enough and $C^{1,\epsilon}$ regular domains $\Omega$ for some $\epsilon \in (0,1)$ instead of imposing the $C^{\infty}$ smoothness assumptions as above. In order to avoid technicalities, we however do not discuss this further but restrict our attention to the smooth set-up.
\item[(ii)] Without loss of generality, we may assume that not only $\Omega$ but also $W_1, W_2$ are smooth open sets. Indeed, this can always be achieved by decreasing the size of the open sets $W_1, W_2$ and considering the measurements only on these correspondingly smaller, smooth sets. We will always assume this in the following sections.
\end{itemize}
\end{rem}

In the setting of Assumption \ref{assump:main}, in what follows below, we adopt an extension perspective introduced in the seminal works \cite{CS07, ST10}. Hence, instead of working in \eqref{eq_int_1} with the fractional Laplacian $(-\D_g)^{\alpha}$ defined through functional calculus, for $f\in H^{\alpha}(\R^n)$ we will investigate \eqref{eq_int_1} by means of the following mixed Dirichlet-Neumann problem:
\begin{align}
\label{eq:mixed}
\begin{split}
\nabla \cdot x_{n+1}^{1-2\alpha} \tilde{g} \nabla \tilde{v} & = 0 \mbox{ in } \R^{n+1}_+,\\
\tilde{v} & = f \mbox{ in } \Omega_e \times \{0\},\\
\lim\limits_{x_{n+1} \rightarrow 0} x_{n+1}^{1-2\alpha} \p_{n+1} \tilde{v} & = 0 \mbox{ in } \Omega .
\end{split}
\end{align}
Here 
\begin{align}
\label{eq:tildeg}
\tilde{g} = \sqrt{|g|} \begin{pmatrix} g^{-1} & 0 \\ 0 & 1 \end{pmatrix} \in C^{\infty}(\R^{n+1}_+, \R^{(n+1)\times (n+1)})
\end{align}
 is a uniformly elliptic tensor field. We highlight that $\tilde{g}$ only depends on the tangential variables $x \in \R^n$ and is independent of the $x_{n+1}$ variable. 

The perspective in \eqref{eq:mixed} is motivated by the seminal works of Caffarelli-Silvestre \cite{CS07} and Stinga-Torrea \cite{ST10} which then was generalized to various other contexts \cite{ATW18,BGS15,CG11}. In these articles it was proved that the fractional Laplacian can itself be interpreted as a Dirichlet--to--Neumann map. Indeed, for a function $u\in H^{\alpha}(\R^n)$ by a Fourier diagonalization argument in the constant coefficient setting \cite{CS07} and by functional calculus in the variable coefficient setting \cite{ST10}, it follows that
\begin{align*}
(-\Delta_g)^{\alpha}u(x)= c_{\alpha}\lim\limits_{x_{n+1} \rightarrow 0} x_{n+1}^{1-2 \alpha} \p_{n+1} \tilde{u}(x,x_{n+1}) \mbox{ in } H^{-\alpha}(\R^n),
\end{align*}
where $\tilde{u} \in \dot{H}^{1}(\R^{n+1}_+, x_{n+1}^{1-2\alpha})$ is a weak solution of
\begin{align*}
\nabla \cdot x_{n+1}^{1-2\alpha} \tilde{g}  \nabla \tilde{u} & = 0 \mbox{ in } \R^{n+1}_+,\\
\tilde{u} & = u \mbox{ on } \R^n \times \{0\},
\end{align*}
with some constant $c_{\alpha} \in \R \setminus \{0\}$.
For a detailed description of the function spaces and the solution notion for these problems we refer to Section \ref{sec:prelim} below. Hence, as shown in Section \ref{sec:equiv} below, for  $f\in H^{\alpha}(\R^n)$, after projecting onto the traces of the solutions on $\R^n \times \{0\}$ the formulation \eqref{eq:mixed} is indeed an equivalent interpretation of the nonlocal problem
\begin{align*}
(-\D_g)^{\alpha} v &= 0 \mbox{ in } \Omega,\\
v & = f \mbox{ in } \Omega_e.
\end{align*}
Moreover, the bulk equation in \eqref{eq:mixed} can be rewritten as
\begin{align}
\label{eq:LB_extended}
(-\p_{n+1} x_{n+1}^{1-2\alpha} \p_{n+1} - x_{n+1}^{1-2\alpha}\Delta_g)  \tilde{v}  = 0 \mbox{ in } \R^{n+1}_+
\end{align}
 by modifying the test function by a factor $\sqrt{|g|}$ in the weak form of the equation.
The weak formulation of \eqref{eq:mixed} and the relevant operators and functions spaces related to it will be discussed in Section \ref{sec:weak1} below. 

The interpretation from \eqref{eq:mixed} will provide access to elliptic PDE techniques in the analysis of the inverse problem from the introduction and hence gives rise to a complementary perspective on the formulation and the proof of Theorem \ref{thm_main}. Here and in what follows below, it will suffice to consider real-valued functions in considering solutions to these elliptic PDEs.

\subsection{Preliminaries}

\label{sec:prelim}

In this section, we present the interpretation of \eqref{eq:mixed} in a weak sense and introduce the measurement data for the inverse problem based on this notion of a weak solution. We also outline the relation of our extension-based interpretation of \eqref{eq:mixed} with the functional calculus interpretation of the fractional Laplacian.

\subsubsection{Sets and function spaces}
\label{sec:func_spaces}

Before turning to the weak form of \eqref{eq:mixed} we collect the relevant sets and function spaces for the analysis of the extension problem in the upper half space. In what follows below, even if not explicitly stated, we will always assume that $n\in \N$ with $n\geq 2$ as well as $\alpha \in (0,1)$.

We begin by compiling a list of sets which will be used in the discussion below: 
\begin{itemize}
\item $\R^{n+1}_+:= \{(x,x_{n+1}) \in \R^{n+1}: \ x \in \R^n, \ x_{n+1}>0\}$. In the subsequent sections, we will mainly use the convention that $x \in \R^n$ and that a point in $\R^{n+1}_+$ is denoted by $(x,x_{n+1})$. We will also denote Lebesgue measure on $\R^{n+1}_+$ by $d(x,x_{n+1})$.
\item For $x_0  \in \R^{n+1}_+$, $r>0$, we use the notation $ B_r^+(x_0):= \{x \in \R^{n+1}_+: \ |x-x_0|< r\}$ to denote a relatively open ball in the upper half plane. For $x_0 =0$, we also omit the center from the notation and simply write $B_r^+$.
\item We often identify $\R^{n}\times \{0\}$ with $\R^n$ and denote the $n$-dimensional open ball of radius $r>0$ centered around $x_0 \in \R^n$ by $B_{r}(x_0)$. Similarly as above, we set $B_r:= B_r(0)$. When convenient, we will also occasionally view $\R^n$ as a Riemannian manifold with the Riemannian metric $g$ satisfying the conditions from Assumption \ref{assump:main}. In this case, we also sometimes use the associated volume form which we denote by $d V_g$. As in the previous sections, we use the notation $|g|:= \det(g_{jk})$.
 \end{itemize}

In what follows below, we will make use of various function spaces. We first recall the definition of the homogeneous $\dot{H}^{\alpha}(\R^n)$ semi-norm as $\|u\|_{\dot{H}^{\alpha}(\R^n)}:=\||\cdot|^{\alpha} \mathcal{F} u\|_{L^2(\R^n)}$, where $\F u(\xi):=  \int\limits_{\R^n} e^{-i x \cdot \xi} u(x )dx $ is used to denote the Fourier transform for $u \in L^1(\R^n)$. The associated homogeneous Sobolev space is then defined as 
\begin{align*}
\dot{H}^{\alpha}(\R^n):= \mbox{ closure of } C_0^{\infty}(\R^n) \mbox{ with respect to the } \dot{H}^{\alpha}(\R^n) \mbox{ semi-norm}. 
\end{align*}
This represents the natural energy space associated with the fractional Laplacian. As in the previous sections, we also use the notation 
\begin{align*}
H^{\alpha}(\R^n):= \{u \in \mathcal{S}'(\R^n): \ \|(1+ |\cdot|^2)^{\alpha/2} \mathcal{F} u\|_{L^2(\R^n)}< \infty\}
\end{align*}
for the associated inhomogeneous space.

For the extended equation \eqref{eq:mixed}, we work with homogeneous, weighted spaces which are naturally connected with the equation \eqref{eq:mixed}. To this end, we first introduce the following semi-norm: for $\tilde{v} \in C_0^{\infty}(\overline{\R^{n+1}_+})$ we set
\begin{align*}
\|\tilde{v}\|_{\dot{H}^1(\R^{n+1}_+, x_{n+1}^{1-2\alpha})}^2 := \|\nabla \tilde{v}\|_{L^2(\R^{n+1}_+, x_{n+1}^{1-2\alpha})}^2:= \int\limits_{\R^{n+1}_+} x_{n+1}^{1-2\alpha} |\nabla \tilde{v}|^2 d(x,x_{n+1}).
\end{align*}
Here we use the notation $C_0^{\infty}(\overline{\R^{n+1}_+})$ for smooth compactly supported functions on $\overline{\R^{n+1}_+}$. In particular, these functions do not need to vanish on $\R^{n}\times \{0\}$. Building on this, we define the following class of functions
 \begin{align*} 
 X_{\alpha}:=\dot{H}^{1}(\R^{n+1}_+, x_{n+1}^{1-2\alpha}):= \mbox{closure of $C^{\infty}_0(\overline{\R^{n+1}_+})$ with respect to  $\|\cdot\|_{\dot{H}^1(\R^{n+1}_+, x_{n+1}^{1-2\alpha})}$}.
 \end{align*}
For $\tilde{v} \in C_0^{\infty}(\overline{\R^{n+1}_+})$, and hence by density also for $\tilde{v} \in \dot{H}^{1}(\R^{n+1}_+, x_{n+1}^{1-2\alpha})$, we have the following trace estimate into $\dot{H}^{\alpha}(\R^n)$:
\begin{align}
\label{eq:trace}
\|\tilde{v}|_{\R^n \times \{0\}}\|_{\dot{H}^{\alpha}(\R^n)} \leq C \|x_{n+1}^{\frac{1-2\alpha}{2}} \nabla \tilde{v}\|_{L^2(\R^{n+1}_+)}.
\end{align}
We refer to \cite[Equation (2.3)]{BCdPS13} and \cite{N93} for this estimate. Here and in the subsequent sections, we will use the notation $\tilde{v}|_{\R^n \times \{0\}}:= \tilde{v}|_{\{x_{n+1}=0\}}= \mathrm{Tr}(\tilde{v})|_{\{x_{n+1}=0\}}$  to denote the trace onto the subspace $\R^n \times \{0\}$. In order to distinguish between the extended function and its trace onto the boundary in our notation, for $\tilde{v}\in X_{\alpha}$, we will also often set $v:=\tilde{v}|_{\R^n \times \{0\}}$. 
The estimate \eqref{eq:trace} can further be combined with the Sobolev embedding $\dot{H}^{\alpha}(\R^n) \rightarrow L^{\frac{2n}{n-2\alpha}}(\R^n)$, see, e.g., \cite[Theorem 2.1]{BCdPS13}, which yields the following Sobolev trace estimate
\begin{align}
\label{eq:Sob_trace}
\|\tilde{v}|_{\R^n \times \{0\}}\|_{L^{\frac{2n}{n-2\alpha}}(\R^n)} \leq C \|x_{n+1}^{\frac{1-2\alpha}{2}} \nabla \tilde{v}\|_{L^2(\R^{n+1}_+)}
\end{align}
for $\tilde{v}\in C_0^{\infty}(\overline{\R^{n+1}_+})$ and hence also for $\tilde{v} \in \dot{H}^{1}(\R^{n+1}_+, x_{n+1}^{1-2\alpha})$.
Then, endowed with the natural bilinear form
\begin{align*}
(\cdot, \cdot)_{\dot{H}^{1}(\R^{n+1}_+, x_{n+1}^{1-2\alpha})}: X_{\alpha} \times X_{\alpha} &\rightarrow \R,\\
(\tilde{u}, \tilde{v}) \mapsto (\tilde{u}, \tilde{v})_{\dot{H}^{1}(\R^{n+1}_+, x_{n+1}^{1-2\alpha})}&:= \int\limits_{\R^{n+1}_+} x_{n+1}^{1-2\alpha} \nabla \tilde{u} \cdot \nabla \tilde{v} d(x,x_{n+1}),
\end{align*}
the space $X_{\alpha}$ becomes a Hilbert space. Indeed, only positive definiteness needs to be checked; this however follows from the fact that $\|\tilde{v}\|_{\dot{H}^1(\R^{n+1}_+, x_{n+1}^{1-2\alpha})} = 0$ iff $ \nabla \tilde{v} = 0$ together with the Sobolev trace estimate \eqref{eq:Sob_trace}.

Relying on the trace estimates \eqref{eq:trace} and \eqref{eq:Sob_trace}, we further define a function space involving a trace condition
\begin{align*}
X_{\alpha,0,\Omega_e}:=\dot{H}^1_{0,\Omega_e}(\R^{n+1}_+, x_{n+1}^{1-2\alpha}) :=\{\tilde{v} \in \dot{H}^{1}(\R^{n+1}_+, x_{n+1}^{1-2\alpha}): \ \tilde{v}|_{\Omega_e \times \{0\}} = 0\}.
\end{align*}
This will play the role of taking into account the exterior data for the mixed Dirichlet-Neumann problem \eqref{eq:mixed}. Similarly as above, $(X_{\alpha,0,\Omega_e}, (\cdot, \cdot)_{\dot{H}^1(\R^{n+1}_+, x_{n+1}^{1-2 \alpha})})$ is a Hilbert space.

While mainly working in the homogeneous weighted Sobolev spaces from above, occasionally, we also use the following inhomogeneous version:
\begin{align*}
H^{1}(\R^{n+1}_+, x_{n+1}^{1-2\alpha}):= \{ \tilde{u}: \R^{n+1}_+ \rightarrow \R: \ \|\tilde{u}\|_{H^1(\R^{n+1}_+,x_{n+1}^{1-2\alpha})}<\infty\},
\end{align*}
where
\begin{align*}
\|\tilde{u}\|_{H^1(\R^{n+1}_+,x_{n+1}^{1-2\alpha})}^2 := \|x_{n+1}^{\frac{1-2\alpha}{2}}\tilde{u}\|_{L^2(\R^{n+1}_+)}^2 + \|x_{n+1}^{\frac{1-2\alpha}{2}} \nabla \tilde{u}\|_{L^2(\R^{n+1}_+)}^2.
\end{align*}
Similar as the spaces above, this also is a Hilbert space with the scalar product 
\begin{align*}
(\tilde{u}, \tilde{v})_{{H}^{1}(\R^{n+1}_+, x_{n+1}^{1-2\alpha})}:= \int\limits_{\R^{n+1}_+} x_{n+1}^{1-2\alpha} \nabla \tilde{u} \cdot \nabla \tilde{v} d(x,x_{n+1})+ \int\limits_{\R^{n+1}_+} x_{n+1}^{1-2\alpha}  \tilde{u}  \tilde{v} d(x,x_{n+1}).
\end{align*}

For later use, we also record the following Poincar\'e inequality. For $u \in \dot{H}^{\alpha}(\R^n)$ with $\supp(u)\subset K$ for some compact set $K \subset \R^n$ it holds that for a constant $C>0$ depending on $K,\alpha,n$ 
\begin{align}
\label{eq:Poincare_nonlocal}
\|u\|_{L^2(K)} \leq C \|u\|_{\dot{H}^{\alpha}(\R^n)}.
\end{align}
We refer to \cite{RO16} and the references therein for further information on this Poincar\'e estimate.

When dealing with finite differences and, in particular, in our regularity estimates, we will rely on
\begin{itemize}
\item a (tangential) finite difference operator: for $j \in \{1,\dots,n\}$, $h \in \R\setminus \{0\}$ and $f \in L^2_{loc}(\R^{n+1}_+)$ we set $\tau_{j,h} f(x,x_{n+1}):= f(x+ he_j,x_{n+1})- f(x,x_{n+1})$,
\item a (tangential) finite difference quotient operator: for $j \in \{1,\dots,n\}$, $h \in \R\setminus \{0\}$ and $f \in L^2_{loc}(\R^{n+1}_+)$ we set $\Delta_{j,h} f(x,x_{n+1}):= |h|^{-1} (f(x+ he_j,x_{n+1})- f(x,x_{n+1}))$. 
\end{itemize}

In the sections below, we will frequently make use of these difference quotients to increase the Sobolev differentiability of functions. To this end, we note the following finite difference estimates for Sobolev functions. These results are well-known, we collect and prove them for the convenience of the reader.

\begin{lem}
\label{lem:Sob_charact}
Let $\mu \in \R$ and $\beta \in (0,1)$. 
\begin{itemize}
\item[(i)] Let $f \in \dot{H}^{\mu}(\R^n)$. Assume that for some $K>0$
\begin{align*}
\sup\limits_{h \in (0,1)} \sum\limits_{j=1}^{n} h^{-\beta}\|\tau_{j,h} f\|_{\dot{H}^{\mu}(\R^n)} \leq K,
\end{align*}
then $f \in \dot{H}^{\mu + \epsilon}(\R^n)$ for any $\epsilon \in (0,\beta)$
and there exists a constant $C=C(n,\mu,\beta,\epsilon)>0$ such that
\begin{align*}
\|f\|_{\dot{H}^{\mu + \epsilon}(\R^n)} \leq C (K + \|f\|_{\dot{H}^{\mu }(\R^n)}).
\end{align*}
\item[(ii)] Let $f \in H^{\mu}(\R^n)$. Assume that for some $K>0$
\begin{align*}
\sup\limits_{h \in (0,1)}\sum\limits_{j=1}^{n} h^{-\beta}\|\tau_{j,h} f\|_{H^{\mu}(\R^n)} \leq K,
\end{align*}
then $f \in H^{\mu + \epsilon}(\R^n)$ for any $\epsilon \in (0,\beta)$
and there exists a constant $C=C(n,\mu,\beta,\epsilon)>0$ such that
\begin{align*}
\|f\|_{H^{\mu + \epsilon}(\R^n)} \leq C (K +  \|f\|_{H^{\mu }(\R^n)}) .
\end{align*}
\item[(iii)] Let $f \in H^{\mu+ \beta}(\R^n) $. Then, there exists $C>0$ such that
\begin{align*}
\sup\limits_{h \in (0,1)} \sum\limits_{j=1}^{n} h^{-\beta}\|\tau_{j,h} f\|_{{H}^{\mu}(\R^n)} \leq C \|f\|_{H^{\mu+\beta}(\R^n)}.
\end{align*}
 \item[(iv)] Let $\alpha \in (0,1)$ and $\delta \in [0,1/2)$. Let $f \in H^{-\alpha + \delta}(\R^n)$ and $v \in H^{\alpha + \beta-\delta}(\R^n)$. Then, for any $j \in \{1,\dots,n\}$ and $h \in (0,1) $,
 \begin{align*}
 |\langle f, \tau_{j,h} v \rangle_{H^{-\alpha}(\R^n), H^{\alpha}(\R^n)}| \leq   h^{\beta} \|f\|_{H^{-\alpha+ \delta}(\R^n)} \sup\limits_{h\in (0,1)} h^{-\beta} \| \tau_{j,h} v \|_{H^{\alpha-\delta}(\R^n)}.
 \end{align*}
\end{itemize}
\end{lem}

Finite difference arguments in the form of (i), (ii) will enter frequently in our regularity analysis. The estimates from (iii) and (iv) will mainly play a role in the extension perspective on the Vishik-Eskin estimates from Appendix \ref{sec:reg}. The estimate in (iv) is a slight adaptation of \cite[Lemma 2.2]{KM07}.

\begin{proof}
We first prove (i). Without loss of generality, by density, we assume that $f \in C_0^{\infty}(\R^n)$.
The claim then follows from the embedding properties of Besov spaces. We give a short proof for the convenience of the reader.
We need to estimate the $\dot{H}^{\mu + \beta}(\R^n)$ semi-norm of $f$. Without loss of generality, we prove the result for $\mu = 0$.

We begin by considering the high frequencies and fix $h \in (0,1)$. By assumption, we have that
\begin{align*}
h^{-2 \beta} \int\limits_{\{\xi \in \R^n: \ |\xi_j| \in [2^{-1} h^{-1}, 2 h^{-1}]\}} |e^{i \xi_j h}-1|^2 |\F f(\xi)|^2 d \xi \leq K^2.
\end{align*}
Hence, by Taylor expanding the symbol $|e^{i \xi_j h}-1|^2$, for $|\xi_j| \in [2^{-1} h^{-1}, 2  h^{-1}]$, we obtain that for some constant $c>1$
\begin{align*}
&c^{-1} \int\limits_{\{\xi \in \R^n: \ |\xi_j| \in [2^{-1} h^{-1}, 2  h^{-1}]\}}  |\xi_j|^{2\beta} |\F f(\xi)|^2 d \xi\\
  &\leq h^{-2 \beta} \int\limits_{\{\xi \in \R^n: \ |\xi_j| \in [2^{-1} h^{-1}, 2 h^{-1}]\}}  |e^{i \xi_j h}-1|^2 |\F f(\xi)|^2 d \xi \\
  &\leq c \int\limits_{\{\xi \in \R^n: \ |\xi_j| \in [2^{-1} h^{-1}, 2 h^{-1}]\}}  |\xi_j|^{2\beta} |\F f(\xi)|^2 d \xi.
\end{align*}
We next observe that there exists $C>1$ such that for $\epsilon \in (0,\beta)$
\begin{align*}
C^{-1}\!\int\limits_{\{\xi \in \R^n: \ |\xi_j| \in [2^{-1} h^{-1}, 2 h^{-1}]\}}  |\xi_j|^{2\beta} |\F f(\xi)|^2 d \xi
&\leq  h^{-2(\beta-\epsilon)} \int\limits_{\{ \xi \in \R^n: \ |\xi_j| \in [2^{-1} h^{-1}, 2 h^{-1}] \}}  |\xi_j|^{2\epsilon} |\F f(\xi)|^2 d \xi\\
&\leq C\int\limits_{\{\xi \in \R^n: \ |\xi_j| \in [2^{-1} h^{-1}, 2 h^{-1}]\}}  |\xi_j|^{2\beta} |\F f(\xi)|^2 d \xi.
\end{align*}
We note that this is valid for any $h \in (0,1)$ and hence, in particular, for $h = 2^{-k}$ for $k\in \N$. Summing these bounds over $k \in \N$, we obtain for $\epsilon \in (0,\beta)$
\begin{align*}
 \int\limits_{\{\xi \in \R^n: \ |\xi_j| \geq 1\}}  |\xi_j|^{2\epsilon} |\F f(\xi)|^2 d \xi 
 &\leq \sum\limits_{k\in \N} 2^{-2k(\beta-\epsilon)}\int\limits_{\{\xi \in \R^n: \ |\xi_j| \in [2^{-1} h^{-1}, 2 h^{-1}]\}}  |\xi_j|^{2\beta} |\F f(\xi)|^2 d \xi\\
& \leq \sum\limits_{k\in \N} 2^{-2k(\beta-\epsilon)} K^2 \leq C(\beta, \epsilon)K^2.
\end{align*}
As this holds for any $j \in \{1,\dots,n\}$, we obtain
\begin{align*}
 \int\limits_{\{|\xi| \geq 1 \}}  |\xi|^{2\epsilon} |\F f(\xi)|^2 d \xi \leq C(\beta, \epsilon)K^2.
\end{align*}

For the low frequencies, we recall the $\dot{H}^{\mu}(\R^n)$ bound for $f$, noting that $ |\xi|^{\mu + \epsilon} \leq |\xi|^{\mu}$ for $|\xi|\leq 1$. Combining the two estimates implies the desired result.

With (i) in hand, we turn to the proof of (ii). The argument for the high frequencies follows as above, while the low frequencies are directly controlled by the $L^2$ bound.

The proof of (iii) follows directly from an application of the Fourier transform. For $f\in C_0^{\infty}(\R^n)$, $h\in (0,1)$ and $j \in \{1,\dots,n\}$ we have
\begin{align*}
h^{-2\beta}\|\tau_{j,h} f\|_{{H}^{\mu}(\R^n)}^2
&=  h^{-2\beta} \int\limits_{\R^n} (1+|\xi|^2)^{\mu} |e^{i \xi_j h} -1 |^2 |\F(f)(\xi)|^2 d\xi\\
&\leq Ch^{-2\beta} \int\limits_{\R^n} (1+|\xi|^2)^{\mu}|\xi_j h|^{2\beta} |\F(f)(\xi)|^2 d\xi\\
&\leq  C\int\limits_{\R^n}(1+|\xi|^2)^{\mu} |\xi_j |^{2\beta} |\F(f)(\xi)|^2 d\xi.
\end{align*}
Here we have used that $|e^{i \xi_j h} -1 |^2 \leq C |\xi_j h|^{2\beta}$ for any $\beta \in (0,1)$.
By density this concludes the proof of (iii).

The proof of (iv) is analogous to the one of \cite[Lemma 2.2]{KM07}. We present it for completeness. By density, we may assume that $f, v \in C_0^{\infty}(\R^n)$. Then, by duality, for $h \in (0,1) $
\begin{align*}
\left| \int\limits_{\R^n} f \tau_{j,h} v dx \right|
&\leq \|f\|_{H^{-\alpha +\delta}(\R^n)} \|\tau_{j,h} v\|_{H^{\alpha - \delta}(\R^n)}\\
&\leq h^{\beta}  \|f\|_{H^{-\alpha +\delta}(\R^n)} \sup\limits_{h \in (0,1)} h^{-\beta} \|\tau_{j,h} v\|_{H^{\alpha - \delta}(\R^n)}.
\end{align*}
\end{proof}

\subsubsection{The mixed Dirichlet-Neumann problem}
\label{sec:weak1}

Using the function spaces from the previous section, we consider weak solutions of \eqref{eq:mixed} with data $f \in H^{\alpha}(\Omega_e)$.

\begin{defn}
\label{defi:sol}
Let $\alpha \in (0,1)$, $n\geq 2$ and suppose that the conditions from Assumption \ref{assump:main} hold.
Let $f \in H^{\alpha}(\Omega_e)$.
We say that a function $\tilde{u}^f \in X_{\alpha}$ is a \emph{weak solution of \eqref{eq:mixed} with data in $H^{\alpha}(\Omega_e)$} if
\begin{itemize}
\item[(i)] the bulk equation holds weakly, i.e.,
\begin{align*}
\int\limits_{\R^{n+1}_+} x_{n+1}^{1-2\alpha} \tilde{g} \nabla \tilde{u}^f \cdot \nabla \tilde{\varphi} d(x,x_{n+1}) = 0
\end{align*}
for all $\tilde{\varphi} \in X_{\alpha,0,\Omega_e}$,
\item[(ii)] the boundary data are attained in a trace sense, i.e., $\tilde{u}^f|_{\Omega_e \times \{0\}} = f$.
\end{itemize}
\end{defn}

For the well-definedness of the trace condition in (ii) we recall the trace estimates \eqref{eq:trace} and \eqref{eq:Sob_trace} from above, which are applicable since $\tilde{u}^f \in X_{\alpha}$.

We next argue that the problem in Definition \ref{defi:sol} is well-posed by the lemma of Lax-Milgram.

\begin{lem}
\label{lem:well-posedness}
Let $\alpha \in (0,1)$, $n\geq 2$ and suppose that the conditions from Assumption \ref{assump:main} hold.
Let $f \in H^{\alpha}(\Omega_e)$. Then there exists a unique weak solution $\tilde{u}^f \in X_{\alpha}$  of \eqref{eq:mixed} satisfying
\begin{align*}
\|u^f\|_{H^{\alpha}(\R^n)} + \|u^f\|_{L^{\frac{2n}{n-2\alpha}}(\R^n)} + \|x_{n+1}^{\frac{1-2\alpha}{2}} \nabla \tilde{u}^f \|_{L^2(\R^{n+1}_+)} \leq C \|f\|_{H^{\alpha}(\Omega_e)},
\end{align*}
where we have set $u^f: \R^n \rightarrow \R$, $u^f:= \tilde{u}^f|_{\R^n \times \{0\}}$.
\end{lem}

Here and in what follows below, in order to highlight, relate and distinguish the estimates for the function on the extended compared to the trace level, we often use the convention that the extended function is denoted with a tilde while we do not use a tilde for its restriction.

\begin{proof}
We argue in two steps. Without loss of generality, by the definition of $H^{\alpha}(\Omega_e)$ as a quotient space, we may assume that $f \in H^{\alpha}(\R^n)$. Indeed, the desired estimate then follows by considering the infimum over all admissible extensions.

First, we construct a solution $\tilde{v}^f \in X_{\alpha}$ of
\begin{align}
\label{eq:CS_weak_const}
\begin{split}
\nabla \cdot x_{n+1}^{1-2\alpha} \nabla \tilde{v}^f & = 0 \mbox{ in } \R^{n+1}_+,\\
\tilde{v}^f & = f \mbox{ on } \R^n.
\end{split}
\end{align} 
This exists by an application of a tangential Fourier transform and satisfies the bounds associated with the space $X_{\alpha}$:
\begin{align}
\label{eq:FT_bound} 
\|\lim\limits_{x_{n+1} \rightarrow 0} x_{n+1}^{1-2\alpha} \p_{n+1} \tilde{v}^f\|_{H^{-\alpha}(\R^n)}+ \|x_{n+1}^{\frac{1-2\alpha}{2}} \nabla \tilde{v}^f \|_{L^2(\R^{n+1}_+)} \leq C \|f\|_{H^{\alpha}(\R^n)}.
\end{align}
Indeed, carrying out a tangential Fourier transform, the equation \eqref{eq:CS_weak_const} turns into an ODE in the normal variable (decoupled in the tangential variable)
\begin{align*}
(\p_{n+1} x_{n+1}^{1-2\alpha} \p_{n+1} - x_{n+1}^{1-2\alpha} |\xi|^2) \F_{x}(\tilde{v}^f) & = 0 \mbox{ in } (0,\infty),\\
\F_{x}(\tilde{v}^f) & = \F_{x}(f) \mbox{ for }  x_{n+1}=0,
\end{align*} 
where $\F_{x}$ denotes the Fourier transform in tangential directions, i.e., for $u \in C_0^{\infty}(\R^{n+1}_+)$, we set $\F_x u(\xi, x_{n+1}):= \int\limits_{\R^n} e^{-i x \cdot \xi} u(x,x_{n+1}) dx$.
This ODE is solved by 
\begin{align}
\label{eq:explicit}
\F_{x}(\tilde{v}^f)(\xi, x_{n+1}) := \tilde{c}_{\alpha} (|\xi| x_{n+1}) ^{\alpha} K_{\alpha}(|\xi| x_{n+1}) \F_{x}(f)(\xi),
\end{align}
where $K_{\alpha}$ denotes the Bessel function of the second kind (see \cite[Chapter 10.25]{NIST}) and $\tilde{c}_{\alpha} \in \R \setminus \{0\}$.
Inserting this expression for $\tilde{v}^f$ into the norms on the left hand side in \eqref{eq:FT_bound} and using Plancherel's theorem in the tangential directions, implies the desired estimate from \eqref{eq:FT_bound}.

We then consider $\tilde{u}^f:= \tilde{v}^f + \tilde{w}^f$ with $\tilde{v}^f, \tilde{w}^f \in X_{\alpha}$ and seek for a solution $\tilde{w}^f$ of
\begin{align*}
\int\limits_{\R^{n+1}_+} x_{n+1}^{1-2\alpha}\tilde{g} \nabla \tilde{w}^f \cdot \nabla \tilde{\varphi} d(x,x_{n+1})& = -\int\limits_{\R^{n+1}_+} x_{n+1}^{1-2\alpha} \tilde{g} \nabla \tilde{v}^f \cdot \nabla \tilde{\varphi} d(x,x_{n+1})\\
&\mbox{ \qquad \qquad \qquad for all } \tilde{\varphi} \in X_{\alpha,0,\Omega_e}.
\end{align*}
We note that, by Hölder and the a priori estimate \eqref{eq:FT_bound} from above,
\begin{align*}
\left| \int\limits_{\R^{n+1}_+} x_{n+1}^{1-2\alpha}\tilde{g} \nabla \tilde{v}^f \cdot \nabla \tilde{\varphi} d(x,x_{n+1}) \right| \leq C \|f\|_{H^{\alpha}(\R^n)} \|x_{n+1}^{\frac{1-2\alpha}{2}} \nabla \tilde{\varphi}\|_{L^2(\R^{n+1}_+)}.
\end{align*}
Thus, for $f \in H^{\alpha}(\R^n)$,
\begin{align*}
\ell_{f}(\tilde{\varphi}):=  \int\limits_{\R^{n+1}_+} x_{n+1}^{1-2\alpha}\tilde{g} \nabla \tilde{v}^f \cdot \nabla \tilde{\varphi} d(x,x_{n+1}),
\end{align*}
defines a bounded functional on $X_{\alpha,0, \Omega_e }$.

As a second step, we invoke the lemma of Lax-Milgram in $X_{\alpha,0,\Omega_e}$. 
Recalling the bounds from \eqref{eq:FT_bound}, we hence obtain a unique solution of the weak equation for $\tilde{w}^f \in X_{\alpha,0, \Omega_e}$ satisfying
\begin{align*}
 \|x_{n+1}^{\frac{1-2\alpha}{2}} \nabla \tilde{w}^f \|_{L^2(\R^{n+1}_+)} \leq C \|f\|_{H^{\alpha}(\R^n)},
\end{align*}
and $\tilde{w}^f|_{\Omega_e \times \{0\}} = 0$. By the embedding $\dot{H}^{1}(\R^{n+1}_+, x_{n+1}^{1-2\alpha}) \rightarrow \dot{H}^{\alpha}(\R^n)$ (see \eqref{eq:trace}) and the compact support of $w^f:= \tilde{w}^f|_{\R^n \times \{0\}}$ and by Poincar\'e's inequality (see \eqref{eq:Poincare_nonlocal} above), we also infer that
\begin{align*}
 \| w^f \|_{H^{\alpha}(\R^{n})} \leq C \|f\|_{H^{\alpha}(\R^n)}.
\end{align*}
Thus, by the trace condition for $\tilde{v}^f$ and the triangle inequality, the function $\tilde{u}^f := \tilde{w}^f + \tilde{v}^f$ satisfies the desired estimates. 
\end{proof}

With this notation, using the associated bilinear form, we can also define different versions of generalized Dirichlet-to-Neumann maps. We recall the observations from Remark \ref{rmk:reg}(ii) concerning the imposed regularity assumptions on the sets $W_1, W_2$.

\begin{defn}
\label{defi:Neumann1}
Let $\alpha \in (0,1)$, $n\geq 2$ and suppose that the conditions from Assumption \ref{assump:main} hold.
Let $W', W'' \subset \Omega_e \subset \R^n$ be open, non-empty and smooth. Then, we set
\begin{align*}
\tilde{\Lambda}_{g}^{W',W''}: H_0^{\alpha}(W') \rightarrow H^{-\alpha}(W''), \
f \mapsto \tilde{\Lambda}_{g}^{W',W''}(f),
\end{align*}
where for $h \in H_0^{\alpha}(W'')$
\begin{align*}
\langle \tilde{\Lambda}_{g}^{W',W''}(f),h  \rangle_{H^{-\alpha}(W''),H_0^{\alpha}(W'')} := -\int\limits_{\R^{n+1}_+} x_{n+1}^{1-2\alpha} \tilde{g} \nabla \tilde{u}^f \cdot \nabla \tilde{\varphi} d(x,x_{n+1}),
\end{align*}
for $\tilde{\varphi} \in X_{\alpha,0,(W'')_e\cap \Omega_e}:= \{\tilde{v}\in X_{\alpha}: \ \tilde{v}|_{(W'')_e\cap \Omega_e}=0\}$ with $\tilde{\varphi}|_{\R^n \times \{0\}} = h$ in $W''$ and $\tilde{u}^f \in X_{\alpha }$ being a weak solution of \eqref{eq:mixed} with data $f\in H_0^{\alpha}(W')$. Here we have used the notation $(W'')_e:= \R^n \setminus \overline{W''}$.
\end{defn}

We remark that the map $\tilde{\Lambda}_{g}^{W',W''}$ is well-defined and satisfies the bounds 
\begin{align*}
\|\tilde{\Lambda}_{g}^{W',W''}(f)\|_{H^{-\alpha}(W'')} \leq C \|f\|_{H^{\alpha}(\R^n)}.
\end{align*}
Indeed, for the well-definedness we observe that if $\tilde{\varphi}_1, \tilde{\varphi}_2 \in X_{\alpha,0, (W'')_e \cap \Omega_e}$ are two extensions of $h \in H^{\alpha}_0(W'')$, we have that $(\tilde{\varphi}_1 - \tilde{\varphi}_2)|_{\R^n \times \{0\}} \in \dot{H}^{\alpha}(\R^n)$ as well as $(\tilde{\varphi}_1 - \tilde{\varphi}_2)|_{\Omega_e \times \{0\}} =0 $ in a trace sense, and, thus, $\tilde{\varphi}_1 - \tilde{\varphi}_2 \in X_{\alpha,0,\Omega_e}$. In particular, by the equation satisfied by $\tilde{u}^f$, we therefore obtain that 
\begin{align*}
\int\limits_{\R^{n+1}_+} x_{n+1}^{1-2\alpha} \tilde{g} \nabla \tilde{u}^f \cdot \nabla (\tilde{\varphi}_1 - \tilde{\varphi}_2)  d(x,x_{n+1}) = 0,
\end{align*}
which proves the desired independence of the extension of the function $h$.
Similarly, the claimed boundedness of the generalized Dirichlet-to-Neumann map follows from the weak form of the equation defining $\tilde{u}^f$ together with the following estimate 
\begin{align}
\label{eq:est_Neumann}
\begin{split}
\|\tilde{\Lambda}_{g}^{W',W''}(f)\|_{H^{-\alpha}(W'')} 
&= \sup\limits_{h \in H_0^{\alpha}(W''), \|h\|_{ H^{\alpha}(\R^n)}=1}\left|\langle  \tilde{\Lambda}_{g}^{W',W''}(f), h \rangle_{ H^{-\alpha}(W''),H_0^{\alpha}(W'')} \right|\\
&\leq  \sup\limits_{h \in H_0^{\alpha}(W''), \|h\|_{ H^{\alpha}(\R^n)}=1}
\left| \int\limits_{\R^{n+1}_+} x_{n+1}^{1-2\alpha} \tilde{g} \nabla \tilde{u}^f \cdot \nabla \tilde{u}^h d(x,x_{n+1}) \right|\\
&\leq C \sup\limits_{h \in H_0^{\alpha}(W''), \|h\|_{ H^{\alpha}(\R^n)}=1}\|f\|_{H^{\alpha}(\R^n)} \|h\|_{H^{\alpha}(\R^n)}
= C \|f\|_{H^{\alpha}(\R^n)}.
\end{split}
\end{align}
Here $\tilde{u}^h$ denotes the weak solution of \eqref{eq:mixed} with data $h \in H^{\alpha}_0(W'') \subset H^{\alpha}(\R^n)$ which, in particular, satisfies $\tilde{u}^h \in X_{\alpha,0,(W'')_e\cap \Omega_e}$. As a consequence, it is admissible in the above definition and, by Lemma \ref{lem:well-posedness}, allows us to deduce the desired bounds. 

We remark that, in particular, $W'' = \Omega_e$ is admissible in Definition \ref{defi:Neumann1}.

As a second version of a generalized Dirichlet-to-Neumann map, we further define the operator $\tilde{\Lambda}_{g}^{\Omega_e}$.

\begin{defn}
\label{defi:Neumann_ext}
Let $\alpha \in (0,1)$, $n\geq 2$ and suppose that the conditions from Assumption \ref{assump:main} hold.
Then, we set
\begin{align*}
\tilde{\Lambda}_{g}^{\Omega_e}: H^{\alpha}(\Omega_e) \rightarrow H_0^{-\alpha}(\Omega_e), \
f \mapsto \tilde{\Lambda}_{g}^{\Omega_e}(f),
\end{align*}
where for $h \in H^{\alpha}(\Omega_e)$
\begin{align*}
\langle \tilde{\Lambda}_{g}^{\Omega_e}(f),h  \rangle_{H_0^{-\alpha}(\Omega_e),{H}^{\alpha}(\Omega_e)} := -\int\limits_{\R^{n+1}_+} x_{n+1}^{1-2\alpha} \tilde{g} \nabla \tilde{u}^f \cdot \nabla \tilde{\varphi} d(x,x_{n+1}),
\end{align*}
for $\tilde{\varphi} \in X_{\alpha}$ with $\tilde{\varphi}|_{\R^n \times \{0\}} = h$ in $\Omega_e$  and $\tilde{u}^f \in X_{\alpha }$ being a weak solution to \eqref{eq:mixed} with data in $H^{\alpha}(\Omega_e)$.
\end{defn}

Again the well-posedness and boundedness follow from the properties of weak solutions as in Definition \ref{defi:sol}. We remark that in contrast to Definition \ref{defi:Neumann1} in the definition for $\tilde{\Lambda}_{g}^{\Omega_e}$ we do not necessarily require that the data $f\in H^{\alpha}(\Omega_e)$ vanish in a neighbourhood of $\partial \Omega$. In particular, we allow for data $f \in C^{\infty}(\Omega_e)$ with decay at infinity but without compact support away from $\partial \Omega$.

\begin{rem}
\label{rmk:DtN}
We remark that, in view of the extension perspective from \eqref{eq:mixed}, the above defined maps $\tilde{\Lambda}_{g}^{W', W''}, \tilde{\Lambda}_{g}^{\Omega_e}$ have natural interpretations as weighted Dirichlet-to-Neumann maps. Indeed, we define the generalized weighted Neumann derivative $\lim\limits_{x_{n+1} \rightarrow 0} x_{n+1}^{1-2\alpha} \p_{n+1} \tilde{u}^f$ of a weak solution $\tilde{u}^f \in X_{\alpha}$ of \eqref{eq:mixed} with data $f\in H^{\alpha}(\R^n)$ as follows: For $h \in H^{\alpha}(\R^n)$
\begin{align*}
\langle   \lim\limits_{x_{n+1} \rightarrow 0} x_{n+1}^{1-2\alpha} \p_{n+1} \tilde{u}^f, \sqrt{|g|} h\rangle_{H^{-\alpha}(\R^n),H^{\alpha}(\R^n)} := -\int\limits_{\R^{n+1}_+} x_{n+1}^{1-2\alpha} \nabla \tilde{u}^f \cdot \tilde{g}\nabla \tilde{\varphi} d(x,x_{n+1}),
\end{align*}
where $\tilde{\varphi} \in X_{\alpha}$ is an arbitrary function with $\tilde{\varphi}|_{\R^n \times \{0\}}=h \in H^{\alpha}(\R^n)$. By the equation for $\tilde{u}^f$ and trace estimates, this is well-defined as an element in $H^{-\alpha}(\R^n)$.
Moreover, if $\tilde{u}^f$ is twice classically differentiable with $x_{n+1}^{1-2\alpha} \nabla \tilde{u}^f \in C^0(\R^{n+1}_+)$, then this definition coincides with the strong form of the weighted Neumann derivative, i.e.,
\begin{align*}
\langle  \lim\limits_{x_{n+1} \rightarrow 0} x_{n+1}^{1-2\alpha} \p_{n+1} \tilde{u}^f, \sqrt{|g|} h\rangle_{H^{-\alpha}(\R^n),H^{\alpha}(\R^n)}
= \int\limits_{\R^n} h \lim\limits_{x_{n+1} \rightarrow 0} x_{n+1}^{1-2\alpha} \p_{n+1} \tilde{u}^f dV_g(x).
\end{align*}

Now with this definition, we note that the various generalized Dirichlet-to-Neumann maps $\tilde{\Lambda}_{g}^{W', W''}$,  $\tilde{\Lambda}_{g}^{\Omega_e}$ are different realizations and restrictions of this weak form of the weighted Dirichlet-to-Neumann map with data and test functions in different spaces giving rise to different localizations.
\end{rem}

\subsubsection{The Neumann problem}
\label{sec:Neumann}

In this section, we turn to a weak version of the Neumann problem
\begin{align}
\label{eq:Neumann}
\begin{split}
\nabla \cdot x_{n+1}^{1-2\alpha} \tilde{g} \nabla \tilde{u} & = 0 \mbox{ in } \R^{n+1}_+,\\
\lim\limits_{x_{n+1} \rightarrow 0} x_{n+1}^{1-2\alpha} \p_{n+1} \tilde{u} & = F \mbox{ in } \R^n.
\end{split}
\end{align}

Also for this problem, we begin by defining a weak notion of solution.

\begin{defn}
\label{defi:Neu}
Let $\alpha \in (0,1)$, $n\geq 2$ and suppose that the conditions from Assumption \ref{assump:main} hold.
Let $\sqrt{|g|} F \in \dot{H}^{-\alpha}(\R^n)$. Then $\tilde{u}^F \in X_{\alpha}$ is a weak solution of \eqref{eq:Neumann}, if
\begin{align*}
\int\limits_{\R^{n+1}_+} x_{n+1}^{1-2\alpha} \tilde{g} \nabla \tilde{u}^F \cdot \nabla \tilde{\varphi} d(x,x_{n+1}) =
-\int\limits_{\R^n} F \tilde{\varphi}|_{\R^n \times \{0\}} \sqrt{|g|}  dx \mbox{ for all } \tilde{\varphi} \in X_{\alpha}.
\end{align*}
\end{defn}

Using a similar Lax-Milgram argument as in the previous subsection in the space $X_{\alpha}$ and noting that, by the trace estimate \eqref{eq:trace}, the map $\tilde{\varphi} \mapsto \int\limits_{\R^n} F \tilde{\varphi}|_{\R^n \times \{0\}}  \sqrt{|g|}  dx$ is a bounded functional on $X_{\alpha}$, then yields the existence and uniqueness of a weak solution.

\begin{prop}
\label{prop:Neumann}
Let $\alpha \in (0,1)$, $n\geq 2$ and suppose that the conditions from Assumption \ref{assump:main} hold.
Let $ \sqrt{|g|}  F \in \dot{H}^{-\alpha}(\R^n)$. Then there exists a unique weak solution $\tilde{u}^F \in X_{\alpha}$ of \eqref{eq:Neumann}. It satisfies the bound
\begin{align*}
\|x_{n+1}^{\frac{1-2\alpha}{2}} \nabla \tilde{u}^F\|_{L^2(\R^{n+1}_+)} + \|u^F\|_{\dot{H}^{\alpha}(\R^n)} \leq C \|\sqrt{|g|}F\|_{\dot{H}^{-\alpha}(\R^n)}.
\end{align*}
As above, $u^F:=\tilde{u}^F|_{\R^n \times \{0\}}$.
\end{prop}

\begin{proof}
We argue similarly as in the mixed Dirichlet-Neumann setting and reduce the problem to an application of the lemma of Lax-Milgram in the space $X_{\alpha}$. To this end, we first observe that by duality and the trace estimate \eqref{eq:trace}
\begin{align*}
\left| \int\limits_{\R^n} F \tilde{\varphi}|_{\R^n \times \{0\}}  \sqrt{|g|} dx\right| 
&\leq \| \sqrt{|g|}  F\|_{\dot{H}^{-\alpha}(\R^n)} \|\tilde{\varphi}|_{\R^n \times \{0\}}\|_{\dot{H}^{\alpha}(\R^n)}\\
& \leq \| \sqrt{|g|} F\|_{\dot{H}^{-\alpha}(\R^n)} \|x_{n+1}^{\frac{1-2\alpha}{2}}\nabla \tilde{\varphi}\|_{L^2(\R^{n+1}_+)},
\end{align*}
which shows that $\tilde{\varphi} \mapsto \int\limits_{\R^n} F \tilde{\varphi}|_{\R^n \times \{0\}}  \sqrt{|g|} dx$ is a bounded functional on $X_{\alpha}$. Using that $X_{\alpha}$ together with the bilinear form $(\cdot, \cdot)_{\dot{H}^1(\R^{n+1}_+,x_{n+1}^{1-2 \alpha})}$ is a Hilbert space, by the boundedness and uniform ellipticity of the metric $\tilde{g}$, existence, uniqueness and the bulk estimate for the solution follow from the lemma of Lax-Milgram. The claimed boundary estimate again follows from the trace estimate \eqref{eq:trace}.
\end{proof}

We remark that in low dimensions we only expect the solution of \eqref{eq:Neumann} to map into the homogeneous spaces when restricted to the boundary, even if the prescribed data are smooth and compactly supported. This can already be observed in the constant coefficient setting, in which the Schwartz kernel associated with the map which sends $F$ to $u^F$ has decay of the form $|\cdot|^{-n+2\alpha}$. For $L^2(\R^n)$ or $H^{\alpha}(\R^n)$ data with compact support this does not suffice to deduce $L^2(\R^n)$ integrability of the solution in general.
However, in high dimensions or in case that the data carry derivatives, it is possible to improve on this. This is a consequence of Schwartz kernel bounds for the operator sending $F $ to $u^F$. We refer to Proposition \ref{prop:Schwartz_kernel} below for this.

In order to prove Proposition \ref{prop:Schwartz_kernel} below, we begin with an auxiliary result on the tangential regularity of weak solutions of \eqref{eq:Neumann} outside of the support of $F$. To this end, we make use of Caccioppoli's inequality, which we recall for the convenience of the reader. For instance, we refer to  \cite[Proposition 2.3]{JLX14} for a proof of this. We recall that $B_r^+:=\{x \in \R^{n+1}_+: \ |x|< r\}$.

\begin{prop}[Caccioppoli inequality]
\label{prop:Cacc}
Let $\alpha \in (0,1)$, let $g \in C^{\infty}(B_{2}^+, \R^{n\times n})$ be uniformly elliptic and bounded with $\tilde{g}$ as in \eqref{eq:tildeg} and let $G = (G_1, \dots, G_{n+1}) \in L^2((B_2^+, x_{n+1}^{1-2\alpha}), \R^{n+1})$ with $\lim\limits_{x_{n+1} \rightarrow 0} x_{n+1}^{1-2\alpha} G_{n+1} =0$ in an $\dot{H}^{-\alpha}(B_2^+)$  sense. Let $\tilde{u} \in H^1(B_2^+, x_{n+1}^{1-2\alpha})$ be a weak solution of 
\begin{align*}
\nabla \cdot x_{n+1}^{1-2\alpha} \tilde{g} \nabla \tilde{u} &= \nabla \cdot x_{n+1}^{1-2\alpha} G \mbox{ in } B_2^+,\\
\lim\limits_{x_{n+1} \rightarrow 0} x_{n+1}^{1-2\alpha} \p_{n+1} \tilde{u} & = 0 \mbox{ in } B_2.
\end{align*}
Then, there exists a constant $C>0$ such that for any $r\in (0,1)$ we have that
\begin{align*}
\|x_{n+1}^{\frac{1-2\alpha}{2}} \nabla \tilde{u}\|_{L^2(B_{r}^+)} \leq C (r^{-1} \|x_{n+1}^{\frac{1-2\alpha}{2}} \tilde{u}\|_{L^2(B_{2r}^+)} + \|G\|_{L^2(B_{2r}^+, x_{n+1}^{1-2\alpha})}).
\end{align*}
\end{prop}

As a corollary, we obtain the following iterated tangential regularity estimate.

\begin{cor}[Tangentially iterated Caccioppoli]
\label{cor:Cacc}
Let $\alpha \in (0,1)$, let $g \in C^{\infty}(B_{4}^+, \R^{n\times n})$ be uniformly elliptic and bounded with $\tilde{g}$ as in \eqref{eq:tildeg} and let $\tilde{u} \in H^1(B_4^+, x_{n+1}^{1-2\alpha})$ be a weak solution of 
\begin{align*}
\nabla \cdot x_{n+1}^{1-2\alpha} \tilde{g} \nabla \tilde{u} &= 0 \mbox{ in } B_4^+,\\
\lim\limits_{x_{n+1} \rightarrow 0} x_{n+1}^{1-2\alpha} \p_{n+1} \tilde{u} & = 0 \mbox{ in } B_4.
\end{align*}
Let $j \in \{1,\dots,n\}$ and denote for $h \in \R\setminus \{0\}$ by $\Delta_{j,h}$ the tangential finite difference quotient operator, i.e., $\Delta_{j,h} \tilde{u}(x,x_{n+1}):= |h|^{-1}(\tilde{u}(x+ h e_j,x_{n+1}) - \tilde{u}(x,x_{n+1}))$.
Then, there exists a constant $C>0$ such that for any $r\in (0,1)$, $h \in (-1,1)\setminus \{0\}$ and $j \in \{1,\dots,n\}$, we have that
\begin{align*}
\|x_{n+1}^{\frac{1-2\alpha}{2}} \nabla \Delta_{j,h} \tilde{u}\|_{L^2(B_{r}^+)} \leq C r^{-1} \|x_{n+1}^{\frac{1-2\alpha}{2}} \nabla \tilde{u}\|_{L^2(B_{4r}^+)},
\end{align*}
where the constant $C>0$, in particular, depends on $\|\nabla \tilde{g}\|_{L^{\infty}(B_{4}^+)}$, $\|\tilde{g}\|_{L^{\infty}(B_{4}^+)}$.
\end{cor}

\begin{proof}
We reduce this estimate to the one from Proposition \ref{prop:Cacc}. 
Using that 
\begin{align}
\label{eq:diff_finite}
\Delta_{j,h} (f_1 f_2)(x) = (\Delta_{j,h}f_1(x)) f_2(x) + f_1(x+h e_j)\Delta_{j,h} f_2(x),
\end{align}
 the equation for $\Delta_{j,h} \tilde{u}$ reads
\begin{align*}
\nabla \cdot x_{n+1}^{1-2\alpha} \tilde{g} \nabla \Delta_{j,h} \tilde{u} (x,x_{n+1}) &= -\nabla \cdot x_{n+1}^{1-2\alpha} (\Delta_{j,h} \tilde{g})(x) \nabla \tilde{u}(x+he_j,x_{n+1}) \mbox{ in } B_{3}^+,\\
\lim\limits_{x_{n+1} \rightarrow 0} x_{n+1}^{1-2\alpha} \p_{n+1} \Delta_{j,h}  \tilde{u} & = 0 \mbox{ on } B_{3}.
\end{align*}
This implies that, by Proposition \ref{prop:Cacc}, there exists a constant $C_0>0$ such that for all $r \in (0,1)$
\begin{align*}
\|x_{n+1}^{\frac{1-2\alpha}{2}} \nabla \Delta_{j,h} \tilde{u}\|_{L^2(B_{r}^+)} \leq C_0( r^{-1} \|x_{n+1}^{\frac{1-2\alpha}{2}} \Delta_{j,h} \tilde{u}\|_{L^2(B_{2r}^+)} + \|(\Delta_{j,h} \tilde{g}) x_{n+1}^{\frac{1-2\alpha}{2}} \nabla  \tilde{u}(\cdot + h e_j)\|_{L^2(B_{2r}^+)}).
\end{align*}
Estimating the second contribution further by using the differentiability of $\tilde{g}$, we obtain that 
\begin{align*}
\|x_{n+1}^{\frac{1-2\alpha}{2}} \nabla \Delta_{j,h} \tilde{u}\|_{L^2(B_{r}^+)}& \leq C_0(1 + 2\sup\limits_{x\in B_{4r}^+}(|\nabla \tilde{g}(x)| + |\tilde{g}(x)|))( r^{-1} \|x_{n+1}^{\frac{1-2\alpha}{2}} \nabla \tilde{u}\|_{L^2(B_{4r}^+)}),
\end{align*}
where we used the bound 
\begin{align*}
\|x_{n+1}^{\frac{1-2\alpha}{2}} \Delta_{j,h} \tilde{u}\|_{L^2(B_{2r}^+)} \leq  \|x_{n+1}^{\frac{1-2\alpha}{2}} \nabla \tilde{u}\|_{L^2(B_{4r}^+)}.
\end{align*}
The latter follows for $C_0^{\infty}(\overline{\R^{n+1}_+})$ functions by the fundamental theorem of calculus and then transfers to $H^{1}(B_{4r}^+, x_{n+1}^{1-2\alpha})$ by density (see, for instance, \cite[Propositions 2.9, 2.10]{K14}).
Setting $C:=2 C_0(1 + 2\sup\limits_{x\in B_{4}^+}(|\nabla \tilde{g}(x)| + |\tilde{g}(x)|))<\infty$ then implies the claim.
\end{proof}

\begin{rem}
\label{rmk:radii}
We remark that the results of Proposition \ref{prop:Cacc} and Corollary \ref{cor:Cacc} remain valid with modified ratios between the radii on the left and right hand sides of the estimates as long as the right hand side radius is larger than the left one. The corresponding constants then additionally depend only on the ratio between the radii.
\end{rem}

With these auxiliary results in hand, we obtain the following tangential regularity result.

\begin{lem}[Tangential regularity]
\label{lem:tangential_cont}
Let $u^F:= \tilde{u}^F|_{\R^n \times \{0\}} \in \dot{H}^{\alpha}(\R^n)$ denote the restriction of the weak solution $\tilde{u}^F \in X_{\alpha}$ of \eqref{eq:Neumann} with data $\sqrt{|g|} F\in \dot{H}^{-\alpha}(\R^n)$.
Let $B_{r}(x_0) \subset \R^n$ be an open ball such that $B_{2r}(x_0)\cap \supp(F) = \emptyset$. Then, for any $k\geq 0$ it holds that $u^F|_{B_r(x_0)} \in C^{k}(B_r(x_0))$ and
\begin{align}
\label{eq:uniform_reg}
\|u^F\|_{C^k(B_r(x_0))} \leq C \max\{1,\dist(x_0,\supp(F))^{-\frac{n}{2}+ \alpha -k}\} \|\sqrt{|g|}  F\|_{\dot{H}^{-\alpha}(\R^n)}.
\end{align}
\end{lem}

\begin{proof}
The result follows from standard tangential finite difference quotient arguments iterated together with Caccioppoli, Sobolev and trace estimates.

First, without loss of generality, we observe that by translation and rescaling we may assume that $x_0=0$ and that $r= \dist(x_0,\supp(F))/2 = 1$. Indeed, for $\ell:= \dist(x_0, \supp(F))$ this can always be achieved by defining $F_{\ell,x_0}(x):= \ell^{2 \alpha}F(x_0+\ell (x-x_0))$, $u^{F_{\ell,x_0}}_{\ell,x_0}(x):= u^{F}(x_0+\ell (x-x_0))$ and defining $\tilde{u}^F_{\ell,x_0}(x,x_{n+1}):=\tilde{u}^{F}(x_0+\ell (x-x_0), \ell x_{n+1})$ as the corresponding extension associated with the metric $g_{\ell,x_0}(x):=g(x_0+\ell (x-x_0))$. 

Secondly, by virtue of the energy estimate from Proposition \ref{prop:Neumann}, Hölder's inequality and Sobolev embedding, we obtain the following localized $L^2$ estimate
\begin{align}
\label{eq:L2_est}
\|u^F\|_{L^{2}(B_{3/2}(x_0))} \leq C \|u^F\|_{L^{\frac{2n}{n-2\alpha}}(B_{3/2}(x_0))} \leq C\|u^F\|_{L^{\frac{2n}{n-2\alpha}}(\R^n)} 
\leq C \|\sqrt{|g|} F\|_{\dot{H}^{-\alpha}(\R^n)} .
\end{align}

Thirdly, we note a localized version of the embedding $\dot{H}^{1}(\R^{n+1}_+, x_{n+1}^{1-2\alpha}) \rightarrow \dot{H}^{\alpha}(\R^n)$: for a general function $\tilde{u} \in H^{1}(B_{3/2}^+(x_0), x_{n+1}^{1-2\alpha})$ and $u:= \tilde{u}|_{B_{3/2}(x_0)\times \{0\}}$ it holds that
\begin{align}
\label{eq:trace_loc}
\|u\|_{\dot{H}^{\alpha}(B_{1/2}(x_0))} \leq C( \|x_{n+1}^{\frac{1-2\alpha}{2}} \nabla \tilde{u}\|_{L^2(B_{3/2}^+(x_0))} +
\|x_{n+1}^{\frac{1-2\alpha}{2}}  \tilde{u}\|_{L^2(B_{3/2}^+(x_0))}) .
\end{align}
Indeed, this follows by applying the embedding  $H^{1}(\R^{n+1}_+, x_{n+1}^{1-2\alpha}) \rightarrow \dot{H}^{\alpha}(\R^n)$ to the function $u \eta$, where $\eta \in C_0^{\infty}(\overline{B_{1}^+(x_0)})$, expanding the derivatives on the right hand side and invoking the compact support of $\eta$.

Then, using the trace estimate \eqref{eq:trace_loc} together with Caccioppoli's inequality in the version of Corollary \ref{cor:Cacc} (with a modified ratio of radii compared to Corollary \ref{cor:Cacc}), we observe that
for any $j \in \{1,\dots,n\}$ and $|h| \leq \frac{1}{2}$, $h \neq 0$,
\begin{align}
\label{eq:tangential_improved1}
\begin{split}
\|\Delta_{j,h} u^F\|_{\dot{H}^{\alpha}(B_{1/2}(x_0))}
&\leq C (\|x_{n+1}^{\frac{1-2\alpha}{2}} \nabla \Delta_{j,h} \tilde{u}^F\|_{L^2(B_{3/2}^+(x_0))} + \|x_{n+1}^{\frac{1-2\alpha}{2}}  \Delta_{j,h} \tilde{u}^F\|_{L^2(B_{3/2}^+(x_0))})\\
&\leq 
C  (\|x_{n+1}^{\frac{1-2\alpha}{2}} \nabla \tilde{u}^F\|_{L^2(B_{3}^+(x_0))}+ \|x_{n+1}^{\frac{1-2\alpha}{2}}  \Delta_{j,h} \tilde{u}^F\|_{L^2(B_{3/2}^+(x_0))})\\
&\leq \tilde{C} \|\sqrt{|g|}  F\|_{\dot{H}^{-\alpha}(\R^n)} .
\end{split}
\end{align}
In the last line, we used the a priori estimate from Proposition \ref{prop:Neumann} together with the bound $ \|x_{n+1}^{\frac{1-2\alpha}{2}}  \Delta_{j,h} \tilde{u}^F\|_{L^2(B_{3/2}^+(x_0))} \leq C\|x_{n+1}^{\frac{1-2\alpha}{2}} \nabla \tilde{u}^F\|_{L^2(B_{3}^+(x_0))}$ which follows from the fundamental theorem.

Now, using that since $h \in (0,1)$ we have that for all $\beta \in (0,1]$,
\begin{align*}
\sup\limits_{h \in (0,1)} h^{-\beta} \sum\limits_{j=1}^n \|\tau_{j,h} u^F \|_{\dot{H}^{\alpha}(B_{1/2}(x_0))} \leq \tilde{C} \|\sqrt{|g|}  F\|_{\dot{H}^{-\alpha}(\R^n)}
\end{align*}
we apply Lemma \ref{lem:Sob_charact} with $\beta=1-\alpha +\epsilon$, $\mu=\alpha$ and where $\epsilon \in (0,\beta)$ is chosen such that $1-\alpha + \epsilon \in (0,1)$. By passing to the limit $h \rightarrow 0$, we thus obtain that 
\begin{align*}
\| u^F\|_{\dot{H}^{\alpha+(1-\alpha + \epsilon)}(B_{1/2}(x_0))}
&\leq \tilde{C} \|\sqrt{|g|} F\|_{\dot{H}^{-\alpha}(\R^n)} + \|u^F\|_{\dot{H}^{\alpha}(\R^n)} 
\leq \tilde{C} \|\sqrt{|g|} F\|_{\dot{H}^{-\alpha}(\R^n)}  .
\end{align*}
In the last estimate we have used the energy estimates from Proposition \ref{prop:Neumann}.
Together with \eqref{eq:L2_est}, we infer that 
\begin{align*}
\| u^F\|_{{H}^{1+\epsilon}(B_{1/2}(x_0))}
&\leq \tilde{C} \|\sqrt{|g|} F\|_{\dot{H}^{-\alpha}(\R^n)}  .
\end{align*}

 We iterate these bounds using similar arguments as in the proof of Corollary \ref{cor:Cacc} by iteratively considering higher tangential derivatives of $\tilde{u}^F$ up to the point that the Sobolev exponent is reached. 
Indeed, after sufficient iteration of the above arguments, this yields that  
\begin{align*}
\| u^F\|_{{H}^{\beta}(B_{1/2}(x_0))}
&\leq \tilde{C} \|\sqrt{|g|} F\|_{\dot{H}^{-\alpha}(\R^n)} .
\end{align*}
for some $\beta > \frac{n}{2}$. 
Invoking Sobolev embedding first implies a $C^0(B_{1/2}(x_0))$ bound for $u^F$. Further iteration then also leads to the claimed $C^k(B_{1/2})$ estimates. 
 
 Finally, the decay from the estimate \eqref{eq:uniform_reg} follows by rescaling back to the original function.
 Indeed, for $F_{\ell, x_0}$, ${u}^F_{\ell,x_0}$, $g_{\ell,x_0}$ and for $B_1(0)$ we then infer the estimate
 \begin{align*}
 \|{u}^F_{\ell,x_0}\|_{C^k(B_1(0))} \leq C  \|\sqrt{|g|}  F_{\ell,x_0}\|_{\dot{H}^{-\alpha}(\R^n)}.
 \end{align*}
 Rescaling back, implies that 
  \begin{align*}
\min\{1,\ell^{k}\} \|{u}^F\|_{C^k(B_{\ell}(x_0))} \leq C \ell^{-\frac{n}{2}-\alpha} \ell^{2\alpha} \|\sqrt{|g|}  F\|_{\dot{H}^{-\alpha}(\R^n)}.
 \end{align*}
 Dividing by $\min\{1,\ell^{k}\} $ and carrying out a case distinction depending on the value of the minimum then implies the claim.
\end{proof}

With the auxiliary result of Lemma \ref{lem:well-posedness} in hand, we can turn to the desired decay estimates for our solution of \eqref{eq:Neumann}. To this end, we first recall that $L^{\frac{2n}{n+2\alpha}}(\R^n) \subset \dot{H}^{-\alpha}(\R^n)$ which follows by duality from the inclusion $\dot{H}^{\alpha}(\R^n) \subset L^{\frac{2n}{n-2\alpha}}(\R^n)$. Moreover, we recall that we only consider solutions of \eqref{eq:Neumann} with data which are not in the full space $H^{-\alpha}(\R^n)$ but only those in the smaller, homogeneous space $\dot{H}^{-\alpha}(\R^n) \subset H^{-\alpha}(\R^n)$. We note that for data of the type $H = (-\D_g) F$ with $F\in C_0^{\infty}(\R^n)$ we have that $\sqrt{|g|}H \in \dot{H}^{-\alpha}(\R^n)$. Indeed,
\begin{align}
\label{eq:neg_norm_bound}
\begin{split}
\|\sqrt{|g|}H\|_{\dot{H}^{-\alpha}(\R^n)}
&= \sup\limits_{\|G\|_{\dot{H}^{\alpha}(\R^n)}=1} |\langle \sqrt{|g|} H, G \rangle_{\dot{H}^{-\alpha}(\R^n),\dot{H}^{\alpha}(\R^n)}|\\
&= \sup\limits_{\|G\|_{\dot{H}^{\alpha}(\R^n)}=1}| \langle \sqrt{|g|}(-\D_{g})F, G \rangle_{\dot{H}^{-\alpha}(\R^n),\dot{H}^{\alpha}(\R^n)}|\\
&\leq \sup\limits_{\|G\|_{\dot{H}^{\alpha}(\R^n)}=1} \|\sqrt{|g|}(-\D_{g})F \|_{\dot{H}^{-\alpha}(\R^n)} \| G \|_{\dot{H}^{\alpha}(\R^n)}\\
&\leq  \|\sqrt{|g|}(-\D_{g})F \|_{\dot{H}^{-\alpha}(\R^n)} 
\leq \| \sqrt{|g|} g^{-1} \nabla_x F\|_{\dot{H}^{1-\alpha}(\R^n)} \\
& \leq \| \sqrt{|g|} g^{-1} \nabla_x F\|_{{H}^{1-\alpha}(\R^n)} \leq C_F < \infty.
\end{split}
\end{align}
In the last estimates we used the regularity and support condition for $F$. In particular, such data $H$ are thus admissible in Definition \ref{defi:Neu}. In what follows below, we will heavily exploit this.

\begin{prop}
\label{prop:Schwartz_kernel}
Let $\alpha \in (0,1)$ and let $H = (-\D_g) F$ with $F\in C_0^{\infty}(\R^n)$.  Let $\tilde{u}^H \in X_{\alpha}$ denote the  weak solution of \eqref{eq:Neumann} with data $\sqrt{|g|} H \in \dot{H}^{-\alpha}(\R^n)$ and let $u^H:= \tilde{u}^H|_{\R^n \times \{0\}} \in \dot{H}^{\alpha}(\R^n)$.
Then we have that $u^H\in H^{\alpha}(\R^n)$ for all $n\geq 2$.
\end{prop}

\begin{proof}
We argue in several steps.\\

\emph{Step 1: Existence of a Schwartz kernel.}
We first consider the map $L^{\frac{2n}{n+2\alpha}}(\R^n) \ni G \mapsto u^G \in L^{\frac{2n}{n-2\alpha}}(\R^n)$ where $u^G$ is the restriction of a solution to \eqref{eq:Neumann} with data $G$. Then, since $L^{\frac{2n}{n+2\alpha}}(\R^n)  \subset \dot{H}^{-\alpha}(\R^n)$, by Proposition \ref{prop:Neumann}, it follows that this map is bounded. Hence, by the Schwartz kernel theorem, there exists a distributional kernel $k: \R^n \times \R^n \rightarrow \R$ such that
\begin{align*}
u^G(x) = \int\limits_{\R^n} k(x,y) G(y )dy .
\end{align*}
We next aim at obtaining further bounds on $k$.\\

\emph{Step 2: Schwartz kernel bounds.}
If 
$\dist(\supp(G), x_0) = 2$, by the auxiliary Lemma \ref{lem:tangential_cont}, we have that for any $\ell\geq 0$
\begin{align}
\label{eq:reg_tang}
\|u^G\|_{C^{\ell}(B_1(x_0))} \leq C_{\ell}  \|\sqrt{|g|} G\|_{\dot{H}^{-\alpha}(\R^n)} \leq C_{\ell}  \|\sqrt{|g|} G\|_{L^{\frac{2n}{n+2\alpha}}(\R^n)}.
\end{align}
Varying the function $G$, we thus obtain that, as a function of $x$ (with $x\neq y$), the kernel $k(\cdot, y)$ is a smooth function. By the symmetry of the equation and thus the kernel, this is also true as a function in $y$ (with $y\neq x$). 

We next turn to the specific form of our data $H = (-\D_g) F$. 
In order to obtain the desired result, we observe that $u^H = (-\D_{g}) w^F$, where $w^F$ is the restriction of the weak solution of the equation \eqref{eq:Neumann} with right hand side $F$. Indeed, we note that the problem \eqref{eq:Neumann} with data $F$ is solvable with $\tilde{w}^F \in X_{\alpha}$. By the regularity of $F$,  similarly as in the proof of Lemma \ref{lem:tangential_cont}, difference quotient arguments show that as a function of the tangential directions, $\tilde{w}^F$ is smooth. We can hence commute the equation for $\tilde{w}^F$ and the Laplace-Beltrami operator and obtain a weak solution $\tilde{u}^H = (-\D_g) \tilde{w}^F \in X_{\alpha}$ of \eqref{eq:Neumann} with data $H$, if $H$ is of the given form.
Thus, the Schwartz kernel associated with the map $F\mapsto u^H$ is given by $\tilde{k}(x,y):= (-\D_{g,x}) k(x,y)$ where $k(x,y)$ is the kernel from Step 1 above. Next, we deduce the desired decay estimates for the kernel from this.

To this end, we rely on a regularity and scaling argument. We claim that the kernel $\tilde{k}(\cdot, \cdot)$ satisfies the following decay bounds:
for $x_0, w \in \R^n$ with $|x_0-w|=3r$ we have for $x \in B_r(x_0)$
\begin{align}
\label{eq:decay_improve}
\|\tilde{k}(x,\cdot)\|_{L^{\frac{2n}{n-2\alpha}}(B_{r}(w))} \leq C r^{-\frac{n}{2}+\alpha-2}.
\end{align}
Indeed, we first  consider $x_0, w \in \R^n$ such that $|x_0 -w|=3$. Using \eqref{eq:reg_tang}, this then yields for $x\in B_1(x_0)$
\begin{align*}
&\|(-\D_{g,x})k(x,\cdot)\|_{L^{\frac{2n}{n-2\alpha}}(B_1(w))} \\
&= \sup\limits_{F \in L^{\frac{2n}{n+2\alpha}}(B_1(w)),\ 
\|F\|_{L^{\frac{2n}{n+2\alpha}}(B_1(w))}=1} \left| \int\limits_{\R^n} (-\D_{g,x}) k(x,y) F(y) dV_g(y) \right|\\
&= \sup\limits_{F \in L^{\frac{2n}{n+2\alpha}}(B_1(w)),\ 
\|F\|_{L^{\frac{2n}{n+2\alpha}}(B_1(w))}=1} |(-\D_g)u^{F}(x)| \\
& \leq C  \sup\limits_{F \in L^{\frac{2n}{n+2\alpha}}(B_1(w)),\ 
\|F\|_{L^{\frac{2n}{n+2\alpha}}(B_1(w))}=1}\|F \sqrt{|g|}\|_{\dot{H}^{-\alpha}(\R^n)} \\
&\leq C  \sup\limits_{F \in L^{\frac{2n}{n+2\alpha}}(B_1(w)),\ 
\|F\|_{L^{\frac{2n}{n+2\alpha}}(B_1(w))}=1}\|F\|_{L^{\frac{2n}{n+2\alpha}}(\R^n)} \leq C.
\end{align*}
Here we have identified an $ L^{\frac{2n}{n+2\alpha}}(B_1(w))$ function with an $L^{\frac{2n}{n+2\alpha}}(\R^n)$ function after an extension by zero. In the bound, we first invoked \eqref{eq:reg_tang} (to pass from the third to the fourth line), then the inclusion $L^{\frac{2n}{n+2 \alpha}}(B_1(w)) \subset \dot{H}^{-\alpha}(\R^n)$ (to pass from the fourth to the fifth line) and finally the fact that $\tilde{g}$ is bounded in the last estimate. 
As in the previous lemma, by a scaling argument, we hence obtain the bound \eqref{eq:decay_improve}. 

Moreover, this also allows us to  deduce that for all $R>1$ sufficiently large
\begin{align}
\label{eq:decay_hs_improve}
\begin{split}
& \| (-\D_{g,x}) k(\cdot,\cdot) \|_{L^2(\R^n \setminus B_{2R}) \times L^{\frac{2n}{n-2\alpha}}(B_R)}^2\\
 &\leq \sum\limits_{j\in \N} \int\limits_{B_{2^{j+1}R}\setminus B_{2^jR}}  \|(-\D_{g,x}) k(x,\cdot)\|_{L^{\frac{2n}{n-2\alpha}}(B_{R})}^2 dx\\
& \leq \sum\limits_{j\in \N} \int\limits_{B_{2^{j+1}R}\setminus B_{2^jR}}  C (2^{j}R)^{-n+2\alpha-4} dx\\
 &\leq C \sum\limits_{j\in \N} 2^{j n}R^n  (2^{j}R)^{-n+2\alpha-4} = C R^{2\alpha-4}.
 \end{split}
 \end{align}

\emph{Step 3: Conclusion.}
Concluding the proof, we invoke the compact support condition together with the estimate from Step 2. 
Indeed, this yields that for $x \in \R^n$ such that $|x|\geq 2R$ where $R>1$ is such that $\supp(F)\subset B_{R}$ with $B_{R}:=\{x \in \R^n: \ |x|<R\}$, we have that
\begin{align*}
\|u^H \|_{L^2(\R^n\setminus B_{2R})}
&\leq \sup\limits_{\|G\|_{L^2(\R^n\setminus B_{2R})}=1} \left|\int\limits_{\R^n \times \R^n} G(x) \tilde{k}(x,y) F(y) dy dx \right|\\
&\leq \sup\limits_{\|G\|_{L^2(\R^n\setminus B_{2R})}=1} \|F\|_{L^{\frac{2n}{n+2\alpha}}(B_R)} \int\limits_{\R^n } |G(x)| \| \tilde{k}(x,\cdot) \|_{L^{\frac{2n}{n-2\alpha}}(B_R)} dx\\
& \leq \sup\limits_{\|G\|_{L^2(\R^n\setminus B_{2R})}=1} \|F\|_{L^{\frac{2n}{n+2 \alpha}}(B_R)}  \|G\|_{L^2(\R^n \setminus B_{2R})} \| \tilde{k}(\cdot,\cdot) \|_{L^2(\R^n \setminus B_{2R}) \times L^{\frac{2n}{n-2\alpha}}(B_R)} \\
& = \|F\|_{L^{\frac{2n}{n+2 \alpha}}(B_R)}  \| \tilde{k}(\cdot,\cdot) \|_{L^2(\R^n \setminus B_{2R}) \times L^{\frac{2n}{n-2\alpha}}(B_R)} \\
& \leq C\|F\|_{L^{\frac{2n}{n+2 \alpha}}(B_R)} R^{\alpha-2}.
\end{align*}
Here we used the estimate \eqref{eq:decay_hs_improve} for the estimate for the kernel and viewed the 
$L^2(\R^n\setminus B_{2R})$ function as an $L^2(\R^n)$ function after an extension by zero. This yields the $L^2(\R^n\setminus B_{2R})$ integrability of $u^H$.

Moreover,
in the ball $B_{4R}$ the function $u^H$ is also $L^2$ bounded by the $\dot{H}^{\alpha}$ bound from Proposition \ref{prop:Neumann}, an application of the Sobolev trace estimate \eqref{eq:Sob_trace} and Hölder's inequality. Indeed, by the Sobolev trace inequality, $\|u^H\|_{L^{\frac{2n}{n-2\alpha}}(B_{4R})}$ is bounded in terms of $\|\sqrt{|g|} F\|_{\dot{H}^{-\alpha}(\R^n)}$ and then, by Hölder's inequality, it is also bounded in $L^2(B_{4R})$ in terms of $\|\sqrt{|g|} F\|_{\dot{H}^{-\alpha}(\R^n)}$.
\end{proof}

\subsubsection{Relation between weak solutions of \eqref{eq:mixed} and energy solutions from Section \ref{sec:energy_sol}}
\label{sec:equiv}

Last but not least, we relate the above weak solutions from Definition \ref{defi:sol} with data in $H^{\alpha}(\Omega_e)$ to energy solutions based on a nonlocal spectral definition through a suitable bilinear form.
To this end, we recall the bilinear form $\mathcal{E}$ from Section \ref{sec_well_posedness_dirichlet} (see, for instance, the discussion around Proposition \ref{prop_well-posedness}): for $u,v \in C_0^{\infty}(\R^n)$ we have $\mathcal{E}(u,v):= (u,(-\D_g)^{\alpha} v)_{L^2(\R^n,d V_g)}$ and an energy solution $u$ of 
\begin{align}
\label{eq:spect}
\begin{split}
(-\D_g)^{\alpha} u &= 0 \mbox{ in } \Omega,\\
u & = f \mbox{ in } \Omega_e,
\end{split}
\end{align}
is given by an $H^{\alpha}(\R^n)$ function $u$ with $u-f \in H_0^{\alpha}(\Omega)$ and such that $\mathcal{E}(u,\varphi)=0$ for all $\varphi \in C_0^{\infty}(\Omega)$. Here we have used that the bilinear form $\mathcal{E}$ extends to $H^{\alpha}(\R^n)\times H^{\alpha}(\R^n)$ (see Proposition \ref{prop_eq_2_6} from above). The definition of $(-\D_g)^{\alpha}$ here is through functional calculus. In the following argument we take for granted that for $u \in H^{\alpha}(\R^n)$ it holds
\begin{align}
\label{eq:conv_CS}
\lim\limits_{x_{n+1} \rightarrow 0}\|(-\D_g)^{\alpha} u -  c_{\alpha} x_{n+1}^{1-2\alpha} \p_{n+1} \tilde{u}\|_{{H}^{-\alpha}(\R^n)} = 0,
\end{align}
where $c_{\alpha} \neq 0$ is a constant and $\tilde{u}$ denotes the extension of $u$, i.e., $\tilde{u} \in X_{\alpha}$ is the unique weak solution to
\begin{align*}
\nabla \cdot x_{n+1}^{1-2\alpha} \tilde{g} \nabla \tilde{u} & = 0 \mbox{ in } \R^{n+1}_+,\\
\tilde{u} & = u \mbox{ on } \R^n.
\end{align*}
We remark that existence and uniqueness follows as in the proof of Lemma \ref{lem:well-posedness} (where the space $X_{\alpha,0,\Omega_e}$ is replaced by the corresponding space $X_{\alpha,0,\R^n}$ in which the zero Dirichlet condition holds on the whole tangential domain) or a direct minimization argument of the associated Dirichlet-type energy.

\begin{remark}[On the convergence \eqref{eq:conv_CS}]
For $\tilde{g}= id$, the estimate \eqref{eq:conv_CS} follows from the seminal work by Caffarelli-Silvestre \cite{CS07}.  For a variable coefficient metric this follows from \cite{ST10} and \cite{ATW18}.
Indeed, in \cite{ST10} an $L^2(\R^n)$ convergence of the type \eqref{eq:conv_CS} is proved for $u \in \text{Dom}((-\D_g)^{\alpha})$. As $C_0^{\infty}(\R^n) \subset \text{Dom}((-\D_g)^{\alpha})$ and as $C_0^{\infty}(\R^n)$ is dense in $H^{\alpha}(\R^n)$, we then also obtain \eqref{eq:conv_CS} by continuity. Indeed, by the $H^{-\alpha}(\R^n)$ continuity estimate from \eqref{eq:est_Neumann}, the continuity of $(-\D_g)^{\alpha}: H^{\alpha}(\R^n) \rightarrow H^{-\alpha}(\R^n)$ and by the relation between the operators $\tilde{\Lambda}_{g}^{W',W''}$ and $\lim\limits_{x_{n+1} \rightarrow 0} x_{n+1}^{1-2\alpha} \p_{n+1} \tilde{u}(x,x_{n+1})$ from Remark \ref{rmk:DtN}, the statement of \eqref{eq:conv_CS} then follows by an approximation argument for $u \in H^{\alpha}(\R^n)$ by $\{u_k\}_{k \in \N} \subset C_0^{\infty}(\R^n)$.
\end{remark}

Having fixed this, we claim that our mixed Dirichlet-Neumann perspective from \eqref{eq:mixed} indeed captures the inverse problem for the fractional Laplacian.

\begin{prop}
\label{prop:equiv}
Let $\alpha \in (0,1)$, let $g$ satisfy the conditions from Assumption \ref{assump:main} and let $f\in H^{\alpha}(\R^n)$. Let $u \in H^{\alpha}(\R^n)$ be an energy solution of \eqref{eq:spect} and let $\tilde{u}^f \in X_{\alpha}$ be a weak solution of \eqref{eq:mixed}. Then, for $u^f:=\tilde{u}^f|_{\R^n \times \{0\}}$ it holds that $ u^f= u$ a.e. in $\R^n$.
\end{prop}

\begin{proof}
\emph{Step 1: The restriction of a solution of \eqref{eq:mixed} is an energy solution of \eqref{eq:spect}.}
Let $\tilde{u}^f \in X_{\alpha}$ be the solution to \eqref{eq:mixed} provided by Lemma \ref{lem:well-posedness}. Then, by assumption, we have that
\begin{align*}
\int\limits_{\R^{n+1}_+} x_{n+1}^{1-2\alpha} \tilde{g} \nabla \tilde{u}^f \cdot \nabla \tilde{v} d(x,x_{n+1}) = 0,
\end{align*}
for all $\tilde{v} \in X_{\alpha,0,\Omega_e}$. We note that for $\varphi \in C_0^{\infty}(\Omega)$ and for $\tilde{\varphi}$ being the variable coefficient Caffarelli-Silvestre-type extension of $\varphi$, i.e., $\tilde{\varphi}$ is assumed to be a weak solution of
\begin{align*}
\nabla \cdot \tilde{g} x_{n+1}^{1-2\alpha} \nabla \tilde{\varphi} & = 0 \mbox{ on } \R^{n+1}_+,\\
\tilde{\varphi} & = \varphi \mbox{ on } \R^n \times \{0\},
\end{align*}
by the proof of Lemma \ref{lem:well-posedness}, we have $\tilde{\varphi} \in X_{\alpha,0,\Omega_e}$. Hence, by the equation for $\tilde{u}^f$, the fact that $u^f \in H^{\alpha}(\R^n)$ and the definition of the generalized Neumann derivative (see Remark \ref{rmk:DtN})
\begin{align*}
0 &= \int\limits_{\R^{n+1}_+} x_{n+1}^{1-2\alpha} \nabla \tilde{u}^f \cdot \tilde{g} \nabla \tilde{\varphi} d(x,x_{n+1}) 
 =  -\langle \sqrt{|g|} \lim\limits_{x_{n+1} \rightarrow 0} x_{n+1}^{1-2\alpha} \p_{n+1} \tilde{\varphi}, u^f \rangle_{{H}^{-\alpha}(\R^n),{H}^{\alpha}(\R^n)} \\
 & 
 = -\frac{1}{c_{\alpha}} \langle \sqrt{|g|} (-\D_g)^{\alpha} \varphi, u^f \rangle_{{H}^{-\alpha}(\R^n), {H}^{\alpha}(\R^n)} 
 = -\frac{1}{c_{\alpha}}\mathcal{E}(u^f, \varphi ).
\end{align*}
Again, we here used that $\mathcal{E}$ extends to and is well-defined on $H^{\alpha}(\R^n) \times H^{\alpha}(\R^n)$.
Moreover, by definition and the domain regularity, we also have that $u^f-f \in H_0^{\alpha}(\Omega)$ which concludes the proof of the claim.

\emph{Step 2: An extension of an energy solution to \eqref{eq:spect} is a weak solution to \eqref{eq:mixed}.}
Conversely, we assume that $u \in H^{\alpha}(\R^n)$ is a solution to \eqref{eq:spect}. We define $\tilde{u} \in X_{\alpha}$ to be the extension of $u$. Moreover, let $\tilde{\varphi} \in X_{\alpha,0,\Omega_e}$. Then, $\varphi:= \tilde{\varphi}|_{\R^n} \in \dot{H}^{\alpha}(\R^n)$ and, by Poincar\'e's inequality \eqref{eq:Poincare_nonlocal}, also $\varphi \in H^{\alpha}(\R^n)$. Due to the trace condition for $\tilde{\varphi}$, we even have $\varphi \in H_0^{\alpha}(\Omega)$. In particular, there exist $\varphi_k \in C_0^{\infty}(\Omega)$ such that $\varphi_k \rightarrow \varphi$ in $H^{\alpha}(\R^n)$. Let $\tilde{\varphi}_k \in X_{\alpha}$ denote the Caffarelli-Silvestre-type extension of $\varphi_k$. Then, by construction, $\tilde{\varphi}_k \in X_{\alpha,0,\Omega_e}$. In particular, since $u$ is a solution to \eqref{eq:spect}, similarly as above,
\begin{align*}
\int\limits_{\R^{n+1}_+} x_{n+1}^{1-2\alpha} \nabla \tilde{u} \cdot \tilde{g}\nabla \tilde{\varphi}_k d(x,x_{n+1}) = -\frac{1}{c_{\alpha}}\mathcal{E}(u,\varphi_k) = 0.
\end{align*}
We seek to replace $\tilde{\varphi}_k$ by the function $\tilde{\varphi}$. To this end, by the definition of the generalized normal derivative and since $\varphi_k \rightarrow \varphi$ in $H^{\alpha}(\R^n)$, we observe that
\begin{align*}
&\left|\int\limits_{\R^{n+1}_+} x_{n+1}^{1-2\alpha} \nabla \tilde{u} \cdot \tilde{g}\nabla (\tilde{\varphi}-\tilde{\varphi}_k) d(x,x_{n+1}) \right| \\
& = \left|  \langle \sqrt{|g|} \lim\limits_{x_{n+1} \rightarrow 0} x_{n+1}^{1-2\alpha} \p_{n+1} \tilde{u},  \varphi - \varphi_k  \rangle_{H^{-\alpha}(\R^n), H^{\alpha}(\R^n)} \right|\\
& \leq \|\lim\limits_{x_{n+1} \rightarrow 0} x_{n+1}^{1-2\alpha} \p_{n+1} \tilde{u}\|_{H^{-\alpha}(\R^n)} \|\sqrt{|g|}(\varphi_k - \varphi)\|_{H^{\alpha}(\R^n)} \rightarrow 0,
\end{align*}
as $k \rightarrow \infty$.
As a consequence, we also obtain that 
\begin{align*}
\int\limits_{\R^{n+1}_+} x_{n+1}^{1-2\alpha} \nabla \tilde{u} \cdot \tilde{g}\nabla \tilde{\varphi}d(x,x_{n+1}) = 0 \mbox{ for all } \tilde{\varphi} \in X_{\alpha,0,\Omega_e},
\end{align*}
which concludes the proof of the second implication.
\end{proof}

\subsection{Proof of Theorem \ref{thm_main}. Uniqueness for the inverse problem}

In this section, we provide the proof of the uniqueness result from Theorem \ref{thm_main} based on the elliptic extension perspective from \eqref{eq:mixed}. To this end, we argue in several steps, successively reducing the measurement data to a source-to-solution problem.

\subsubsection{Step 1: From $\tilde{\Lambda}_{g}^{W_1,W_2}$ to $\tilde{\Lambda}_{g}^{W_1,\Omega_e}$ and $\tilde{\Lambda}_{g}^{ \Omega_e}$}

It is the purpose of this section to extend the available measurement data from the measurements in $ W_2$ to measurements in the whole complement $\Omega_e$ and also the available data from localized ones on the set $W_1$ to data on $\Omega_e$. 
We begin by showing that the generalized Dirichlet-to-Neumann map with data in $H_0^{\alpha}(W_1)$ and measured on $W_2$ determines the fractional Laplacian in the whole complement of $\Omega$ by unique continuation. For completeness, we recall the unique continuation property on the level of the extension problem which takes the form of a boundary unique continuation property \cite{FF14, Yu16,R15,Ghosh_Salo_Uhlmann_2020}.

\begin{prop}[UCP for the extension problem]
\label{prop:UCP}
Let $\Omega \subset \R^n$ satisfy the condition in Assumption \ref{assump:main}(a). 
Let $\alpha \in (0,1)$, let $\tilde{g}$ satisfy the conditions in Assumption \ref{assump:main}(b) and let $\tilde{u} \in H^1_{loc}(\R^{n+1}_+, x_{n+1}^{1-2\alpha})$ be a weak solution of 
\begin{align*}
\nabla \cdot x_{n+1}^{1-2\alpha} \tilde{g} \nabla \tilde{u} & = 0 \mbox{ in } \R^{n+1}_+\setminus \overline{(\Omega \times \R_+)}.
\end{align*}
Assume that for some open set $W \subset \Omega_e$ it holds that $\tilde{u}|_{W \times \{0\}} = 0 = \lim\limits_{x_{n+1} \rightarrow 0} x_{n+1}^{1-2\alpha} \p_{n+1} \tilde{u}|_{W }$ in a trace and in an $H^{-\alpha}(W)$ sense, respectively. Then, $\tilde{u } \equiv 0$ in $\R^{n+1}_+ \setminus \overline{\Omega \times \R_+ }$.
\end{prop}

For a proof of this, we refer to the extension interpretation of the fractional Laplacian in \cite{FF14,Yu16,R15,Ghosh_Salo_Uhlmann_2020} and the proofs from these articles.

With the unique continuation property in hand, we can turn to the implications of the unique continuation property for our inverse problem.

\begin{lem}[Exterior determination, UCP]
\label{lem:UCP}
Let $\alpha \in (0,1)$ and let $W_1, W_2, \Omega \subset \R^n$ and $ g_1, g_2, \gamma:\R^n \rightarrow \R^{n \times n}_{sym}$ satisfy the conditions from Assumption \ref{assump:main} and assume that $g_1, g_2, \gamma$ are equal to an apriori known smooth metric $g$ in $\Omega_e$.
Let $f\in C_0^{\infty}(W_1)$ and let $\tilde{u}^f \in X_{\alpha}$ denote the weak solution of \eqref{eq:mixed} with data $f$.
Then, the following results hold:
\begin{itemize}
\item[(i)] The knowledge of $(f,\tilde{\Lambda}_{\gamma}^{W_1,W_2}(f)) \in H_0^{\alpha}(W_1) \times H^{-\alpha}(W_2)$ determines 
$\tilde{\Lambda}_{\gamma}^{W_1,\Omega_e}(f)\linebreak \in  H^{-\alpha}(\Omega_e)$. In particular, if for two metrics $g_1, g_2$ and the associated generalized Dirichlet-to-Neumann data $\tilde{\Lambda}_{g_1}^{W_1,W_2}(f)$, $\tilde{\Lambda}_{g_2}^{W_1,W_2}(f)$ it holds that
\begin{align*}
\tilde{\Lambda}_{g_1}^{W_1,W_2}(f) = \tilde{\Lambda}_{g_2}^{W_1,W_2}(f) ,
\end{align*}
then 
\begin{align*}
\tilde{\Lambda}_{g_1}^{W_1,\Omega_e}(f) = \tilde{\Lambda}_{g_2}^{W_1,\Omega_e}(f).
\end{align*}
\item[(ii)] The knowledge of $(h,\tilde{\Lambda}_{\gamma}^{\Omega_e}(h)|_{W_2}) \in H^{\alpha}(\Omega_e) \times H^{-\alpha}(W_2)$ determines $\tilde{\Lambda}_{\gamma}^{\Omega_e}(h)\in H^{-\alpha}(\Omega_e)$. In particular, if for two metrics $g_1, g_2$ and their associated generalized Dirichlet-to-Neumann data $\tilde{\Lambda}_{g_1}^{\Omega_e}(h)$, $\tilde{\Lambda}_{g_2}^{\Omega_e}(h)$  it holds that
\begin{align*}
\tilde{\Lambda}_{g_1}^{\Omega_e}(h)|_{W_2} = \tilde{\Lambda}_{g_2}^{\Omega_e}(h)|_{W_2} ,
\end{align*}
then 
\begin{align*}
 \tilde{\Lambda}_{g_1}^{\Omega_e}(h) =  \tilde{\Lambda}_{g_2}^{\Omega_e}(h).
\end{align*}
\end{itemize}
\end{lem}

\begin{proof}
Both claims follow from the (boundary) unique continuation property and the fact that the metrics $g_1, \gamma, g_2 $ are known to be equal to the metric $g$ in $\Omega_e$. As the arguments are analogous, we only consider the case (i).

We first note that by Definition \ref{defi:Neumann1} (see also Remark \ref{rmk:DtN}) the measurement data $(f,\tilde{\Lambda}_{\gamma}^{W_1,W_2}(f))$ can be rewritten as $(f, \sqrt{|\gamma|} \lim\limits_{x_{n+1} \rightarrow 0} x_{n+1}^{1-2\alpha} \p_{n+1} \tilde{u}^f|_{W_2})$. 

Hence, the determination of $(f, \sqrt{|\gamma|}  \lim\limits_{x_{n+1} \rightarrow 0} x_{n+1}^{1-2\alpha} \p_{n+1} \tilde{u}^f|_{\Omega_e})$ (and thus $(f, \tilde{\Lambda}_{\gamma}^{W_1,\Omega_e}(f))$) from the datum $(f, \tilde{\Lambda}_{\gamma}^{W_1,W_2}(f))$ follows from the boundary unique continuation property, see Proposition \ref{prop:UCP}.

As a consequence, we then also infer the second claim in (i). Indeed, since $g_1 = g_2=:g$ is assumed to be known outside of $\Omega$, we have that, in a weak sense, it holds that for solutions $\tilde{u}^f_1, \tilde{u}^f_2$ of \eqref{eq:mixed} with metrics $g_1, g_2$ the function $\tilde{u}_1^f- \tilde{u}_2^f$ satisfies
\begin{align*}
\nabla \cdot  x_{n+1}^{1-2\alpha} \tilde{g}_1 \nabla (\tilde{u}^f_1- \tilde{u}^f_2) & = 0 \mbox{ in } \R^{n+1}_+ \setminus \overline{(\Omega \times \R_+)},\\
\lim\limits_{x_{n+1} \rightarrow 0} x_{n+1}^{1-2\alpha} \p_{n+1} (\tilde{u}^f_1- \tilde{u}^f_2) & = 0 \mbox{ on } W_2,\\
\tilde{u}_1^f - \tilde{u}_2^f & = 0 \mbox{ on } \Omega_e \times \{0\}.
\end{align*} 
 Thus, by Proposition \ref{prop:UCP} we hence infer that $\tilde{u}^f_1 -\tilde{u}^f_2 \equiv 0$ in $\R^{n+1}_+ \setminus \overline{(\Omega \times \R_+)}$. In particular, this then implies that $\tilde{\Lambda}_{g_1}^{W_1,\Omega_e}(f) = \tilde{\Lambda}_{g_2}^{W_1,\Omega_e}(f)$.
\end{proof}

Next, we discuss an auxiliary lemma on regularity properties, which we will frequently invoke in what follows below. More precisely, we use estimates for the mixed Dirichlet-Neumann problem \eqref{eq:mixed} from above (in the nonlocal context these are also known as Vishik-Eskin estimates \cite{Vishik_Eskin_1965}; we refer to Section \ref{sec:VE_p1} for a discussion of these) to rule out concentration on $\partial \Omega$ for solutions of the nonlocal problem with exterior data.

\begin{lem}
\label{lem:VE}
Let $\alpha \in (0,1)$ and let $W_1, \Omega$ be as in Assumption \ref{assump:main}(a) and let $\gamma:\R^n \rightarrow \R^{n \times n}_{sym}$ satisfy the conditions from Assumption \ref{assump:main}(b).
Let $f\in H^{\alpha+\delta}(\Omega_e)$ with $\delta \in [0,1/2)$. Let $\tilde{u}^f \in X_{\alpha}$ be the unique weak solution of
\eqref{eq:mixed} with data $f$ and metric $\gamma$.
Then the following results hold:
\begin{itemize}
\item[(i)] The restriction $u^f := \tilde{u}^f|_{\R^n \times \{0\}} \in H^{\alpha}(\R^n)$ with data $f \in H^{\alpha+\delta}(\Omega_e)$ enjoys higher regularity $u^f \in H^{\alpha+\delta-\epsilon}(\R^n)$ for any $\epsilon \in (0,\delta)$. Moreover, also for any $\epsilon \in (0,\delta)$
\begin{align*}
\lim\limits_{x_{n+1} \rightarrow 0} x_{n+1}^{1-2\alpha} \p_{n+1} \tilde{u}^f \in H^{-\alpha+\delta-\epsilon}(\R^n), \  \tilde{\Lambda}_{\gamma}^{\Omega_e}(f) \in H^{-\alpha+\delta-\epsilon}_0(\Omega_e).
\end{align*}
Similarly, for  $f \in H^{\alpha+\delta}_0(W_1)$  and for any $\epsilon \in (0,\delta)$ we have $\tilde{\Lambda}_{\gamma}^{ W_1, \Omega_e}(f) \in  H^{-\alpha+\delta-\epsilon}(\Omega_e)$.
\item[(ii)] For  $\delta \in (0,1/2)$, for two solutions $\tilde{u}^f, \tilde{u}^h$ of \eqref{eq:mixed} with data $f,h \in H^{\alpha+\delta}(\Omega_e)$ and their restrictions $u^f, u^h \in H^{\alpha+\delta-\epsilon}(\R^n)$ with $\epsilon \in (0,\delta)$, we have
\begin{align*}
\langle  \tilde{\Lambda}_{\gamma}^{ \Omega_e} (f), u^h|_{\Omega_e} \rangle_{ H^{-\alpha}_0(\Omega_e),H^{\alpha}(\Omega_e)} = \langle  \tilde{\Lambda}_{\gamma}^{\Omega_e} (h), u^f|_{\Omega_e} \rangle_{ H^{-\alpha}_0(\Omega_e),H^{\alpha}(\Omega_e)}.
\end{align*}
An analogous result holds for $\delta = 0$ in which case we then have that $u^f, u^h \in H^{\alpha}(\R^n)$.
\item[(iii)] Let $\delta \in (\max\{0,\alpha - \frac{1}{2}\}, \frac{1}{2})$. Let $f\in C_0^{\infty}(W_1)$ and $h \in H^{\alpha+\delta}(\Omega_e) $ and let $\tilde{u}^f, \tilde{u}^h$ be solutions of \eqref{eq:mixed} with data $f,h$ and metric $\gamma$. Then, for their restrictions ${u}^f, {u}^h$ to $\R^n \times \{0\}$, it holds that 
\begin{align*}
\langle   \lim\limits_{x_{n+1}\rightarrow 0} x_{n+1}^{1-2\alpha} \p_{n+1}  \tilde{u}^f , \sqrt{|\gamma|} u^h\rangle_{ H^{-\alpha}(\R^n),H^{\alpha}(\R^n)}  
& =  \langle  \tilde{\Lambda}_{\gamma}^{\Omega_e} (h), u^f|_{\Omega_e} \rangle_{ H^{-\alpha}_0(\Omega_e),H^{\alpha}(\Omega_e)}\\
&= \langle  \tilde\Lambda_{\gamma}^{\Omega_e}(h), f\rangle_{ H^{-\alpha}_0(\Omega_e),H^{\alpha}(\Omega_e)}.
\end{align*}
\item[(iv)] 
For $\delta \in (\max\{\alpha- \frac{1}{2},0\}, \frac{1}{2} )$ and any $f \in C_0^{\infty}(W_1)$ it holds that $\tilde{\Lambda}_{\gamma}^{W_1, \Omega_e}(f) \in {H}^{-\alpha +\delta}_0(\Omega_e)$ and $\tilde{\Lambda}_{\gamma}^{W_1, \Omega_e}(f)  = \tilde{\Lambda}_{\gamma}^{ \Omega_e}(f) $ as elements in $H^{-\alpha + \delta}_0(\Omega_e)$.
\end{itemize}
\end{lem}

\begin{remark}
We remark that the loss of $\epsilon \in (0,\delta)$ in (i) is not optimal. By using Besov space arguments in the space $B^{\alpha}_{2,\infty}$ instead of $H^{\alpha}$ arguments in combination with (real) interpolation results for Besov spaces these losses could be avoided. We refer to \cite{KM07} for a discussion of this in the case $\alpha = \frac{1}{2}$. In the subsequent sections, we only need a sufficiently large regularity gain which can be achieved by considering $\epsilon>0$ sufficiently small. Hence, we do not discuss the Besov space regularity theory in Appendix \ref{sec:reg} but accept the slightly lossy estimates from above.
\end{remark}

\begin{proof}
(i) Let $f\in H^{\alpha + \delta}(\Omega_e)$ with $\delta \in [0,\frac{1}{2})$. If $\tilde{\gamma}= id_{(n+1)\times (n+1)}$, the regularity of $u^f$ follows from Vishik-Eskin estimates \cite{Vishik_Eskin_1965} and \cite{Grubb_2015}. For general $\tilde{\gamma}$, the regularity can be obtained by proving that $(-\D_{\gamma})^{\alpha}$ is a pseudodifferential operator or by relying on regularity of mixed Dirichlet-Neumann problems. We here follow the second point of view. For $\alpha= \frac{1}{2}$ this is well-known, see, for instance, \cite{Savare,KM07}. In Appendix \ref{sec:reg} we provide an analogous argument for general $\alpha \in (0,1)$; see Proposition \ref{prop:higher_reg} which in particular entails that $u^f \in H^{\alpha +\delta-\epsilon}(\R^n)$.

Given the regularity for the tangential directions from Proposition \ref{prop:higher_reg} in Appendix \ref{sec:reg}, it remains to discuss the regularity of the normal derivative. This follows by duality, a finite difference argument and the regularity of the tangential restriction of the solution. Letting $\tau_{j,h}$ with $j \in \{1,\dots,n\}$, $h\in (0,1)$ denote the tangential finite difference operator, i.e., $\tau_{j,h} \tilde{u}^f(x,x_{n+1}): = \tilde{u}^f(x + h e_j,x_{n+1})- \tilde{u}^f(x,x_{n+1})$, we first observe that if $\tilde{u}^f$ denotes the weak solution of \eqref{eq:mixed} with data $f$, then by \eqref{eq:diff_finite} $\tau_{j,h} \tilde{u}^f$ is a weak solution of
\begin{align}
\label{eq:tangen_diff}
\begin{split}
\nabla \cdot x_{n+1}^{1-2\alpha} \tilde{\gamma} \nabla \tau_{j,h} \tilde{u}^f (x,x_{n+1}) &= -\nabla \cdot x_{n+1}^{1-2\alpha} (\tau_{j,h} \tilde{\gamma})(x) \nabla \tilde{u}^f(x+he_j, x_{n+1}) \mbox{ in } \R^{n+1}_+,\\
 \tau_{j,h} \tilde{u}^f & = \tau_{j,h} u^f \mbox{ on } \R^n \times \{0\}.
 \end{split}
\end{align}
As above, $u^f:= \tilde{u}^f|_{\R^n \times \{0\}}$. 
Using this together with the observation that $(\tau_{j,h} \tilde{\gamma}(x))_{n+1,k} = 0$ for all $k \in \{1,\dots,n+1\}$, we obtain 
\begin{align*}
&\|\sqrt{|\gamma|}\tau_{j,h}\lim\limits_{x_{n+1} \rightarrow 0} x_{n+1}^{1-2\alpha} \p_{n+1}  \tilde{u}^f\|_{H^{-\alpha}(\R^n)}
=\|\sqrt{|\gamma|}\lim\limits_{x_{n+1} \rightarrow 0} x_{n+1}^{1-2\alpha} \p_{n+1} \tau_{j,h} \tilde{u}^f\|_{H^{-\alpha}(\R^n)}\\
&= \sup\limits_{\ell \in H^{\alpha}(\R^n), \|\ell\|_{H^{\alpha}(\R^n)} = 1} \left| \int\limits_{\R^n} \ell  \lim\limits_{x_{n+1} \rightarrow 0} x_{n+1}^{1-2\alpha} \p_{n+1} \tau_{j,h} \tilde{u}^f d V_{\gamma}(x) \right|\\
&\leq  \sup\limits_{\ell \in H^{\alpha}(\R^n), \|\ell\|_{H^{\alpha}(\R^n)} = 1}\left(\left| \int\limits_{\R^{n+1}_+} x_{n+1}^{1-2\alpha} \tilde{\gamma} \nabla \tau_{j,h} \tilde{u}^f \cdot  \nabla \tilde{u}^{\ell} d(x,x_{n+1})\right| \right.\\
& \qquad \qquad \qquad \left. + \left| \int\limits_{\R^{n+1}_+} x_{n+1}^{1-2\alpha} (\tau_{j,h}\tilde{\gamma})(x)  \nabla  \tilde{u}^f(x + he_j,x_{n+1}) \cdot \nabla \tilde{u}^{\ell}(x,x_{n+1}) d(x,x_{n+1})\right| \right)\\
& \leq C \sup\limits_{\ell \in H^{\alpha}(\R^n), \|\ell\|_{H^{\alpha}(\R^n)} = 1} (\|\tau_{j,h} u^f\|_{H^{\alpha}(\R^n)} + \sup\limits_{x \in \R^n} |\tau_{j,h}\gamma(x)| \|f\|_{H^{\alpha}(\Omega_e)}) \|\ell\|_{H^{\alpha}(\R^n)}\\
& \leq C(\|\tau_{j,h} u^f\|_{H^{\alpha}(\R^n)} +\sup\limits_{x \in \R^n} |\tau_{j,h}\gamma(x)| \|f\|_{H^{\alpha}(\Omega_e)}) .
\end{align*}
Here $ \tilde{u}^{\ell} \in X_{\alpha}$ denotes the weak solution to the variable coefficient Caffarelli-Silvestre-type extension problem for the fractional Laplacian with data $ \ell$ on $\R^n$.
In the penultimate line, we used the energy estimates for the solutions $\tilde{u}^f, \tilde{u}^{\ell}$ (see Lemma \ref{lem:well-posedness} with a slight variant for $\tilde{u}^{f}$ treating the right hand side in \eqref{eq:tangen_diff} perturbatively).
Now dividing by $h^{\delta-\epsilon}$ for $h \in (0,1)$ and $\delta  \in [0,\frac{1}{2})$ and using Lemma \ref{lem:Sob_charact}(ii), for any $\epsilon \in (0,\delta)$ and $\epsilon' \in (0,\epsilon)$ we obtain (with a constant which, in particular, depends on $\delta,\epsilon, \epsilon'$)
\begin{align*}
\|\sqrt{|\gamma|} \lim\limits_{x_{n+1} \rightarrow 0} x_{n+1}^{1-2\alpha} \p_{n+1}  \tilde{u}^f\|_{H^{-\alpha+\delta-\epsilon}(\R^n)}
 \leq C (\|u^f\|_{H^{\alpha+\delta-\epsilon'}(\R^n)} + \|f\|_{H^{\alpha}(\Omega_e)}) \leq C\|f\|_{H^{\alpha+\delta}(\Omega_e)} ,
\end{align*}
where we used the tangential regularity of $u^f$ deduced in Appendix \ref{sec:reg} together with Lemma \ref{lem:Sob_charact}(iii) to bound the right hand side as well as the $C^{\infty}$ regularity of $\gamma$ and the fact that $\gamma \neq id$ only on a compact set.
Given the relation between the measurement data $\tilde{\Lambda}_{\gamma}^{ W_1, \Omega_e}$, $ \tilde{\Lambda}_{\gamma}^{\Omega_e}$ and $\lim\limits_{x_{n+1} \rightarrow 0} x_{n+1}^{1-2\alpha} \p_{n+1}  \tilde{u}^f$ (see Remark \ref{rmk:DtN}), the argument for the higher regularity of $\tilde{\Lambda}_{\gamma}^{ W_1, \Omega_e}(f)$, $\tilde{\Lambda}_{\gamma}^{\Omega_e}(f)$ is analogous.
\\

(ii) The identity for solutions follows from the symmetry of the bilinear form defining the generalized Dirichlet-to-Neumann operators. Indeed, first we note that by (i) the restrictions $u^f, u^h$ indeed are $H^{\alpha+\delta-\epsilon}(\R^n)$ regular for $f,h \in H^{\alpha +\delta}(\Omega_e)$ and for any $\epsilon \in (0,\delta)$. Using the inclusion $H^{\alpha +\delta-\epsilon}(\Omega_e)\subset H^{\alpha}(\Omega_e)$, by Definition \ref{defi:Neumann_ext} and the independence of the choice of the extension, we then have, as desired,
\begin{align*}
\langle  \tilde{\Lambda}_{\gamma}^{ \Omega_e} (f), u^h|_{\Omega_e} \rangle_{ H^{-\alpha}_0(\Omega_e),H^{\alpha}(\Omega_e)}
&= - \int\limits_{\R^{n+1}_+} x_{n+1}^{1-2\alpha} \tilde{\gamma} \nabla \tilde{u}^f \cdot  \nabla \tilde{u}^h d(x, x_{n+1})\\
&= \langle \tilde{\Lambda}_{\gamma}^{ \Omega_e} (h), u^f|_{\Omega_e} \rangle_{ H^{-\alpha}_0(\Omega_e),H^{\alpha}(\Omega_e)} .
\end{align*}

(iii) The claimed identity follows from \eqref{eq_Sobolev_duality_mikko}. By assumption, both $u^f$ and $u^h$ are in \linebreak $H^{\alpha+\delta-\epsilon}(\R^n)$ for $\delta \in (\max\{\alpha - \frac{1}{2},0\}, 1/2 )$ and $\epsilon \in (0,\delta)$. Hence, choosing $\epsilon>0$ such that $-\alpha+\delta-\epsilon \in (-1/2,1/2)$ and such that the functions $\lim\limits_{x_{n+1}\rightarrow 0} x_{n+1}^{1-2\alpha} \p_{n+1}  \tilde{u}^f  $, $\tilde{\Lambda}_{\gamma}^{ \Omega_e}(h)$ are in $H^{-\alpha+\delta-\epsilon}(\Omega_e)$ and $H^{-\alpha+\delta-\epsilon}_0(\Omega_e)$, respectively, then allows us to invoke the splitting property from \eqref{eq_Sobolev_duality_mikko}:
\begin{align}
\label{eq:split_higher_reg}
\begin{split}
&\langle   \lim\limits_{x_{n+1}\rightarrow 0} x_{n+1}^{1-2\alpha} \p_{n+1}  \tilde{u}^f , \sqrt{|\gamma|} u^h\rangle_{ H^{-\alpha}(\R^n),H^{\alpha}(\R^n)}  \\
&=\langle   \lim\limits_{x_{n+1}\rightarrow 0} x_{n+1}^{1-2\alpha} \p_{n+1}  \tilde{u}^f , \sqrt{|\gamma|} u^h\rangle_{ H^{-\alpha+\delta-\epsilon}(\R^n),H^{\alpha-\delta+\epsilon}(\R^n)}  \\
& = \langle   \lim\limits_{x_{n+1}\rightarrow 0} x_{n+1}^{1-2\alpha} \p_{n+1}  \tilde{u}^f , \sqrt{|\gamma|} u^h\rangle_{ H^{-\alpha+\delta-\epsilon}(\Omega),H^{\alpha-\delta+\epsilon}_0(\Omega)} \\
& \quad \quad   + \langle   \lim\limits_{x_{n+1}\rightarrow 0} x_{n+1}^{1-2\alpha} \p_{n+1}  \tilde{u}^f , \sqrt{|\gamma|} u^h\rangle_{ H^{-\alpha+\delta-\epsilon}(\Omega_e),H^{\alpha-\delta+\epsilon}_0(\Omega_e)} \\
& = \langle   \lim\limits_{x_{n+1}\rightarrow 0} x_{n+1}^{1-2\alpha} \p_{n+1}  \tilde{u}^f , \sqrt{|\gamma|} u^h\rangle_{ H^{-\alpha+\delta-\epsilon}(\Omega_e),H^{\alpha-\delta+\epsilon}_0(\Omega_e)}  \\
& = \langle   \lim\limits_{x_{n+1}\rightarrow 0} x_{n+1}^{1-2\alpha} \p_{n+1}  \tilde{u}^f , \sqrt{|\gamma|} u^h\rangle_{ H^{-\alpha}_0(\Omega_e),H^{\alpha}(\Omega_e)}  \\
& =  \langle  \tilde{\Lambda}_{\gamma}^{\Omega_e} (f), u^h|_{\Omega_e} \rangle_{ H^{-\alpha}_0(\Omega_e),H^{\alpha}(\Omega_e)}
 =  \langle  \tilde{\Lambda}_{\gamma}^{\Omega_e} (h), u^f|_{\Omega_e} \rangle_{ H^{-\alpha}_0(\Omega_e),H^{\alpha}(\Omega_e)} \\
 & =   \langle  \tilde{\Lambda}_{\gamma}^{\Omega_e} (h), f \rangle_{ H^{-\alpha}_0(\Omega_e),H^{\alpha}(\Omega_e)} .
 \end{split}
\end{align}
In the second line, we used the higher regularity of the weighted generalized normal derivative, the third line follows from \eqref{eq_Sobolev_duality_mikko}, in the fourth line we used the equation \eqref{eq:mixed} for $\tilde{u}^f$, in the fifth line, we again made use of the regularity of $u^h$ and $\tilde{u}^f$, and in the penultimate line we relied on property (ii) from above. In these regularity arguments, we exploited that $H^{\alpha-\delta+\epsilon}(\Omega) = H^{\alpha-\delta + \epsilon}_0(\Omega)$ and $H^{-\alpha+\delta-\epsilon}_0(\Omega) = H^{-\alpha+\delta - \epsilon}(\Omega)$ (and analogously also for these spaces on the set $\Omega_e$) since $-\alpha+\delta-\epsilon \in (-1/2,1/2)$ (see Theorem \ref{thm_Agranovich}(ii)).\\

(iv) The proof of the first assertion follows from (i) by noting that for $\delta \in (\max\{0,\alpha- \frac{1}{2}\}, \frac{1}{2})$ we have that $\beta:=-\alpha+ \delta \in (-\frac{1}{2},\frac{1}{2})$ and that for this range of $\beta$ it holds that $H^{\beta}(\Omega_e) = H^{\beta}_0(\Omega_e)$. By density of $C_0^{\infty}(\Omega_e)$ in $H^{\beta}(\Omega_e)$, it thus suffices to test the identity $\tilde{\Lambda}_{\gamma}^{W_1, \Omega_e}(f)  = \tilde{\Lambda}_{\gamma}^{ \Omega_e}(f) $ on $C_0^{\infty}(\Omega_e)$. On this set it is valid by definition.
\end{proof}

As a consequence of the above regularity estimates, it is possible to pass from the generalized Dirichlet-to-Neumann measurements  with Dirichlet data localized in $W_1$ to measurements based on Dirichlet data (not necessarily compactly supported) in $\Omega_e$. 
More precisely, for $j \in \{1,2\}$, we argue that the knowledge of the data $(f,\tilde{\Lambda}_{g_j}^{W_1,W_2}(f))$ for $f\in H_0^{\alpha}(W_1)$ and $W_2 \subset \Omega_e$ yields the knowledge of the data $(h, \tilde{\Lambda}_{g_j}^{\Omega_e}(h))$ with $h\in H^{\alpha}(\Omega_e)$, the difference between the two measurements being that for the measurements of $\tilde{\Lambda}_{g_j}^{\Omega_e}(h)$ we allow for (not necessarily compactly supported) Dirichlet data with support in $\Omega_e$ (instead of $W_1$).

\begin{lem}
\label{lem:modified_data1}
Let $\alpha \in (0,1)$ and let $W_1, W_2, \Omega$ be as in Assumption \ref{assump:main}(a) and let $g_1,g_2:\R^n \rightarrow \R^{n \times n}_{sym}$ satisfy the condition from Assumption \ref{assump:main}(b).
Let $j\in\{1,2\}$ and 
assume that $\tilde{\Lambda}_{g_j}^{W_1,W_2}$, $\tilde\Lambda_{g_j}^{\Omega_e}$ are defined as in Definitions \ref{defi:Neumann1}, \ref{defi:Neumann_ext}, associated with the metrics $g_1, g_2$, respectively.
Suppose further that $\tilde{\Lambda}_{g_1}^{W_1, W_2} = \tilde{ \Lambda}_{g_2}^{W_1, W_2}$. Then, $\tilde{\Lambda}_{g_1}^{\Omega_e} = \tilde{\Lambda}_{g_2}^{\Omega_e}$.
\end{lem}

\begin{proof}
Let $f\in C_0^{\infty}(W_1)$, $h \in C^{\infty}(\Omega_e)\cap H^{\alpha}(\Omega_e)$, and denote by $u^f_j$, $u^h_j$ the restrictions of the weak solutions of \eqref{eq:mixed} with the metrics $g_j$ and data $f,h$, respectively. Then, by unique continuation (Lemma \ref{lem:UCP}(i))
and the fact that $g_j = g$ in $\Omega_e$, we obtain that $\tilde{\Lambda}_{g_1}^{W_1,\Omega_e} (f)= \tilde{\Lambda}_{g_2}^{W_1,\Omega_e} (f)$. By Lemma \ref{lem:VE}(iv) this further entails that 
\begin{align}
\label{eq:assump_local_maps}
\tilde{\Lambda}_{g_1}^{\Omega_e} (f)= \tilde{\Lambda}_{g_2}^{\Omega_e} (f).
\end{align}
Here we used that by Lemma \ref{lem:VE}(i) it holds that $\tilde{\Lambda}_{g_j}^{\Omega_e}(f) \in H^{-\alpha+\delta-\epsilon}(\Omega_e)$,  $j\in \{1,2\}$; and moreover, we also obtain that $\lim\limits_{x_{n+1} \rightarrow 0} x_{n+1}^{1-2\alpha} \p_{n+1} \tilde{u}^f_1\in H^{-\alpha+\delta-\epsilon}(\R^n) $ for any $\delta \in [0,1/2)$ and $\epsilon \in (0,\delta)$ in case that $f \in C_0^{\infty}(W_1)$. In particular, this is valid for $\delta \in (\max\{\alpha-\frac{1}{2},0\},\frac{1}{2})$, which entails that $\alpha - \delta \in (-\frac{1}{2}, \frac{1}{2})$. Let us choose $\epsilon>0$ such that $-\alpha + \delta -\epsilon \in (-\frac{1}{2},\frac{1}{2})$. With such a choice of $\delta>0$ and $\epsilon>0$, using that by virtue of the equation \eqref{eq:mixed}, it holds weakly that
\begin{align*}
\lim\limits_{x_{n+1}\rightarrow 0} x_{n+1}^{1-2\alpha} \p_{n+1}  \tilde{u}^f_1|_{\Omega} = 0 = \lim\limits_{x_{n+1}\rightarrow 0} x_{n+1}^{1-2\alpha} \p_{n+1}  \tilde{u}^f_2|_{\Omega}  
\end{align*}
and by using \eqref{eq:assump_local_maps}, Remark \ref{rmk:DtN} and \eqref{eq_Sobolev_duality_mikko}, similarly as in \eqref{eq:split_higher_reg}, we obtain that
\begin{align}
\label{eq:concatenate1}
\begin{split}
&\langle  \lim\limits_{x_{n+1}\rightarrow 0} x_{n+1}^{1-2\alpha} \p_{n+1}  \tilde{u}^f_1, \sqrt{|g_1|} u^h_1 \rangle_{ H^{-\alpha}(\R^n),H^{\alpha}(\R^n)} \\
&= \langle   \lim\limits_{x_{n+1}\rightarrow 0} x_{n+1}^{1-2\alpha} \p_{n+1}  \tilde{u}^f_2, \sqrt{|g_2|} u^h_2\rangle_{H^{-\alpha}(\R^n),H^{\alpha}(\R^n)} .
\end{split}
\end{align}
By definition of the Dirichlet-to-Neumann maps, and Lemma \ref{lem:VE}(iii)
\begin{align}
\label{eq:concatenate2}
\begin{split}
\langle   \lim\limits_{x_{n+1}\rightarrow 0} x_{n+1}^{1-2\alpha} \p_{n+1}  \tilde{u}^f_1 , \sqrt{|g_1|} u^h_1\rangle_{ H^{-\alpha}(\R^n),H^{\alpha}(\R^n)}  
&= \langle  \tilde\Lambda_{g_1}^{\Omega_e}(h), f\rangle_{ H^{-\alpha}_0(\Omega_e),H^{\alpha}(\Omega_e)} .
\end{split}
\end{align}
Combining \eqref{eq:concatenate1} and \eqref{eq:concatenate2}, we hence arrive at
\begin{align*}
&\langle  \tilde\Lambda_{g_1}^{\Omega_e}(h),f \rangle_{ H^{-\alpha}_0(\Omega_e),H^{\alpha}(\Omega_e)} 
=\langle  \lim\limits_{x_{n+1}\rightarrow 0} x_{n+1}^{1-2\alpha} \p_{n+1}  \tilde{u}^f_1, \sqrt{|g_1|} u_1^h \rangle_{ H^{-\alpha}(\R^n),H^{\alpha}(\R^n)}\\  
&= \langle  \lim\limits_{x_{n+1}\rightarrow 0} x_{n+1}^{1-2\alpha} \p_{n+1}  \tilde{u}^f_2, \sqrt{|g_2|} u_2^h \rangle_{ H^{-\alpha}(\R^n),H^{\alpha}(\R^n)} 
= \langle  \tilde\Lambda_{g_2}^{\Omega_e}(h), f\rangle_{ H^{-\alpha}_0(\Omega_e),H^{\alpha}(\Omega_e)} .
\end{align*}
By density of $C^{\infty}(\Omega_e)\cap H^{\alpha}(\Omega_e) \subset H^{\alpha}(\Omega_e)$ and since $h \in C^{\infty}(\Omega_e)\cap H^{\alpha}(\Omega_e)$ was arbitrary, this implies that $\tilde{\Lambda}_{g_1}^{\Omega_e}|_{W_1} = \tilde{\Lambda}_{g_2}^{\Omega_e}|_{W_1}$. Last but not least, by a unique continuation argument as in Lemma \ref{lem:UCP}(ii), this also implies that $\tilde{\Lambda}_{g_1}^{\Omega_e} = \tilde{\Lambda}_{g_2}^{\Omega_e}$.
\end{proof}

\begin{remark}
Let us comment on the density $C^{\infty}(\Omega_e)\cap H^{\alpha}(\Omega_e) \subset H^{\alpha}(\Omega_e)$. The only slightly non-standard point here is the unboundedness of $\Omega_e$. To observe the desired density, we simply exhaust $\Omega_e$ by $B_{k}(0)\cap \Omega_e$ with $k \in \N$, $k\geq k_0$, where $k_0>1$ such that $\overline{\Omega} \subset B_{k_0}(0)$. Next, we observe that by a cut-off argument, it is possible to approximate any $H^{\alpha}(\Omega_e)$ function $f$ by a sequence of functions $\{f_k\}_{k=k_0}^{\infty}$ such that $f_k|_{\partial B_k(0)}=0$ in a trace sense. 
Moreover, it is possible to approximate any function $h$ in $H^{\alpha}(B_{k}(0)\cap \Omega_e)$ satisfying $h|_{\partial B_k(0)}=0$ by $C^{\infty}(B_{k}(0)\cap \Omega_e)$ functions with vanishing trace on $\partial B_{k}(0)$. Taking the union over all such $k \geq k_0$ hence implies the claim.
\end{remark}

\subsubsection{Step 2: From $\tilde{\Lambda}_{g}^{\Omega_e}$ to $\tilde{L}_{g}^{\Omega_e.\Omega_e}$ }

It is the purpose of this section to reduce the Dirichlet-to-Neumann measurements to measurements for a source-to-solution problem. To this end, we recall that in $\Omega_e$ it holds that $g_1 = g_2 =:g$. For $H= (-\D_g) F$ with $F \in C_0^{\infty}(\Omega_e)$, we consider the following source-to-solution operator
\begin{align}
\label{eq:source_to_sol}
\tilde{L}_{g_j}^{\Omega_e,\Omega_e}: C_0^{\infty}(\Omega_e) \rightarrow H^{\alpha}(\Omega_e), \ F \mapsto  v^F_j|_{\Omega_e},
\end{align} 
where, with slight abuse of notation, $\tilde{v}^F_j\in X_{\alpha}$ denotes a weak solution of \eqref{eq:Neumann} with generalized Neumann boundary data $H:=(-\D_{g_j}) F$ with $v^F_j:= \tilde{v}^F_j|_{\R^n \times \{0\}}$ and with metric $g_j$ with $j\in\{1,2\}$ satisfying the conditions from Assumption \ref{assump:main}. We stress that we have used the result of Proposition \ref{prop:Schwartz_kernel} to infer the $L^2(\R^n)$ integrability of $v^F_j$.

Using the results of Proposition \ref{prop:Schwartz_kernel}, we now reduce the exterior measurement data to the source-to-solution problem. In particular, once this is achieved, we do not return to the original measurement data.

\begin{prop}
\label{prop:source-to-sol}
Let $\alpha \in (0,1)$. Let $ \Omega$ be a smooth set satisfying the conditions in Assumption \ref{assump:main}(a) and let $g_1, g_2$ satisfy the conditions in Assumption \ref{assump:main}(b).
Let $\tilde{\Lambda}_{g_j}^{\Omega_e}$ with $j \in \{1,2\}$ be as above associated with the metrics $g_1, g_2$ and assume that $\tilde{\Lambda}_{g_1}^{\Omega_e} = \tilde{\Lambda}_{g_2}^{\Omega_e}$.  Let the operators $\tilde{L}_{g_j}^{\Omega_e, \Omega_e}$ be as in \eqref{eq:source_to_sol} for $j\in \{1,2\}$. Let $H= (-\D_g) F$ for some $F \in C_0^{\infty}(\Omega_e)$.
Then,
\begin{align*}
\tilde{L}_{g_1}^{\Omega_e,\Omega_e} = \tilde{L}_{g_2}^{\Omega_e,\Omega_e}.
\end{align*}
\end{prop}

\begin{proof}
For $H := (-\Delta_g) F \in C_0^{\infty}(\Omega_e)$ with some $F \in C_0^{\infty}(\Omega_e)$, we consider a weak solution $\tilde{v}^F_1$ of the Neumann problem \eqref{eq:Neumann} with data $H$ and the metric $g_1$.
By Proposition \ref{prop:Neumann} this problem has a unique weak solution $\tilde{v}^F_1 \in X_{\alpha}$ and thus $v^F_1 \in H^{\alpha}(\R^n)$ (see Proposition \ref{prop:Schwartz_kernel} for the fact that $v^F_1 \in L^2(\R^n)$). By construction, the function $\tilde{v}_1^F$ is a weak solution of \eqref{eq:mixed} with Dirichlet data $h:= v_1^F|_{\Omega_e} \in H^{\alpha}(\Omega_e)$. By higher regularity for the Neumann problem (e.g., see Proposition \ref{prop:higher_reg} for a sufficient setting for our argument) we have that since $H \in C_0^{\infty}(\Omega_e)$, it holds that $h \in H^{\alpha+\delta}(\Omega_e)$ for any $\delta \in [0,\frac{1}{2})$.
With slight abuse of notation, we will also refer to the function $\tilde{v}_1^F$ as $\tilde{u}_1^h$ (in order to be consistent with the notation from Lemma \ref{lem:well-posedness}). In particular, by construction, the function $\tilde{u}_1^h$ has the generalized Dirichlet-to-Neumann data $(h,\tilde{\Lambda}_{g_1}^{\Omega_e}(h))=(h,H|_{\Omega_e}) \in H^{\alpha}(\Omega_e) \times H^{-\alpha}(\Omega_e)$.

Next, we consider the unique weak solution $\tilde{u}_2^h \in X_{\alpha}$ of \eqref{eq:mixed} with data $h \in H^{\alpha}(\Omega_e)$ and metric $g_2$.
As, by the assumption that $\tilde{\Lambda}_{g_1}^{\Omega_e}(h) = \tilde{\Lambda}_{g_2}^{\Omega_e}(h)$ and since $H|_{\Omega_e}=\tilde{\Lambda}_{g_1}^{\Omega_e}(h)$, we obtain that
\begin{align*}
H|_{\Omega_e}=\tilde{\Lambda}_{g_1}^{\Omega_e}(h) = \tilde{\Lambda}_{g_2}^{\Omega_e}(h) . 
\end{align*}

By Lemma \ref{lem:VE}(i), as $h \in H^{\alpha + \delta}(\Omega_e)$, we have that 
\begin{align*}
u_2^h \in H^{\alpha + \delta -\epsilon }(\R^n) \mbox{ and } \lim\limits_{x_{n+1} \rightarrow 0 } x_{n+1}^{1-2\alpha}\p_{n+1} \tilde{u}_2^h \in H^{-\alpha + \delta -\epsilon }(\R^n) \mbox{ for any } \epsilon \in (0,\delta).
\end{align*}
 As a consequence, by this higher $H^{\alpha+\delta-\epsilon}(\R^n)$ regularity of $u_2^h$, we infer that $\tilde{u}_2^h$ is a weak solution of \eqref{eq:Neumann} with data $H \in C_0^{\infty}(\Omega_e)$.
In particular, since $(u_1^h)|_{\Omega_e} = (u_2^h)|_{\Omega_e}$ and since $u_1^h, u_2^h$ satisfy \eqref{eq:Neumann} with data $H$ and metric $g_1, g_2$, we have that
\begin{align*}
\tilde{L}_{g_1}^{\Omega_e,\Omega_e}(F) = \tilde{L}_{g_2}^{\Omega_e,\Omega_e}(F).
\end{align*}
\end{proof}

\subsubsection{Step 3: Conclusion -- uniqueness for the source-to-solution problem}

The final step in our uniqueness proof is to conclude that the source-to-solution data determine the heat semigroup, i.e., to deduce that, since $\tilde{L}_{g_1}^{\Omega_e,\Omega_e} = \tilde{L}_{g_2}^{\Omega_e,\Omega_e}$, it follows that $p_t^{(-\D_{g_1})}(x,z):=e^{t\D_{g_1}}(x,z) = e^{t\D_{g_2}}(x,z)=:p_t^{(-\D_{g_2})}(x,z)$ for all $x,z \in \Omega_e$, where $p_t^{(-\D_{g_j})}(x,z)$ denotes the heat kernel of $(-\D_{g_j})$ with respect to the metric $g_j$ for $j \in \{1,2\}$. The uniqueness of the metrics then follows from the boundary control method as in \cite{FGKU_2021}. See Section \ref{sec:hk} above for more on this.

\begin{prop}[Source-to-solution problem, recovery of heat kernel]
\label{prop:conclude}
Let $\alpha \in (0,1)$ and let the metrics $g_1, g_2$ satisfy the conditions from Assumption \ref{assump:main}(b).
Let $\tilde{L}_{g_1}^{\Omega_e,\Omega_e}, \tilde{L}_{g_2}^{\Omega_e,\Omega_e}$ denote the source-to-solution operators associated with the metrics $g_1, g_2$ and assume that 
\begin{align*}
\tilde{L}_{g_1}^{\Omega_e,\Omega_e} = \tilde{L}_{g_2}^{\Omega_e,\Omega_e}.
\end{align*}
Then, for $x, z \in \Omega_e$, $x\neq z$ and for $t>0$, we have that
\begin{align*}
p_t^{(-\D_{g_1})}(x,z) = p_t^{(-\D_{g_2})}(x,z).
\end{align*}
\end{prop}

We begin by discussing two auxiliary results. First, we consider a Neumann Poisson-type representation formula, which is essentially contained in \cite[Proposition 2.3]{ST10}. For completeness, and since \cite{ST10} does not provide a uniqueness result that would have allowed us to directly invoke Proposition 2.3, we include a proof of this formula.
\begin{lem}
\label{lem:representation}
Let $\alpha \in (0,1)$, let $\gamma$ be a metric satisfying the conditions of Assumption \ref{assump:main}(b) and let $F\in C_0^{\infty}(\R^n)$ and $H= (-\D_{\gamma}) F$.  Let $\tilde{u}^F \in X_{\alpha}$ be the weak solution of \eqref{eq:Neumann} with data $H$. Then, for $x \in \R^n \setminus \supp(F)$ and $x_{n+1}\geq 0$,
\begin{align}
\label{eq:repr_heat}
\tilde{u}^F(x,x_{n+1}) = \hat{c}_{\alpha} \int\limits_{\R^n} \int\limits_{0}^{\infty} p_t^{(-\D_{\gamma})}(x,z) e^{- \frac{x_{n+1}^2}{4 t}} \frac{dt}{ t^{1-\alpha}} H(z) dV_\gamma(z).
\end{align}
Here $p_t^{(-\D_{\gamma})}(x,z)$ denotes the heat kernel of $(-\D_{\gamma})$ on $\R^n$ with metric $\gamma$ and $\hat{c}_{\alpha} \in \R\setminus \{0\}$. 
\end{lem}

\begin{remark}
We remark that by the heat kernel estimates from \eqref{eq_2_5_1} for $x \in \R^n \setminus \supp(F)$ and $x_{n+1}\geq 0$, the function $ p_t^{(-\D_{\gamma})}(x,z) e^{- \frac{x_{n+1}^2}{4 t}} \frac{1}{ t^{1-\alpha}} H(z) $ is $L^1(\R^n \times (0,\infty))$ integrable as a function in $(z,t)$.
\end{remark}

\begin{proof}

\emph{Step 1: Comparison with the Dirichlet problem in large domains.}
Let $R>0$ be so large that $\supp(F) \subset B_R(0)$ and $H= (-\D_{\gamma}) F$.  We then consider the following auxiliary problem, again understood weakly through the associated bilinear form:
\begin{align}
\label{eq:CS_R}
\begin{split}
\nabla \cdot x_{n+1}^{1-2\alpha} \tilde{\gamma} \nabla \tilde{u}_R^F &= 0 \mbox{ in } B_R(0) \times \R_+,\\
\lim\limits_{x_{n+1} \rightarrow 0} x_{n+1}^{1-2\alpha} \p_{n+1} \tilde{u}_R^F & = H \mbox{ in } B_R(0) ,\\
\tilde{u}_R^F &= 0 \mbox{ on } \partial B_R(0) \times \R_+.
\end{split}
\end{align}
By, for instance, carrying out a tangential diagonalization of the Laplacian and solving the vertical ODE, one obtains that this problem has a unique solution  $\tilde{u}_R^F \in \dot{H}^1(B_R(0)\times \R_+)$. Extending this by zero (and using the zero trace condition) we can also identify this with an element in $ X_{\alpha}$. Let us outline the existence and uniqueness argument in more detail. Let $0<\lambda_1\leq \lambda_2 \leq \dots$ be the Dirichlet eigenvalues of $-\D_{\gamma}$ on $B_R(0)$ and let $\{\psi_k\}_{k =1}^{\infty}$ denote an associated $L^2(B_R(0), dV_{\gamma})$ orthonormal basis of Dirichlet eigenfunctions, i.e.,
\begin{align*}
-\D_{\gamma} \psi_k & = \lambda_k \psi_k \mbox{ in } B_R(0),\\
\psi_k & = 0 \mbox{ on } \partial B_R(0).
\end{align*}
We then diagonalize the tangential contribution of the operator from \eqref{eq:CS_R} (viewed as the Laplace-Beltrami operator on $(B_R(0),\gamma)$ after dividing by $\sqrt{|\gamma|}$). Recalling the formulation of the bulk operator as in \eqref{eq:LB_extended}, we are thus lead to the following decoupled ODEs
\begin{align*}
\p_{n+1}  x_{n+1}^{1-2\alpha} \p_{n+1} (\tilde{u}^F_R,\psi_k)_{L^2(B_R(0),d V_{\gamma})} - \lambda_k (\tilde{u}^F_R,\psi_k)_{L^2(B_R(0), d V_{\gamma})} &= 0 \mbox{ in } \{x_{n+1}>0\},\\
\lim\limits_{x_{n+1} \rightarrow 0} x_{n+1}^{1-2\alpha} \p_{n+1} (\tilde{u}^F_R,\psi_k)_{L^2(B_R(0), d V_{\gamma})} & = (H,\psi_k)_{L^2(B_R(0), d V_{\gamma})} .
\end{align*}
The solution of this family of ODEs is then given by 
\begin{align}
\label{eq:repr1a}
\tilde{u}_R^F(x,x_{n+1}) = c_{\alpha}\sum\limits_{k =1}^{\infty} (H,\psi_k)_{L^2(B_R(0), d V_{\gamma})} \lambda_k^{-\alpha} \left( \sqrt{\lambda_k} x_{n+1} \right)^{\alpha} K_{\alpha}(\sqrt{\lambda_k} x_{n+1}) \psi_k(x),
\end{align}
where $c_{\alpha}\neq 0$ is an explicit constant and $K_{\alpha}$ denotes the modified Bessel function of the second kind (see, for instance, \cite[Chapter 10.25]{NIST} and its defining ODE). We refer to, e.g., \cite{ST10} or Appendix A in \cite{Ruland_2023} for a more detailed derivation of this. We note that the sum in \eqref{eq:repr1a} is convergent in $H^{1}(B_R(0)\times \R_+, x_{n+1}^{1-2\alpha})$. Indeed, this follows from the $L^2(B_R(0),dV_{\gamma})$ orthogonality of the functions $\psi_k$ together with the asymptotics of the modified Bessel functions \cite[equations (10.30.2), (10.29.2)]{NIST}: For $z\geq 0$ we have
\begin{align}
\label{eq:Bessel_ids}
\begin{split}
&K_{\alpha}(z) \sim 2^{\alpha-1} \Gamma(\alpha) z^{-\alpha} \mbox{ for } z \rightarrow 0, \ K_{\alpha}(z) \sim \sqrt{\frac{\pi}{2z}} e^{-z} \mbox{ for } z \rightarrow \infty,\\
&\frac{d}{dz}(z^{\alpha} K_{\alpha}(z)) = z^{\alpha}K_{1-\alpha}(z).
\end{split}
\end{align}
For instance, by monotone convergence and the identities in \eqref{eq:Bessel_ids}, we then obtain that, since $H \in L^2(B_R(0),dV_\gamma)$,
\begin{align*}
&\|x_{n+1}^{\frac{1-2\alpha}{2}} \p_{n+1} \tilde{u}_R^F \|_{L^2(B_R(0)\times \R_+, dV_{\gamma} \times dx_{n+1})}^2\\
&= c_{\alpha}^2\sum\limits_{k=1}^{\infty}|(H,\psi_k)_{L^2(B_R(0),dV_{\gamma})}|^2 \lambda_k^{-\alpha} \int\limits_{\R_+} z |K_{1-\alpha}(z)|^2 dz < \infty.
\end{align*}
By these and similar considerations for the tangential derivatives, we hence conclude that the function $\tilde{u}_R^F$ satisfies that $\tilde{u}_R^F \in H^1(B_{R}(0)\times \R_+, x_{n+1}^{1-2\alpha})$ and that $\tilde{u}_R^F = 0$ on $\partial B_R(0) \times \R_+$ in a trace sense.

Next, we seek to deduce uniform in $R>0$ energy estimates. To this end, we invoke the fact that $\tilde{u}_R^F = 0$ on $\partial B_R(0) \times \R_+$ and extend the functions $\tilde{u}_R^F$ by zero. This yields new functions which are defined on the whole of $\R^{n+1}_+$. With slight abuse of notation, we still refer to these functions as $\tilde{u}_R^F$. Now, by energy estimates, these functions satisfy the following bounds
\begin{align}
\label{eq:apriori}
\begin{split}
\|x_{n+1}^{\frac{1-2\alpha}{2}} \nabla \tilde{u}_R^F\|_{L^2(\R^{n+1}_+)} + \|u_R^F\|_{L^{\frac{2n}{n-2\alpha}}(\R^{n})}
+ \|u_R^F\|_{\dot{H}^{\alpha}(\R^{n})}
&\leq C \|\sqrt{|\gamma|} H\|_{\dot{H}^{-\alpha}(\R^{n})},
\end{split}
\end{align}
where $u_R^F:= \tilde{u}_R^F|_{\R^n \times \{0\}}$ in the trace sense and with $C>0$ independent of $R>0$. Indeed, this follows by inserting the function $\tilde{u}_R^F \in H^1(\R^{n+1}_+, x_{n+1}^{1-2\alpha})$ as a test function into the weak formulation of \eqref{eq:CS_R}. After our extension procedure this reads
\begin{align*}
\int\limits_{\R^{n+1}_+} x_{n+1}^{1-2\alpha} \tilde{\gamma} \nabla \tilde{u}_R^F \cdot \nabla \tilde{\varphi} d(x, x_{n+1}) = -\int\limits_{B_R(0)} H \tilde{\varphi}|_{\R^n \times \{0\}} \sqrt{|\gamma|} dx,
\end{align*}
for all $\tilde\varphi \in \dot{H}^1( \R^{n+1}_+, x_{n+1}^{1-2\alpha})$ with $\tilde\varphi = 0$ on $\R^{n+1}_+ \setminus (B_R(0) \times \R_+)$. Since $\tilde{u}_R^F$ is admissible as a test function in the weak formulation, by the uniform ellipticity of $\tilde{\gamma}$ and the trace estimate \eqref{eq:trace}, we infer a bound on the first left hand side term in \eqref{eq:apriori}:
\begin{align*}
\|x_{n+1}^{\frac{1-2\alpha}{2}} \nabla \tilde{u}_R^F\|_{L^2(\R^{n+1}_+)} 
&\leq C \|\sqrt{|\gamma|}H\|_{\dot{H}^{-\alpha}(\R^{n+1}_+)}.
\end{align*}
The remaining bounds in \eqref{eq:apriori} follow by the trace and Sobolev trace estimates from \eqref{eq:trace}, \eqref{eq:Sob_trace}. Finally, the computation in \eqref{eq:neg_norm_bound} shows that the right hand side of \eqref{eq:apriori} is bounded uniformly in $R>0$.

Next, we show that for $x \notin \supp(F)$ and $x_{n+1}\geq 0$ and some constant $\hat{c}_{\alpha} \neq 0$, the representation formula \eqref{eq:repr1a} can be rewritten as
\begin{align}
\label{eq:repr1}
\tilde{u}_R^F(x,x_{n+1}) =  \hat{c}_{\alpha} \int\limits_{B_R(0)} \int\limits_{0}^{\infty} p_t^{(-\D_{\gamma}^R)}(x,z) e^{- \frac{x_{n+1}^2}{4 t}} \frac{dt}{ t^{1-\alpha}} H(z) dV_\gamma(z).
\end{align}
Here $ p_t^{(-\D_\gamma^R)}(x,z)=: e^{t \Delta_\gamma^R}(x,z)$ denotes the heat kernel for the Dirichlet Laplacian $(-\D_\gamma)$ on $B_R(0)$. We recall the upper heat kernel bounds \cite[Corollary 3.2.8]{Davies_book}
\begin{align*}
|p_t^{(-\D_{\gamma}^R)}(x,z)| \leq C t^{-\frac{n}{2}} \exp\left(-c \frac{|x-z|^2}{t} \right),
\end{align*}
for all $t>0$, $x,z \in B_R(0)$ and some constants $c,C>0$.
In particular, we have that for $H \in C_0^{\infty}(\R^n)$, $x_{n+1}\geq 0$ and $x \notin  \supp(F)$, as a function in $(z,t)$
\begin{align*}
 \frac{p_t^{(-\D_{\gamma}^R)}(x,z) e^{- \frac{x_{n+1}^2}{4 t}}}{ t^{1-\alpha}} H(z) \in L^1(B_R(0) \times (0,\infty)).
\end{align*}
This allows us to invoke Fubini arguments on the right hand side of \eqref{eq:repr1} and to exchange the $z,t$ integrations.

In order to deduce \eqref{eq:repr1} from \eqref{eq:repr1a}, we combine the explicit representation formula \eqref{eq:repr1a} with the representation of the heat kernel in terms of the eigenfunctions $\{\psi_k\}_{k = 1}^{\infty}$ from above.
To this end, we observe that also $p_t^{(-\D^R_{\gamma})}(x,z) $ can be diagonalized using these eigenfunctions of the Laplace-Beltrami operator on $(B_R(0),\gamma)$ and their associated eigenvalues. More precisely, we obtain that
\begin{align*}
p_t^{(-\D^R_\gamma)}(x,z) = \sum\limits_{k = 1}^{\infty} e^{-\lambda_k t} \psi_k(x) \psi_k(z),
\end{align*}
where the series converges in $H^1(B_R(0))$.
Expanding also $H \in C_0^{\infty}(B_R(0))$ in terms of this eigenfunction basis, i.e., 
\begin{align*}
H(x)= \sum\limits_{k=1}^{\infty} (H,\psi_k)_{L^2(B_R(0), dV_\gamma)} \psi_k(x) \mbox{ in } H^1(B_R(0)),
\end{align*}
 we then infer that for $x\notin \supp(F)$ and $x_{n+1}\geq 0$,
\begin{align*}
&\int\limits_{B_R(0)} \int\limits_{0}^{\infty} p_t^{(-\D_\gamma^R)}(x,z) e^{- \frac{x_{n+1}^2}{4 t}} \frac{dt}{ t^{1-\alpha}} H(z) dV_\gamma(z)\\
& = \sum\limits_{k =1}^{\infty} (H,\psi_k)_{L^2(B_R(0), dV_\gamma)}\int\limits_{0}^{\infty} e^{-\lambda_k t} e^{-\frac{x_{n+1}^2}{4t}}\frac{dt}{t^{1- \alpha}} \psi_k(x)\\
& =  2^{1+\alpha} \sum\limits_{k=1}^{\infty}(H, \psi_k)_{L^2(B_R(0), dV_\gamma)} \lambda_k^{-\alpha} (\sqrt{\lambda_k} x_{n+1})^{\alpha} K_{\alpha}(\sqrt{\lambda_k } x_{n+1}) \psi_k(x)\\
& = 2^{1+\alpha} c_{\alpha}^{-1}\tilde{u}^F(x,x_{n+1}).
\end{align*}
In the second line, in interchanging the summation and integration, we used that the integrand is absolutely convergent as by the compact support of $H$ and by integration by parts it holds 
\begin{align*}
|(H,\psi_k)_{L^2(B_R(0),dV_{\gamma})}| \leq \lambda_k^{-m}|((-\Delta_{\gamma})^{m} H, \psi_k)_{L^2(B_R(0),dV_{\gamma})}|
\leq \lambda_k^{-m} \|(-\D_{\gamma})^m H\|_{L^2(B_R(0),dV_{\gamma})}
\end{align*}
 and $\|\psi_k\|_{L^{\infty}(B_R(0))} \leq C\lambda_k^{\frac{n-1}{4}}$ \cite{Grieser01} for $\lambda_k \geq 1$.
In the penultimate equality, we used the integral identity \cite[10.32.10]{NIST}
\begin{align*}
K_{\alpha}(z) = \frac{1}{2} \left( \frac{1}{2} z\right)^{-\alpha}\int\limits_{0}^{\infty} e^{- \left( t +\frac{z^2}{4t} \right)}\frac{dt}{t^{1-\alpha}}
\end{align*}
together with a change of variables $\tilde{t} = \lambda_k t$. This implies claim from \eqref{eq:repr1}.

\emph{Step 2: Limit.} In this step, we pass to the limit $R \rightarrow \infty$ in \eqref{eq:repr1}. This relies on the maximum principle for the heat equation as well as the apriori estimates \eqref{eq:apriori}. On the one hand, by the a priori estimates from \eqref{eq:apriori}, the functions $\tilde{u}_R^F$ have a weak limit $\tilde{u}\in X_{\alpha}$ and this limit satisfies the weak form of the equation \eqref{eq:Neumann} tested against $\tilde\varphi \in C_0^{\infty}(\overline{\R^{n+1}_+})$: for a fixed test function $\tilde{\varphi} \in C_0^{\infty}(\overline{\R^{n+1}_+})$ and as $R \rightarrow \infty$ (in particular such that $\supp(\tilde\varphi) \Subset B_R(0)\times \R_+$ and $\supp(F) \subset B_R(0)$)
\begin{align*}
- \int\limits_{\R^n} H \sqrt{|\gamma|} \tilde{\varphi}|_{\R^n \times \{0\}} dx =  \int\limits_{\R^{n+1}_+} x_{n+1}^{1-2\alpha} \tilde{\gamma} \nabla \tilde{u}_R^F \cdot \nabla \tilde{\varphi} d(x,x_{n+1}) \rightarrow \int\limits_{ \R^{n+1}_+} x_{n+1}^{1-2\alpha} \tilde{\gamma} \nabla \tilde{u} \cdot \nabla \tilde{\varphi} d (x,x_{n+1}) .
\end{align*}
 By density, this also implies that the function $\tilde{u}$ is a solution to the weak equation tested with $\tilde\varphi \in \dot{H}^1(\R^{n+1}_+, x_{n+1}^{1-2\alpha})$. 

 Moreover, by the maximum principle, it holds that $p_t^{(-\D_\gamma^R)}(x,z)  \rightarrow p_t^{(-\D_\gamma)}(x,z) $ uniformly as $R \rightarrow \infty$ if $x\neq z$ \cite[Exercise 7.40]{Grigoryan_book_2009}. Thus, for $x \notin \supp(F)$ and $x_{n+1}\geq 0$, by dominated convergence, also the integral representations from \eqref{eq:repr1} converge to the desired one from \eqref{eq:repr_heat}.

Combining these two observations, we may pass to the limit $R \rightarrow \infty $ in \eqref{eq:repr1} and obtain the desired representation formula \eqref{eq:repr_heat}. By the heat kernel estimates in the whole space, for $x \notin \supp(F)$ and $x_{n+1}\geq 0$, by the same arguments as above, this integral exists in an $L^1$ sense and the $t$ and $z$ integration can be arbitrarily exchanged.
\end{proof}

Using the previous lemma, as in \cite{Ruland_2023}, we deduce analyticity properties in the normal variable for the function $\tilde{u}^F$.

\begin{lem}
\label{lem:normal_analyticity}
Let $\alpha \in (0,1)$ and let $\gamma$ be a metric satisfying the condition from Assumption \ref{assump:main}(b). Let $F\in C_0^{\infty}(\Omega_e)$ and $H= (-\D_\gamma) F$. Let $\tilde{u}^F \in X_{\alpha}$ be the weak solution of \eqref{eq:Neumann} with data $H$. Then,  there exists a constant $\tilde{c}>0$ such that for $x \in \Omega_e \setminus \supp(F)$ and $x_{n+1} \in (0,\tilde{c}\dist(x,\supp(F)))$ we have the absolutely convergent expansion
\begin{align*}
\tilde{u}^F(x,x_{n+1}) = \sum\limits_{j=0}^{\infty} C_j(x) x_{n+1}^{2j},
\end{align*}
with
\begin{align*}
C_j(x) := \hat{c}_{\alpha}\frac{4^{-j}}{j!} (-1)^j \int\limits_{\R^n} \int\limits_{0}^{\infty} p_t^{(-\Delta_\gamma)}(x,z) \frac{dt}{t^{1-\alpha+j}} H(z) dV_\gamma(z),
\end{align*}
where $\hat{c}_{\alpha} \neq 0$ is the constant from Lemma \ref{lem:representation}.
The coefficients satisfy
\begin{align*}
|C_j(x)| \leq  C(c \dist^2(x,\supp(F)))^{-\frac{n}{2}+  \alpha-j} j^{\frac{n}{2}-\alpha}
\end{align*}
for some $c>0$ and $C=C(F)>0$.
In particular, as a function of $x_{n+1}$ for $0\leq x_{n+1}< c\dist(x,\supp(F))/2$, with $0<c/2<\tilde{c}$, the function $\tilde{u}^F$ is analytic.
\end{lem}

In the lemma $\dist(x, \cdot)$ denotes the Euclidean distance function.

\begin{proof}
Our starting point is the expression \eqref{eq:repr_heat} from Lemma \ref{lem:representation}.
For $t\geq \epsilon>0$, we expand the factor $e^{-\frac{x_{n+1}^2}{4t}}$ around zero and pass to the limit $t \rightarrow 0$ by dominated convergence. 
More precisely, to this end, we use the heat kernel bounds as well as the fact that $x$ is outside of the support of $F$ which yields that for $x\in \Omega_e \setminus \supp(F)$, $x_{n+1}\geq 0$ and $\epsilon>0$ it holds that
\begin{align*}
\tilde{u}^F(x,x_{n+1}) 
&= \hat{c}_{\alpha} \lim\limits_{\epsilon \rightarrow 0} \int\limits_{\epsilon}^{\infty} \int\limits_{\R^n} p_t^{(-\D_{\gamma})}(x,z) H(z) dV_{\gamma}(z) e^{- \frac{x_{n+1}^2}{4t}} \frac{dt}{t^{1-\alpha}}\\
&= \hat{c}_{\alpha} \lim\limits_{\epsilon \rightarrow 0}  \int\limits_{\epsilon}^{\infty} \int\limits_{\R^n}  \sum\limits_{j=0}^{\infty} \frac{1}{j!} \left(- \frac{x_{n+1}^2}{4 t} \right)^j t^{\alpha-1}  p_t^{(-\D_{\gamma})}(x,z) H(z) dV_{\gamma}(z) dt
\end{align*}
The application of Fubini's theorem is justified by the absolute convergence of the integrals in \eqref{eq:repr_heat}.
We next seek to exchange summation and integration. For $x \in \Omega_e \setminus \supp(F)$ and $x_{n+1}\in (0, \bar{c} \dist(x,\supp(F)))$, with $\bar{c}>0$ a sufficiently small constant, this is achieved by dominated convergence with integrable majorant 
\begin{align*}
t^{-\frac{n}{2}-1+ \alpha}e^{- c\frac{|x-z|^2}{t} + \frac{x_{n+1}^2}{4t}}.
\end{align*}
Here we used the heat kernel estimates \eqref{eq_2_5_1} for $p_t^{(-\Delta_{\gamma})}$ as well as the fact that $x \in \Omega_e \setminus \supp(F)$. We hence obtain
\begin{align*}
\tilde{u}^F(x,x_{n+1}) 
&= \hat{c}_{\alpha}\lim\limits_{\epsilon \rightarrow 0}   \sum\limits_{j=0}^{\infty} \frac{(-1)^j 4^{-j}}{j!}  \int\limits_{\epsilon}^{\infty} \int\limits_{\R^n}   p_t^{(-\D_{\gamma})}(x,z) H(z) dV_{\gamma}(z) \frac{dt}{t^{-\alpha+1+j} } x_{n+1}^{2j}.
\end{align*}
It remains to exchange the limit $\epsilon \rightarrow 0$ and the summation and to deduce that
\begin{align*}
\tilde{u}^F(x,x_{n+1}) = \sum\limits_{j=0}^{\infty} C_j(x) x_{n+1}^{2j},
\end{align*}
with the coefficients $C_j(x)$ as in the lemma. In order to achieve this, it suffices to derive the claimed bounds for these coefficients $\int\limits_{\epsilon}^{\infty} \int\limits_{\R^n}   p_t^{(-\D_{\gamma})}(x,z) H(z) dV_{\gamma}(z) \frac{dt}{t^{-\alpha+1+j} }$ independently of $\epsilon>0$. The absolute convergence of the series then follows from the positive estimate on the radius of convergence. To this end we note that for $x\in \Omega_e \setminus \supp(F)$, $x_{n+1}\geq 0$, we have that $ p_t^{(-\Delta_\gamma)}(x,z) \frac{1}{t^{1-\alpha+j}} H(z) \in L^1(\R^n \times (0,\infty))$ which allows us to apply Fubini. Using this, by the heat kernel estimates from \eqref{eq_2_5_1}, for any $\epsilon>0$
\begin{align*}
&\left|\frac{4^{-j}}{j!} \int\limits_{\R^n} \int\limits_{\epsilon}^{\infty} p_t^{(-\Delta_\gamma)}(x,z) \frac{dt}{t^{1-\alpha+j}} H(z) dV_\gamma(z)\right|\\
& \leq C \frac{4^{-j}}{j!} \left| \int\limits_{\R^n} \int\limits_{0}^{\infty} e^{- \frac{c |x-z|^2}{t}}  \frac{dt}{t^{1- \alpha+j+\frac{n}{2}}} H(z) dV_\gamma(z) \right|\\
& \leq C (c \dist^2(x,\supp(F)))^{-\frac{n}{2}+\alpha +j} j^{\frac{n}{2}-\alpha},
\end{align*}
for some constant $C=C(F,n,\alpha)>0$.
Here we used that for $z \in \supp(F)$, $x \in \Omega_e \setminus \supp(F)$
\begin{align*}
 \int\limits_{0}^{\infty} e^{- \frac{c |x-z|^2}{t}}  \frac{dt}{t^{1- \alpha +j+\frac{n}{2}}}  
\leq C\left(c |x-z|^2 \right)^{-\frac{n}{2}+ \alpha -j} \Gamma(\frac{n}{2} - \alpha +j),
\end{align*}
together with the estimate (see \cite[formula (5.11.12)]{NIST} for a corresponding asymptotic identity)
\begin{align*}
\Gamma(\frac{n}{2}-\alpha + j) \leq C \Gamma(j) j^{\frac{n}{2}-\alpha} \mbox{ for } j \in \N
\end{align*}
and the fact that $\frac{4^{-j}}{j!} \Gamma(j) \leq C$ for all $j \in \N$.
This proves the claim.
\end{proof}

Combining the previous two lemmas, we conclude the uniqueness proof by invoking the proof variant from \cite{Ruland_2023} of the argument from \cite{FGKU_2021}. For completeness and the convenience of the reader, we include it here.

\begin{proof}[Proof of Proposition \ref{prop:conclude}]
We argue in two steps.\\

\emph{Step 1: Moment estimates.}
Combining the results from Proposition \ref{prop:source-to-sol}, Lemmas \ref{lem:representation} and \ref{lem:normal_analyticity} as well as $g_j = g$ in $\Omega_e$, our measurement data yield that for $H= (-\D_{g_j}) F = (-\D_g) F$ and $F \in C_0^{\infty}(\Omega_e)$, $j\in \{1,2\}$,
\begin{align*}
{u}_1^F(x) = \tilde{L}_{g_1}^{\Omega_e, \Omega_e}(F) = \tilde{L}_{g_2}^{\Omega_e, \Omega_e}(F) = {u}_2^F(x), \ x \in \Omega_e\setminus \supp(F),
\end{align*}
and
\begin{align}
\label{eq:repr_series}
\begin{split}
&\tilde{u}_j^F(x,x_{n+1}) = \sum\limits_{k=0}^{\infty} C_{k,j}(x) x_{n+1}^{2k}, \ j\in\{1,2\},\\
& \qquad \mbox{ for }\ x  \in (\Omega_e \setminus \supp(F)), \ x_{n+1}\in (0, \tilde{c} \dist(x,\supp(F)) ),
\end{split}
\end{align}
and $C_{k,j}(x)$ as in Lemma \ref{lem:normal_analyticity} with metric $g_j$. 
In particular, we directly infer that $C_{0,1}(x) = C_{0,2}(x)$ for all $x \in \Omega_e \setminus \supp(F)$.
Recalling that $g_j = g$ in $\Omega_e$ and  invoking the tangential regularity from Lemma \ref{lem:tangential_cont}, we obtain that for $m \in \{1,2,3, \dots, \}$
\begin{align*}
(-\D_g)^m{u}_1^F(x) = (-\D_g)^m{u}_2^F(x), \ x \in \Omega_e \setminus \supp(F).
\end{align*}
Now, by equation \eqref{eq:Neumann} combined with the tangential and normal regularities from Lemmas \ref{lem:tangential_cont} and \ref{lem:normal_analyticity}, we deduce that in a strong sense
\begin{align*}
(-\D_{g_j}) \tilde{u}_j^F(x) = x_{n+1}^{2\alpha-1} \p_{n+1} x_{n+1}^{1-2\alpha} \p_{n+1} \tilde{u}_j^F(x)
\end{align*}
for $x \in \Omega_e \setminus \supp(F)$.
 Recalling the series representation \eqref{eq:repr_series}, we hence arrive at
\begin{align*}
& C_{m,1}(x) \prod_{j=1}^m (2m- 2(j-1)) (2m-2\alpha-2(j-1)) \\
&= \lim\limits_{x_{n+1} \rightarrow 0} (x_{n+1}^{2\alpha -1} \p_{{n+1}} x_{n+1}^{1-2\alpha} \p_{{n+1}})^m \tilde{u}_1^F(x,x_{n+1})\\
&= (-\D_g)^m{u}_1^F(x)
 = (-\D_g)^m{u}_2^F(x)\\
&= \lim\limits_{x_{n+1} \rightarrow 0} (x_{n+1}^{2\alpha -1} \p_{{n+1}} x_{n+1}^{1-2\alpha} \p_{{n+1}})^m \tilde{u}_2^F(x)\\
& = C_{m,2}(x) \prod_{j=1}^m (2m- 2(j-1)) (2m-2\alpha -2(j-1)), \\
& \qquad \mbox{ for }  x  \in \Omega_e \setminus \supp(F),\ m \in \{1,2,3,\dots\}.
\end{align*}
In this computation differentiation and summation can be interchanged in the representation from \eqref{eq:repr_series} by the absolute convergence of the series for $x\in \Omega_e \setminus \supp(F)$ for $x_{n+1}\geq 0$ sufficiently small.

\emph{Step 2: Conclusion.}
By the argument from Step 1, it holds that for all $H = (-\D_g)F$ and $F \in C_0^{\infty}(\Omega_e)$
\begin{align}
\label{eq:coeff1}
 C_{m,1}(x)= C_{m,2}(x)  ,  \ x \in \Omega_e \setminus \supp(F), \ m \in \{0,1,2,3,\dots\}.
\end{align}
We rewrite this condition by using the commutation of the heat semi-group and the Laplace-Beltrami operator, by using the equation satisfied by the heat semi-group and by integrating by parts in time (which is justified by the heat kernel bounds from \eqref{eq_2_5_1})
\begin{align*}
& \int\limits_{\R^n} \int\limits_{0}^{\infty}  p_t^{(-\Delta_{g_1})}(x,z) (-\D_{g_1,z} ) F(z) \frac{dt}{t^{1-\alpha+m}}  dV_{g_1}(z)\\
& = \int\limits_{\R^n} \int\limits_{0}^{\infty}  (-\D_{g_1,z} ) p_t^{(-\Delta_{g_1})}(x,z)  F(z) \frac{dt}{t^{1-\alpha+m}}  dV_{g_1}(z)\\
 & = -\int\limits_{\R^n} \int\limits_{0}^{\infty}  \p_t p_t^{(-\Delta_{g_1})}(x,z)  F(z) \frac{dt}{t^{1-\alpha+m}}  dV_{g_1}(z)\\
 & = -(m+1-\alpha)\int\limits_{\R^n} \int\limits_{0}^{\infty}   p_t^{(-\Delta_{g_1})}(x,z)  F(z) \frac{dt}{t^{2-\alpha+m}}  dV_{g_1}(z)
\mbox{ for } x \in \Omega_e \setminus \supp(F).
\end{align*}
As a consequence, \eqref{eq:coeff1} implies that for all $m \in \{0,1,2,3,\dots\}$ and $\mbox{ for } x \in \Omega_e \setminus \supp(F)$, we have
\begin{align*}
\int\limits_{\R^n} \int\limits_{0}^{\infty}   p_t^{(-\Delta_{g_1})}(x,z)  F(z) \frac{dt}{t^{2-\alpha+m}}  dV_{g_1}(z)
= \int\limits_{\R^n} \int\limits_{0}^{\infty}   p_t^{(-\Delta_{g_2})}(x,z)  F(z) \frac{dt}{t^{2-\alpha+m}}  dV_{g_2}(z).
\end{align*}
Relying on an argument as in \eqref{eq_4_20} and following, this then implies that  for all $x\neq z$, $x,z \in \Omega_e$, $t>0$
\begin{align*}
p_t^{(-\Delta_{g_1})}(x,z) =  p_t^{(-\Delta_{g_2})}(x,z).
\end{align*}
This concludes the argument.
\end{proof}

The remainder of the proof of Theorem \ref{thm_main} follows by combining Lemmas \ref{lem:modified_data1}, Propositions \ref{prop:source-to-sol} and \ref{prop:conclude} together with a reduction to the wave equation and the boundary control method as in \cite{FGKU_2021} by using the completeness of $(\R^n,g_j)$, $j \in \{1,2\}$ (see the remark on this below Assumption \ref{assump:main}(b)). See Section \ref{sec:hk}.

\begin{appendix}

\section{Obstruction to uniqueness in the anisotropic fractional Calder\'on problem with external data}

\label{app_obstruction}

Let $g_1$ and $g_2$ be $C^\infty$ Riemannian metrics on $\mathbb{R}^n$,  $n \ge 2$, which agree with the Euclidean metric outside a compact set. Let $\Omega \subset \mathbb{R}^n$ be a bounded open set with $C^\infty$ boundary, and let $W_1, W_2 \subset \Omega_e$ be nonempty open subsets. Let $\Phi: (\mathbb{R}^n, g_1) \to (\mathbb{R}^n, g_2)$ be a $C^\infty$ Riemannian isometry fixing the exterior $\Omega_e$ of $\Omega$, i.e., $\Phi: \mathbb{R}^n \to \mathbb{R}^n$ is a $C^\infty$ diffeomorphism such that:
\begin{itemize}
    \item[(i)] $g_1 = \Phi^* g_2$ on $\mathbb{R}^n$, i.e.,
    \begin{equation}
    \label{eq_app_new_1}
    g_1(x)(t, s) = g_2(\Phi(x))(\Phi'(x)t, \Phi'(x)s), \quad x \in \mathbb{R}^n, \quad t, s \in \mathbb{R}^n,
    \end{equation}
    \item[(ii)] $\Phi(x) = x$ for all $x \in \Omega_e$.
\end{itemize}
Thus, it follows that $g_1 = g_2$ on $\Omega_e$ and $\Phi: \Omega \to \Omega$ is a diffeomorphism.

We need the following result, see also \cite[Theorem 4.2]{Ghosh_Uhlmann_2021}.
\begin{prop}
\label{prop_app_obstruction}
We have $\Lambda_{g_1}^{W_1, W_2} = \Lambda_{g_2}^{W_1, W_2}$ with $g_1 = \Phi^* g_2$.
\end{prop}

\begin{proof}
First, since $\Phi: (\mathbb{R}^n, g_1) \to (\mathbb{R}^n, g_2)$ is a $C^\infty$ Riemannian isometry, we have
\begin{equation}
\label{eq_app_new_2}
-\Delta_{g_1}(u \circ \Phi) = (-\Delta_{g_2} u) \circ \Phi \quad \text{in } \mathcal{D}'(\mathbb{R}^n),
\end{equation}
for all $u \in \mathcal{D}'(\mathbb{R}^n)$; see \cite[pp. 99-100]{GPR_book}.

We claim that the map
\[
U: L^2(\mathbb{R}^n; dV_{g_2}) \to L^2(\mathbb{R}^n; dV_{g_1}), \quad u \mapsto u \circ \Phi,
\]
is unitary. Indeed, we have
\begin{align*}
\|u \circ \Phi\|_{L^2(\mathbb{R}^n; dV_{g_1})}^2 &= \int_{\mathbb{R}^n} |u(\Phi(x))|^2 \sqrt{|g_1(x)|} \, dx \\
&= \int_{\mathbb{R}^n} |u(y)|^2 \sqrt{|g_1(\Phi^{-1}(y))|} \frac{dy}{|\det \Phi'(\Phi^{-1}(y))|} = \|u\|_{L^2(\mathbb{R}^n; dV_{g_2})}^2,
\end{align*}
showing the claim. Here, we have used that $g_1(x) = (\Phi'(x))^t g_2(\Phi(x)) \Phi'(x)$, which follows from \eqref{eq_app_new_1}, and made the change of variables $y = \Phi(x)$.

Rewriting \eqref{eq_app_new_2} as 
\[
-\Delta_{g_1} = U \circ (-\Delta_{g_2}) \circ U^{-1},
\]
and using the functional calculus of self-adjoint operators, we get 
\begin{equation}
\label{eq_app_new_3}
(-\Delta_{g_1})^\alpha = U \circ (-\Delta_{g_2})^\alpha \circ U^{-1}.
\end{equation}

Let $f \in C^\infty_0(W_1)$ and let $u^f_2 \in H^\alpha(\mathbb{R}^n)$ be the unique energy solution to the following exterior Dirichlet problem:
\[
\begin{cases}
(-\Delta_{g_2})^\alpha u^f_2 = 0 & \text{in} \quad \Omega, \\
u^f_2 = f & \text{in} \quad \Omega_e.
\end{cases}
\]

By Lemma \ref{lem_distrubutional_solution}, $u^f_2$ satisfies $(-\Delta_{g_2})^\alpha u^f_2 = 0$ in $\mathcal{D}'(\Omega)$. Thus, using \eqref{eq_app_new_3}, we get 
\begin{equation}
\label{eq_app_new_5}
0 = (-\Delta_{g_2})^\alpha u^f_2 = (-\Delta_{g_1})^\alpha (u^f_2 \circ \Phi) \circ \Phi^{-1} \quad \text{in} \ \mathcal{D}'(\Omega).
\end{equation}
Therefore, since $\Phi: \Omega \to \Omega$ is a $C^\infty$ diffeomorphism, using \eqref{eq_app_new_5} and Proposition \ref{prop_eq_2_6}, we see that $u^f_2 \circ \Phi \in H^\alpha(\mathbb{R}^n)$ is the unique energy solution to the exterior  Dirichlet problem:
\[
\begin{cases}
(-\Delta_{g_1})^\alpha (u^f_2 \circ \Phi) = 0 & \text{in} \quad \Omega, \\
u^f_2 \circ \Phi = f & \text{in} \quad \Omega_e.
\end{cases}
\]
It follows that 
\[
\Lambda_{g_1}^{W_1,W_2}(f) = ((-\Delta_{g_1})^\alpha (u^f_2 \circ \Phi))|_{W_2} = ((-\Delta_{g_2})^\alpha u^f_2)|_{W_2} = \Lambda_{g_2}^{W_1,W_2}(f).
\]
\end{proof}

\section{The fractional Laplace--Beltrami operator on the Euclidean space as a pseudodifferential operator. Proof of Theorem \ref{thm_pseudodiff_op}}

\label{sec_pseudodiff_op}

The purpose of this appendix is to provide a complete proof of Theorem \ref{thm_pseudodiff_op}, see also \cite[Theorem 4.7]{Strichartz_1983}. We rely on this proof to derive the Vishik--Eskin estimates for solutions to the exterior Dirichlet problem for the fractional Laplace--Beltrami operator in our setting, based on the theory developed in \cite{Hormander_1965-66, Grubb_2015, Eskin_book_1981, Vishik_Eskin_1965}. Before presenting the proof, we note that the origins of such results trace back to the seminal work \cite{Seeley_1966}, which, in particular, establishes that fractional Laplace--Beltrami operators on compact Riemannian manifolds without boundary, defined via functional calculus, are pseudodifferential operators (see Theorem \ref{thm_Seeley} below; see also \cite[Lemma 2.9]{Grubb_2015}). In \cite[Section 3]{Grubb_2020}, it is shown that fractional powers of second-order strongly elliptic operators on $\mathbb{R}^n$ are pseudodifferential operators, provided they coincide with $1 - \Delta$  near infinity. Additionally, \cite{Ammann_Lauter_Nistor_Vasy} investigates complex powers of elliptic, strictly positive pseudodifferential operators on non-compact manifolds.

We now proceed to prove Theorem \ref{thm_pseudodiff_op}. To begin, we first gather some necessary facts for the proof.

\subsection{The fractional Laplacian as a pseudodifferential operator on a compact manifold}

Let $(M,g)$ be a $C^\infty$ compact Riemannian manifold of dimension $n \geq 2$ without boundary, and let $P = -\Delta_g$ be the Laplace--Beltrami operator on $M$, associated with the metric $g$. The operator $P$ is self-adjoint on $L^2(M)$ with the domain $\mathcal{D}(P) = H^2(M)$, the standard Sobolev space on $M$, and it has a discrete spectrum $\sigma(P) \subset [0,\infty)$. 

The following result was established in \cite{Seeley_1966}.
\begin{thm}
\label{thm_Seeley}
Let $(M,g)$ be a $C^\infty$ compact Riemannian manifold without boundary, and let $0 < \alpha < 1$. Then $(-\Delta_g)^\alpha$, defined via functional calculus, satisfies $(-\Delta_g)^\alpha \in \Psi^{2\alpha}_{1,0}(M)$ and is a classical elliptic pseudodifferential operator on $M$. 
\end{thm}

\subsection{Resolvent bounds in Sobolev spaces}

In what follows, let $(M, g)$ be either a $C^\infty$ compact Riemannian manifold of dimension $n$ without boundary, or $(\mathbb{R}^n, g)$, where $g$ is a $C^\infty$ Riemannian metric on $\mathbb{R}^n$ that agrees with the Euclidean metric in a neighborhood of infinity. Assume further that $n \geq 2$, and let $P = -\Delta_g$ be the Laplace--Beltrami operator on $M$. 

For $s \in \mathbb{R}$, define the $L^2$-based Sobolev space on $M$ by
\begin{equation}
\label{eq_3_1_p}
H^s(M) = \{u \in \mathcal{S}'(M) : (1 + P)^{s/2} u \in L^2(M)\},
\end{equation}
equipped with the norm
\begin{equation}
\label{eq_3_2_p}
\|u\|_{H^s(M)} = \|(1 + P)^{s/2} u\|_{L^2(M)}.
\end{equation}
Here, in the case where $M$ is a compact manifold, we set $\mathcal{S}'(M) = \mathcal{D}'(M)$. The Bessel potential $(1 + P)^{s/2}$ is defined by the self-adjoint functional calculus. In the case where $(M, g) = (\mathbb{R}^n, g)$, with $g$ a $C^\infty$ Riemannian metric on $\mathbb{R}^n$ that agrees with the Euclidean metric in a neighborhood of infinity, the standard Sobolev space on $\mathbb{R}^n$ coincides with the one defined in \eqref{eq_3_1_p}, and the standard Sobolev norm is equivalent to the norm given in \eqref{eq_3_2_p}; see the discussion in Section~\ref{sec_Sobolev_spaces_direct_problem}.

We have the following basic result. In what follows, $ \sigma(P)$ denotes the spectrum of~$P$. 
\begin{prop}
\label{prop_resolvent_large_z}
Let $C_1, C_2 > 0$, and define 
\begin{equation}
\label{eq_3_3_p}
\Omega_{C_1,C_2} := \bigg\{z \in \mathbb{C} : \text{\emph{dist}}(z, \sigma(P)) \geq \frac{|z|}{C_1}, \, |z| \geq C_2\bigg\}.
\end{equation}
Then, for any $s \in \mathbb{R}$, the resolvent $(P - z)^{-1}$ satisfies the following estimates:  
\begin{equation}
\label{eq_3_4_p}
(P - z)^{-1} = \mathcal{O}_{s, C_1, C_2}(1): H^s(M) \to H^{s+2}(M),
\end{equation}
and 
\begin{equation}
\label{eq_3_5_p}
(P - z)^{-1} = \mathcal{O}_{s, C_1}\left(\frac{1}{|z|}\right): H^s(M) \to H^s(M),
\end{equation}
for all $z \in \Omega_{C_1, C_2}$.
\end{prop}

\begin{proof}
Let $u \in H^s(M)$. Then, using \eqref{eq_3_2_p}, the functional calculus for self-adjoint operators, and \eqref{eq_3_3_p}, we get
\begin{align*}
\|(P-z)^{-1}u\|_{H^{s+2}(M)}&=\|(1+P)^{s/2+1}(P-z)^{-1}(1+P)^{-s/2}(1+P)^{s/2}u\|_{L^2(M)}\\
&\le \|(1+P)(P-z)^{-1}\|_{\mathcal{L}(L^2(M), L^2(M))}\|u\|_{H^s(M)}\\
&\le \sup_{t\in \sigma(P)}\bigg|\frac{1+t}{t-z}\bigg|\|u\|_{H^s(M)}=\mathcal{O}_{C_1,C_2}(1)\|u\|_{H^s(M)},
\end{align*}
which shows \eqref{eq_3_4_p}.

Similarly, we have 
\begin{align*}
\|(P-z)^{-1}u\|_{H^{s}(M)}&=\|(1+P)^{s/2}(P-z)^{-1}u\|_{L^2(M)}\\
&\le \|(P-z)^{-1}\|_{\mathcal{L}(L^2(M), L^2(M))}\|u\|_{H^s(M)}\\
&\le \sup_{t\in \sigma(P)}\frac{1}{|t-z|}\|u\|_{H^s(M)}=\mathcal{O}_{C_1}\bigg(\frac{1}{|z|}\bigg)\|u\|_{H^s(M)},
\end{align*}
which establishes \eqref{eq_3_5_p}.
\end{proof} 

By interpolating between the bounds \eqref{eq_3_4_p} and \eqref{eq_3_5_p}, we obtain the following result.
\begin{cor}
Let $C_1,C_2 > 0$, and let $\Omega_{C_1,C_2}$ be given by \eqref{eq_3_3_p}. Then, for any $s \in \mathbb{R}$ and $0 \leq \theta \leq 1$, the resolvent $(P - z)^{-1}$ satisfies the estimate
\[
(P-z)^{-1} = \mathcal{O}_{s,C_1,C_2, \theta}\left(\frac{1}{|z|^{1-\theta}}\right): H^s(M) \to H^{s+2\theta}(M),
\]
for all $z \in \Omega_{C_1,C_2}$. 
In particular, for $\theta = \frac{1}{2}$, and for any $s \in \mathbb{R}$, we have
\begin{equation}
\label{eq_3_7_p}
(P-z)^{-1} = \mathcal{O}_{s,C_1,C_2}\left(\frac{1}{|z|^{\frac{1}{2}}}\right): H^s(M) \to H^{s+1}(M),
\end{equation}
for all $z \in \Omega_{C_1,C_2}$.
\end{cor}

We have the following analog of Proposition \ref{prop_resolvent_large_z} in the case when the spectral parameter $z$ belongs to a bounded set.  
\begin{prop}
\label{prop_resolvent_bounded_z}
Let $C > 0$, and define
\begin{equation} 
\label{eq_3_9_new_D_C}
D_C := \{z \in \mathbb{C} \setminus \mathbb{R} : |z| \leq C\}.
\end{equation} 
Then, for any $s \in \mathbb{R}$, the resolvent $(P - z)^{-1}$ satisfies the estimate
\begin{equation}
\label{eq_3_9_p}
(P-z)^{-1} = \mathcal{O}_{s,C}\bigg(\frac{1}{|\emph{\text{Im}}\,z|}\bigg): H^s(M) \to H^{s+2}(M),
\end{equation}
for all $z \in D_C$.
\end{prop}

\subsection{The Dynkin--Helffer--Sj\"ostrand formula}
Let $\psi\in C^\infty_0(\R)$ be such that $\psi=1$ near $0$, and let $-1<\beta<0$. To pass from resolvent estimates to estimates for the operator $(1-\psi)(P)P^{\beta}$  in the proof of Theorem \ref{thm_pseudodiff_op}, we shall use the following extension of the Dynkin--Helffer--Sj\"ostrand formula, see \cite[Theorem 8.1]{Dimassi_Sjostrand_book}. To state the result, we let  $f(t)=(1-\psi(t))t^{\beta}$, $t\ge 0$. Note that $f(t)=0$ near $t=0$, and $f\in (C^\infty\cap L^\infty)([0,\infty))$.  Let $\tilde \psi\in C^\infty_0(\C)$ be an almost holomorphic extension of $\psi$ such that $\tilde \psi=1$ in a complex neighborhood of $0$. Here we recall from \cite[page 93]{Dimassi_Sjostrand_book} that $\tilde \psi\in C^\infty_0(\C)$ is an  almost holomorphic extension of $\psi$ if for all $N\ge 0$, 
\begin{equation}
\label{eq_3_20_p}
|\bar\p\tilde \psi(z)|\le C_N |\text{Im}\, z|^N, \quad z\in\C,
\end{equation}
and 
\begin{equation}
\label{eq_3_21_p}
\tilde \psi|_{\R}=\psi. 
\end{equation}

\begin{prop} 
\label{prop_Dynkin_Helffer_Sjostrand_formula}
Let $-1<\beta<0$. We have in the sense of bounded linear operators on $L^2(M)$, 
\begin{equation}
\label{eq_3_22_p}
\begin{aligned}
f(P)&=(1-\psi(P))P^{\beta}\\
&=\frac{1}{2\pi i}\int_{\p \omega} (1-\tilde \psi(z))z^{\beta }(z-P)^{-1}dz+\frac{1}{\pi}\int_{\omega}(\bar\p \tilde \psi)(z)z^{\beta}(z-P)^{-1}L(dz).
\end{aligned}
\end{equation}
Here $z^\beta=e^{\beta \log z}$, where we take the principal branch of the logarithm, $\omega:=\{z\in \C: |\arg z|<\pi/4\}$,  the boundary $\p \omega$ is positively oriented, and $L(dz)$ is the Lebesgue measure on $\C$.  
\end{prop}

\begin{rem}
Note that the integration in the first integral in the right hand side of \eqref{eq_3_22_p} occurs over two rays  in the right half plane $\text{Re}\, z > 0$, bounded away from $z=0$, and the integration in the second integral in the right hand side of \eqref{eq_3_22_p} occurs over a compact subset of the right half plane $\text{Re}\, z > 0$. 
\end{rem}

\begin{proof} 
We shall follow the proof of \cite[Theorem 8.1]{Dimassi_Sjostrand_book}. To that end, we let $\tilde f(z)=(1-\tilde \psi(z))z^\beta\in C^\infty(\{z\in \C: \text{Re}\, z>0\})$, and consider $\tilde \omega:=\omega\cap \{z\in\C: |z|< R\}$, $R>0$ large. Then by Cauchy's integral formula, we get 
\begin{equation}
\label{eq_3_23_p}
\tilde f(\zeta)=\frac{1}{2\pi i}\int_{\p \tilde \omega}\frac{\tilde f(z)}{z-\zeta}dz-\frac{1}{\pi}
\int_{\tilde \omega} \frac{ (\bar \p \tilde f)(z)}{z-\zeta}L(dz), \quad \zeta\in \tilde \omega.
\end{equation}

Letting $R\to \infty$ in \eqref{eq_3_23_p}, we obtain that 
\begin{equation}
\label{eq_3_24_p}
\tilde f(\zeta)=\frac{1}{2\pi i}\int_{\p \omega}\frac{\tilde f(z)}{z-\zeta}dz-\frac{1}{\pi}\int_{\omega} \frac{ (\bar \p \tilde f)(z)}{z-\zeta}L(dz), \quad \zeta\in \omega.
\end{equation}
Indeed, this follows from the fact that as $\beta<0$, we have 
\[
\bigg| \int_{|z|=R, |\arg z|<\pi/4}\frac{(1-\tilde \psi(z))z^\beta}{z-\zeta}dz\bigg|\le \mathcal{O}(R^\beta)\to 0,
\]
as $R\to \infty$, and the fact that in  the second integral in the right hand side of \eqref{eq_3_23_p}, the integrand has a compact support, independent of $R$, thanks to $(\bar \p \tilde f)(z)=-(\bar \p \tilde \psi)(z)z^\beta$.

Restricting \eqref{eq_3_24_p} to the positive real axis, in view of \eqref{eq_3_21_p}, we get 
\begin{equation}
\label{eq_3_25_p}
 f(\lambda)=\frac{1}{2\pi i}\int_{\p \omega}\frac{\tilde f(z)}{z-\lambda}dz-\frac{1}{\pi}\int_{\omega} \frac{ (\bar \p \tilde f)(z)}{z-\lambda}L(dz), \quad \lambda> 0.
\end{equation}

As $f\in (C^\infty\cap L^\infty)([0,\infty))$, using the bounded functional calculus for self-adjoint operators, we get 
\begin{equation}
\label{eq_3_26_p}
(f(P)u,u)_{L^2(M)}=\int_0^\infty f(\lambda)d(E_\lambda u, u)_{L^2(M)},
\end{equation}
for all $u\in L^2(M)$. Here $d(E_\lambda u, u)_{L^2(M)}$ is the spectral measure of $P$. Substituting \eqref{eq_3_25_p} into \eqref{eq_3_26_p} and using Fubini's theorem, we obtain that 
\begin{equation}
\label{eq_3_27_p}
\begin{aligned}
(f(P)u,u)_{L^2(M)}=&\frac{1}{2\pi i}\int_0^\infty \int_{\p \omega}\frac{\tilde f(z)}{z-\lambda}dzd(E_\lambda u, u)_{L^2(M)}\\
&+\frac{1}{\pi}\int_0^\infty \int_{\omega} \frac{ (\bar \p \tilde \psi)(z)z^\beta}{z-\lambda}L(dz)  d(E_\lambda u, u)_{L^2(M)}\\
= &(Qu,u)_{L^(M)},
\end{aligned}
\end{equation}
where 
\[
Q=\frac{1}{2\pi i} \int_{\p \omega}\tilde f(z)(z-P)^{-1}dz+\frac{1}{\pi} \int_{\omega}(\bar \p \tilde \psi)(z) z^\beta(z-P)^{-1}L(dz).
\]
Note that an application of Fubini's theorem in the first integral in the right hand side of \eqref{eq_3_27_p} is justified thanks to the fact that $\beta<0$ and it is justified in the second integral due to \eqref{eq_3_20_p}. This shows  \eqref{eq_3_22_p}. 
\end{proof}

\subsection{Some useful facts about pseudodifferential operators}

In order to check that an operator is pseudodifferential, we shall use the following result, which is a direct consequence of \cite[Proposition 18.1.19]{Hormander_book_III}. 
\begin{prop}
\label{prop_Hormander_pseudo}
If $A: C_0^\infty(\R^n)\to C^\infty(\R^n)$ is a continuous linear map and for all $\chi_0\in C^\infty_0(\R^n)$, the operator $\chi _0A \chi_0\in \Psi^m_{1,0}(\R^n)$ with some $m\in \R$, then $A\in \Psi^m_{1,0}(\R^n)$. 
\end{prop}

\begin{proof}
Indeed, it follows from \cite[Proposition 18.1.19]{Hormander_book_III} that to check that  $A\in \Psi^m_{1,0}(\R^n)$, one has to show that for all $\chi,\tilde \chi\in C^\infty_0(\R^n)$, the operator $\chi A\tilde \chi\in \Psi^m_{1,0}(\R^n)$. Now letting $\chi_0\in C^\infty_0(\R^n)$ be such that $\chi_0=1$ near $\supp(\chi)\cup\supp(\tilde \chi)$, we may write $\chi A\tilde \chi=\chi (\chi_0A\chi_0)\tilde \chi$. Therefore, if one shows that $\chi_0A\chi_0\in  \Psi^m_{1,0}(\R^n)$ then $\chi A\tilde \chi\in \Psi^m_{1,0}(\R^n)$.
\end{proof}

Following \cite[page 28]{Grigis_Sjostrand_book}, the operator $A: C^\infty(M)\to \mathcal{D}'(M)$ is said to be smoothing if its Schwartz kernel  $K_A(x,y)\in C^\infty(M\times M)$. A smoothing operator $A$ satisfies $A\in \Psi^{m}_{1,0}(M)$ for all $m\in \R$, see 
\cite[page 28]{Grigis_Sjostrand_book}.
\begin{prop}
\label{prop_smoothing}
If for all 
$N\ge 0$, 
\begin{equation}
\label{eq_3_11_p}
A: H^{-N}(M)\to H^N(M)
\end{equation}
is bounded then $A$ is smoothing. 
\end{prop}

\begin{proof}
 Indeed, if \eqref{eq_3_11_p} holds for all $N\ge 0$, then $A$ can be extended to a continuous operator: $\mathcal{E}'(M)\to C^\infty(M)$.  By \cite[Exercise 1.11, page 17]{Grigis_Sjostrand_book}, the latter is equivalent to the fact that $K_A(x,y)\in C^\infty(M\times M)$. 
\end{proof}

\subsection{Proof of Theorem \ref{thm_pseudodiff_op}}

Let $(M_1,g_1)$ be $(\R^n,g_1)$, $n\ge 2$, where $g_1$ is a $C^\infty$ Riemannian metric on $\R^n$, which agrees with the Euclidean one in a neighborhood of infinity. Let $P_1=-\Delta_{g_1}$ be the Laplace--Beltrami operator on $M_1$, and  let $0<\alpha<1$. 

Let $\psi\in C^\infty_0(\R)$ be such that $\psi=1$ near $0$. We have the following observation. 
\begin{lem}
\label{lem_near_zero}
The operator $\psi(P_1)P_1^\alpha$ is smoothing on $M_1$. 
\end{lem}
\begin{proof}
Indeed, by Proposition \ref{prop_smoothing} it suffices to check that for all $N\ge 0$, 
\[
\psi(P_1)P_1^\alpha: H^{-N}(M_1)\to H^N(M_1)
\]
is bounded. To that end, letting $N\ge 0$, $u\in H^{-N}(M_1)$, and using \eqref{eq_3_2_p}, we get 
\begin{align*}
\|\psi(P_1)P_1^\alpha u\|_{H^N(M_1)}&=\|(1+P_1)^{N/2} \psi(P_1)P_1^\alpha (1+P_1)^{N/2} (1+P_1)^{-N/2}u\|_{L^2(M_1)}\\
&\le \|(1+P_1)^N \psi(P_1)P_1^\alpha \|_{\mathcal{L}(L^2(M_1), L^2(M_1))}\|u\|_{H^{-N}(M_1)}\\
&\le \sup_{t\in\sigma(P_1)}|(1+t)^N\psi(t)t^\alpha|\|u\|_{H^{-N}(M_1)}=\mathcal{O}_N(1)\|u\|_{H^{-N}(M_1)}.
\end{align*}
\end{proof}

In view of Lemma \ref{lem_near_zero}, we only have to show that $(1-\psi)(P_1)P_1^\alpha\in \Psi^{2\alpha}_{1,0}(M_1)$. In doing so, we shall use Proposition \ref{prop_Hormander_pseudo}. To that end, we first claim that 
\begin{equation}
\label{eq_4_1_p}
(1-\psi)(P_1)P_1^\alpha: C^\infty_0(M_1)\to C^\infty(M_1)
\end{equation}
is continuous. 
Indeed, we have $C^\infty_0(\R^n)\subset \mathcal{D}(P_1^k) = H^{2k}(M_1)$ for all $k\in \N$, and by the functional calculus of self-adjoint operators, $(1-\psi)(P_1)P_1^\alpha: \mathcal{D}(P_1^k)\to \mathcal{D}(P_1^{k-\alpha})$ continuously. Therefore, $(1-\psi)(P_1)P_1^\alpha : C^\infty_0(M_1)\to H^k(M_1)$ is continuous, for all $k\in\N$, and therefore, \eqref{eq_4_1_p} follows by the Sobolev embedding.

In view of Proposition \ref{prop_Hormander_pseudo}, it suffices to check that for all $\chi_0\in C^\infty_0(M_1)$,  
\begin{equation}
\label{eq_Hormander_check_p}
\chi_0(1-\psi)(P_1)P_1^\alpha\chi_0\in \Psi^{2\alpha}_{1,0}(M_1).
\end{equation}
 To that end, we let $\chi_0\in C^\infty_0(M_1)$ and let $B(0,R_0):=\{x\in \R^n: |x|<R_0\}$ be the Euclidean ball such that 
\[
\supp(\chi_0)\cup\supp(g_1-e)\subset B(0,R_0). 
\]
Here $e$ stands for the Euclidean metric on $\R^n$. Let $R>R_0$ and let $\Pi_R:=\{x\in\R^n: |x_j|<R,\, j=1,2\}$. We have $B(0,R_0) \Subset \Pi_R$.
Let us extend the metric $g_1|_{\Pi_R}$ periodically to all of $\R^n$ and denote this extension by $g_2\in C^\infty(\R^n)$, and let $M_2=\R^n/(2R\Z)^n$. Then $(M_2,g_2)$ is a $C^\infty$ compact Riemannian manifold without boundary. We can view $B(0,R_0)\subset M_2$, and let us set $P_2=-\Delta_{g_2}$ on $M_2$.  Note that $P_1=P_2$ on $B(0,R_0)$. 

We observe that $(1-\psi)(P_2)P_2^\alpha\in \Psi^{2\alpha}_{1,0}(M_2)$ as by Theorem \ref{thm_Seeley}, $P_2^\alpha\in \Psi^{2\alpha}_{1,0}(M_2)$ and $(1-\psi)(P_2)\in \Psi^{0}_{1,0}(M_2)$. Thus, to prove \eqref{eq_Hormander_check_p} it suffices to show that 
\begin{equation}
\label{eq_4_21_p}
\chi_0 (1-\psi)(P_1)P_1^\alpha\chi_0- \chi_0 (1-\psi)(P_2)P_2^\alpha\chi_0
\end{equation}
is smoothing on $M_1$. In doing so, we shall first prove a similar property for the difference of cutoff resolvents. In what follows $\sigma(P_j)$ stands for the spectrum of $P_j$, $j=1,2$.

\begin{lem}
Let $C_1, C_2 > 0$, and define
\[
\tilde{\Omega}_{C_1,C_2} := \bigg\{z \in \mathbb{C} : \text{\emph{dist}}(z,\sigma(P_1)\cup \sigma(P_2)) \geq \frac{|z|}{C_1}, \, |z| \geq C_2\bigg\}.
\]
Then, for any $s \in \mathbb{R}$ and $N, L \geq 0$, we have 
\begin{equation}
\label{eq_4_17_p}
\chi_0(P_1-z)^{-1}\chi_0-\chi_0(P_2-z)^{-1}\chi_0=\mathcal{O}_{s,N,L,C_1,C_2}\bigg(\frac{1}{|z|^L}\bigg): H^{s}(M_1)\to H^{s+N}(M_1),
\end{equation}
for all $z \in \tilde{\Omega}_{C_1,C_2}$.
Furthermore, let $C > 0$, and let $D_C$ be given by \eqref{eq_3_9_new_D_C}. Then, for any $s \in \mathbb{R}$ and $N \geq 0$, there exists $L = L(N) \geq 0$ such that 
\begin{equation}
\label{eq_4_18_p}
\chi_0(P_1-z)^{-1}\chi_0- \chi_0(P_2-z)^{-1}\chi_0=\mathcal{O}_{s,N,C}\bigg(\frac{1}{|\text{\emph{Im}}\, z|^{L}}\bigg): H^{s}(M_1)\to H^{s+N}(M_1),
\end{equation}
for all $z \in D_C$.
\end{lem}

\begin{proof}
Following \cite[Lemma 4.1]{Sjostrand_1997}, we shall first recall a telescoping formula for the resolvent of $P_1$. Let $z\in \C\setminus \sigma(P_1)$ and let $\chi_1\in C^\infty_0(B(0,R_0))$ be such that $\chi_1=1$ near $K:=\supp(\chi_0)$. We have
\begin{equation}
\label{eq_2_1_p}
\begin{aligned}
(P_1-z)^{-1}\chi_0=&\chi_1(P_1-z)^{-1}\chi_0+(1-\chi_1)(P_1-z)^{-1}\chi_0\\
=&\chi_1(P_1-z)^{-1}\chi_0+ (P_1-z)^{-1}(1-\chi_1)\chi_0+[(1-\chi_1),(P_1-z)^{-1}]\chi_0\\
=&\chi_1(P_1-z)^{-1}\chi_0+[(1-\chi_1),(P_1-z)^{-1}]\chi_0.
\end{aligned}
\end{equation}
Here 
\begin{equation}
\label{eq_2_2_p}
[(1-\chi_1),(P_1-z)^{-1}]\chi_0=-[\chi_1,(P_1-z)^{-1}]\chi_0=-(P_1-z)^{-1}[P_1,\chi_1](P_1-z)^{-1}\chi_0.
\end{equation}
Combining \eqref{eq_2_1_p} and \eqref{eq_2_2_p}, we get the following resolvent identity
\begin{equation}
\label{eq_2_3_p}
(P_1-z)^{-1}\chi_0=\chi_1(P_1-z)^{-1}\chi_0-(P_1-z)^{-1}[P_1,\chi_1](P_1-z)^{-1}\chi_0.
\end{equation}
Here, $[P_1,\chi_1]$ is a differential operator of order one on $M_1$ with coefficients in $C^\infty_0(B(0,R_0) \setminus K)$. Let $\chi_2 \in C^\infty_0(B(0,R_0) \setminus K)$ be such that $\chi_2 = 1$ near the support of the coefficients of $[P_1,\chi_1]$ and $\supp(\chi_2) \cap \supp(\chi_0) = \emptyset$. We can then write, similarly to \eqref{eq_2_3_p}, 
\begin{equation}
\label{eq_2_4_p}
(P_1-z)^{-1}[P_1,\chi_1] =\chi_2(P_1-z)^{-1}[P_1,\chi_1] -(P_1-z)^{-1}[P_1,\chi_2](P_1-z)^{-1}[P_1,\chi_1].
\end{equation}
Substituting \eqref{eq_2_4_p} into \eqref{eq_2_3_p} we get
\begin{equation}
\label{eq_2_5_p}
\begin{aligned}
(P_1-z)^{-1}\chi_0=&\chi_1(P_1-z)^{-1}\chi_0-  \chi_2(P_1-z)^{-1}[P_1,\chi_1]  (P_1-z)^{-1}\chi_0\\
&+ (P_1-z)^{-1}[P_1,\chi_2](P_1-z)^{-1}[P_1,\chi_1](P_1-z)^{-1}\chi_0.
\end{aligned}
\end{equation}
This procedure can be iterated as follows: let $L \in \N$ with $L \geq 3$, and let $\chi_3, \dots, \chi_L \in C^\infty_0(B(0,R_0) \setminus K)$ be such that $\chi_j = 1$ near the support of the coefficients of $[P_1,\chi_{j-1}]$ and $\supp(\chi_j) \cap \supp(\chi_0) = \emptyset$ for $j = 3, \dots, L$.
 We have the iterated resolvent identity, 
\begin{equation}
\label{eq_2_6_p}
\begin{aligned}
&(P_1-z)^{-1}\chi_0=\chi_1(P_1-z)^{-1}\chi_0-  \chi_2(P_1-z)^{-1}[P_1,\chi_1]  (P_1-z)^{-1}\chi_0\\
&+ \chi_3(P_1-z)^{-1}[P_1,\chi_2](P_1-z)^{-1}[P_1,\chi_1](P_1-z)^{-1}\chi_0\\
& -\dots \\
&+(-1)^{L-1} \chi_L(P_1-z)^{-1}[P_1,\chi_{L-1}](P_1-z)^{-1}\cdots(P_1-z)^{-1}[P_1,\chi_1](P_1-z)^{-1}\chi_0\\
& +(-1)^L\times\\
& (P_1-z)^{-1}[P_1,\chi_L](P_1-z)^{-1}[P_1,\chi_{L-1}]\cdots (P_1-z)^{-1}[P_1,\chi_1](P_1-z)^{-1}\chi_0.
\end{aligned}
\end{equation}
Note that if we apply $\chi_0$ on the left in \eqref{eq_2_6_p}, all of the terms in the right hand side, except for the first and the last ones, drop out. 

Now let $z\in \C\setminus (\sigma(P_1)\cup \sigma(P_2))$, and let $\chi_1\in C^\infty_0(B(0,R_0))$ be such that $\chi_1=1$ near $K=\supp(\chi_0)$. We have the following resolvent identity analogous to \eqref{eq_2_3_p},
\begin{equation}
\label{eq_2_7_p}
(P_2-z)^{-1}\chi_0=\chi_1(P_1-z)^{-1}\chi_0-(P_2-z)^{-1}[P_1,\chi_1](P_1-z)^{-1}\chi_0.
\end{equation}
 Note that thanks to the cutoff functions $\chi_0$ and $\chi_1$, supported in the common region $B(0,R_0)\subset M_1\cap M_2$, the identity \eqref{eq_2_7_p}   holds in the sense of linear continuous operators: $L^2(M_2)\to \mathcal{D}(P_2)$.  
To check \eqref{eq_2_7_p}, it suffices to apply the injective operator $P_2-z: \mathcal{D}(P_2)\to L^2(M_2)$, and use that $P_2\chi_1=P_1\chi_1$. 

Let $L \in \N$ with $L \geq 2$, and let $\chi_2, \dots, \chi_L \in C^\infty_0(B(0,R_0) \setminus K)$ be such that $\chi_j = 1$ near the support of the coefficients of $[P_1,\chi_{j-1}]$ and $\supp(\chi_j) \cap \supp(\chi_0) = \emptyset$ for $j = 2, \dots, L$. By analogy with \eqref{eq_2_6_p}, we obtain that
\begin{equation}
\label{eq_2_8_p}
\begin{aligned}
&(P_2-z)^{-1}\chi_0=\chi_1(P_1-z)^{-1}\chi_0-  \chi_2(P_1-z)^{-1}[P_1,\chi_1]  (P_1-z)^{-1}\chi_0\\
&+ \chi_3(P_1-z)^{-1}[P_1,\chi_2](P_1-z)^{-1}[P_1,\chi_1](P_1-z)^{-1}\chi_0\\
& -\dots \\
&+(-1)^{L-1}\chi_L(P_1-z)^{-1}[P_1,\chi_{L-1}](P_1-z)^{-1}\cdots(P_1-z)^{-1}[P_1,\chi_1](P_1-z)^{-1}\chi_0\\
& +(-1)^{L}\times\\
& (P_2-z)^{-1}[P_1,\chi_L](P_1-z)^{-1}[P_1,\chi_{L-1}]\cdots (P_1-z)^{-1}[P_1,\chi_1](P_1-z)^{-1}\chi_0.
\end{aligned}
\end{equation}
Applying $\chi_0$ on the left in \eqref{eq_2_8_p}, we get 
\begin{equation}
\label{eq_2_9_p}
\begin{aligned}
\chi_0&(P_1-z)^{-1}\chi_0-\chi_0(P_2-z)^{-1}\chi_0=(-1)^{L+1}\times \\
 & \chi_0(P_2-z)^{-1}[P_1,\chi_L](P_1-z)^{-1}[P_1,\chi_{L-1}]\cdots (P_1-z)^{-1}[P_1,\chi_1](P_1-z)^{-1}\chi_0.
\end{aligned}
\end{equation}

Let us now proceed to show the estimate \eqref{eq_4_17_p}. First, since $[P_1,\chi_j]$ is a differential operator of order one on $M_1$, we have, for any $s \in \mathbb{R}$,
\begin{equation}
\label{eq_4_10_new_p}
[P_1,\chi_j] = \mathcal{O}_s(1): H^{s+1}(M_1) \to H^s(M_1),
\end{equation}
$j = 1, \dots, L$. Using that $\chi_0: H^s(M_1)\to H^s(M_1)$, \eqref{eq_3_7_p}, and \eqref{eq_4_10_new_p}, we see that for all $z \in \tilde{\Omega}_{C_1,C_2}$, 
\begin{equation}
\label{eq_4_11_p}
[P_1,\chi_1](P_1-z)^{-1}\chi_0=\mathcal{O}_{s,C_1,C_2}\bigg(\frac{1}{|z|^\frac{1}{2}}\bigg): H^s(M_1)\to H^s(M_1).
\end{equation}
Thus, in view of \eqref{eq_4_11_p}, and with the help of \eqref{eq_3_4_p}, we obtain from \eqref{eq_2_9_p} that for any $L\in \N$ and $s\in\R$, 
\begin{equation}
\label{eq_4_12_p}
\chi_0(P_1-z)^{-1}\chi_0-\chi_0(P_2-z)^{-1}\chi_0= \mathcal{O}_{s,L,C_1,C_2}\bigg(\frac{1}{|z|^\frac{L}{2}}\bigg): H^s(M_1)\to H^{s+2}(M_1),
\end{equation}
for all $z \in \tilde{\Omega}_{C_1,C_2}$.

On the other hand, using  \eqref{eq_3_4_p}, we get 
\[
[P_1,\chi_1](P_1-z)^{-1}\chi_0=\mathcal{O}_{s,C_1,C_2}(1): H^s(M_1)\to H^{s+1}(M_1),
\]
for all  $z \in \tilde{\Omega}_{C_1,C_2}$, 
and therefore, it follows from  \eqref{eq_2_9_p} that for  any $N\in \N$ and $s\in \R$, 
\begin{equation}
\label{eq_4_13_p}
\chi_0(P_1-z)^{-1}\chi_0-\chi_0(P_2-z)^{-1}\chi_0=\mathcal{O}_{s,N,C_1,C_2}(1): H^s(M_1)\to H^{s+N+2}(M_1),
\end{equation}
 for all $z\in \tilde{\Omega}_{C_1,C_2}$. 

Interpolating between \eqref{eq_4_12_p} and \eqref{eq_4_13_p}, and varying $N$ and $L$, we obtain that for any $s\in \R$, $N,L\ge 0$,
\[
\chi_0(P_1-z)^{-1}\chi_0-\chi_0(P_2-z)^{-1}\chi_0= \mathcal{O}_{s,L,N,C_1,C_2}\bigg(\frac{1}{|z|^{L}}\bigg): H^s(M_1)\to H^{s+N}(M_1),
\]
for all $z\in \tilde{\Omega}_{C_1,C_2}$. Thus, the bound \eqref{eq_4_17_p} follows. 

Let us next show \eqref{eq_4_18_p}. To that end, using \eqref{eq_4_10_new_p} and \eqref{eq_3_9_p}, we obtain, for all $z \in D_C$,
\[
[P_1,\chi_1](P_1-z)^{-1}\chi_0 = \mathcal{O}_{s,C}\bigg(\frac{1}{|\text{Im}\, z|}\bigg): H^s(M_1) \to H^{s+1}(M_1).
\]
 Hence, it follows from  \eqref{eq_2_9_p} with the help of \eqref{eq_3_9_p} that for any $N\in \N$ and $s\in \R$, 
\begin{equation}
\label{eq_4_16_p}
\chi_0(P_1-z)^{-1}\chi_0-\chi_0(P_2-z)^{-1}\chi_0= \mathcal{O}_{s,N,C}\bigg(\frac{1}{|\text{Im}\, z|^{N+1}}\bigg): H^s(M_1)\to H^{s+N+2}(M_1),
\end{equation}
for all $z\in D_C$. The bound  \eqref{eq_4_18_p} follows. 
\end{proof}

\begin{lem}
We have that the operator in \eqref{eq_4_21_p} is smoothing on $M_1$. 
\end{lem}

\begin{proof}
First, we would like to pass from the operator in \eqref{eq_4_21_p} with $0<\alpha<1$ to the same operator with $\alpha$ replaced by $\beta := \alpha-1 < 0$. In doing so,  we  write
\begin{equation}
\label{eq_4_22_p}
\begin{aligned}
\chi_0 (1-\psi)(P_j)&P_j^\alpha\chi_0
=\chi_0 (1-\psi)(P_j)P_j^{\alpha-1}\chi_0P_j+
\chi_0 (1-\psi)(P_j)P_j^{\alpha-1}[P_j,\chi_0]\\
&=\chi_0 (1-\psi)(P_j)P_j^{\alpha-1}\chi_0P_j\chi_1+
\chi_0\chi_1 (1-\psi)(P_j)P_j^{\alpha-1}\chi_1[P_j,\chi_0],
\end{aligned}
\end{equation}
for $j=1,2$. Here $\chi_1\in C^\infty_0(B(0,R_0))$ is such that $\chi_1=1$ near $\supp(\chi_0)$, and we used that $\chi_0P_j(1-\chi_1)=0$, and $(1-\chi_1)[P_j,\chi_0]=0$. 
In view of \eqref{eq_4_21_p} and \eqref{eq_4_22_p} it suffices to show that both operators
\begin{equation}
\label{eq_4_23_p}
\chi_0 (1-\psi)(P_1)P_1^{\beta}\chi_0-\chi_0 (1-\psi)(P_2)P_2^{\beta}\chi_0
\end{equation}
and 
\begin{equation}
\label{eq_4_24_p}
\chi_1 (1-\psi)(P_1)P_1^{\beta}\chi_1- \chi_1 (1-\psi)(P_2)P_2^{\beta}\chi_1
\end{equation}
are smoothing on $M_1$. The operators  in \eqref{eq_4_23_p} and \eqref{eq_4_24_p} are similar and we shall only show that the operator in \eqref{eq_4_23_p} is smoothing on $M_1$. 

In doing so, we use Proposition \ref{prop_Dynkin_Helffer_Sjostrand_formula} and get  
\begin{equation}
\label{eq_4_25_p}
\chi_0 (1-\psi)(P_1)P_1^{\beta}\chi_0-\chi_0 (1-\psi)(P_2)P_2^{\beta}\chi_0=I_1+I_2, 
\end{equation}
where 
\begin{align*}
I_1=\frac{1}{2\pi i}\int_{\p \omega} (1-\tilde \psi(z))z^{\beta }\bigg(\chi_0(z-P_1)^{-1}\chi_0- \chi_0(z-P_2)^{-1}\chi_0\bigg)dz\\
I_2=\frac{1}{\pi}\int_{\omega}(\bar\p \tilde \psi)(z)z^{\beta}\bigg(\chi_0(z-P_1)^{-1}\chi_0- \chi_0(z-P_2)^{-1}\chi_0\bigg) L(dz).
\end{align*}
Here $\tilde \psi\in C^\infty_0(\C)$ is an almost holomorphic extension of $\psi$ such that $\tilde \psi=1$ in a complex neighborhood of $0$, $\omega:=\{z\in \C: |\arg z|<\pi/4\}$, and the boundary $\p \omega$ is positively oriented.  

We now claim that $I_2$ is smoothing on $M_1$. Indeed, noting that the integration in $I_2$ is over a compact set away from $0$, and using \eqref{eq_3_20_p} and \eqref{eq_4_18_p}, we get that for all $N\ge 0$, there exists $L=L(N)>0$ such that 
\[
\|I_2\|_{\mathcal{L}(H^{-N}(M_1), H^{N}(M_1))}\le \mathcal{O}(1)\int_{\omega}\frac{|z|^\beta}{|\text{Im}\,z|^L}|\text{Im}\, z|^{\infty} L(dz)<\infty.
\]
The claim follows.

We next claim that the operator $I_1$ is smoothing. Indeed, by \eqref{eq_4_17_p}, we obtain that for all $N \geq 0$,
\[
\|I_1\|_{\mathcal{L}(H^{-N}(M_1), H^{N}(M_1))} \leq \mathcal{O}(1) \int_{\p \omega\cap\{|z|\ge c>0\}} \frac{|z|^\beta}{|z|^{L_1}} |dz| < \infty.
\]
Here, $L_1 \geq 1$ is arbitrary. The claim follows.
\end{proof}

We have therefore proved that $P_1^\alpha\in \Psi^{2\alpha}_{1,0}(M_1)$ and it follows from the arguments above that the pseudodifferential operator $P_1^\alpha$ is classical, with the principal symbol of $P_1^\alpha$ given by 
\[
\bigg(\sum_{j,k=1}^n g_1^{jk}(x)\xi_j\xi_k\bigg)^\alpha, \quad (x,\xi)\in T^*M_1\setminus 0. 
\]
This implies the ellipticity of $P_1^\alpha$.  The proof of Theorem \ref{thm_pseudodiff_op} is complete.

\section{Higher regularity for the variable coefficient fractional Laplacian with mixed Dirichlet-Neumann conditions}
\label{sec:reg}

In this section, we deduce higher regularity for the mixed Dirichlet-Neumann problem \eqref{eq:mixed} for the variable coefficient fractional Laplacian in $C^{\infty}$ domains as stated in Lemma \ref{lem:VE}. We view this as an elliptic extension perspective on Vishik-Eskin type estimates \cite{Vishik_Eskin_1965}. We formulate and deduce these estimates as a consequence of a local interpretation of the variable coefficient fractional Laplacian. In the whole section, we suppose that the condition from Assumption \ref{assump:main}(b) holds.

\subsection{Higher tangential regularity following \cite{KM07}}
We present the desired higher regularity results claimed in Lemma \ref{lem:VE}(i). Contrary to \cite{KM07} we directly work in $L^2$ based spaces which leads to a slight loss in the estimates (reflected in the presence of the $\epsilon$ in Lemma \ref{lem:VE}).
We consider the case of the half-space $\Omega = \R^{n}_{-}$ only, the general case follows from a transformation of the domain. We note that this transformation preserves the structure of the metric $\tilde{g}$ in being of the form as in Assumption \ref{assump:main}(b) since the change of coordinates does not affect the $x_{n+1}$ variable, the domain $\Omega \subset \R^n$ being purely tangentially defined in view of the extension. The key technical result is an $L^2$-based variant of \cite[Lemma 2.5]{KM07} which relies on an iteration of (one-sided) finite difference arguments. More precisely, we consider the following variant of \eqref{eq:mixed}:

\begin{align}
\label{eq:mixed_full}
\begin{split}
\nabla \cdot x_{n+1}^{1-2\alpha} \tilde{g} \nabla \tilde{v} & = 0 \mbox{ in } \R^{n+1}_+,\\
\tilde{v} & = f_D \mbox{ on } \R^n_+ \times \{0\},\\
\lim\limits_{x_{n+1} \rightarrow 0} x_{n+1}^{1-2\alpha} \p_{n+1} \tilde{v} & = f_N \mbox{ on } \R^n_-.
\end{split}
\end{align}
Here we have used the notation $\R^n_+:=\{x\in \R^n: x_n>0\}$ and $\R^n_{-}:= \{x\in \R^n: \ x_n \leq 0\}$ and assume that $f_D \in H^{\alpha+\delta}(\R^n_+)$, $f_N\in H^{-\alpha + \delta}(\R^n_-)$ for some $\delta \in [0,1/2)$. 
By considering a Dirichlet extension of $f_D$ (see the proof of Proposition \ref{prop:higher_reg} below), we may further assume that $f_D= 0$ and that only $f_N$ is nontrivial (see the argument in the proof of Proposition \ref{prop:higher_reg} below). The weak solution notion is then modified compared to Definition \ref{defi:sol} by adding in condition (i) a right hand side of the form $\langle \sqrt{|g|} f_N,\tilde{v}|_{\R^n \times \{0\}}\rangle_{{H}^{-\alpha}(\R^n), {H}^{\alpha}(\R^n)}$, i.e., by requiring that
\begin{align}
\label{eq:mixed_weak}
\int\limits_{\R^{n+1}_+} x_{n+1}^{1-2\alpha} \tilde{g} \nabla \tilde{v} \cdot \nabla \tilde{\varphi} d(x,x_{n+1}) = - \langle \sqrt{|g|} f_N,\tilde{v}|_{\R^n \times \{0\}}\rangle_{{H}^{-\alpha}(\R^n), {H}^{\alpha}(\R^n)}
\end{align}
for all $\tilde{\varphi} \in X_{\alpha,0,\Omega_e}$ with $\Omega_e := \R^{n}_+ \times \{0\}$ in the outlined half-space setting. For this solution notion, we then have the following regularity improvement result.

\begin{lem}
\label{lem:improve}
Let $\alpha \in (0,1)$ and let $\tilde{v} \in \dot{H}^{1}(\R^{n+1}_+, x_{n+1}^{1-2\alpha})$ with $\tilde{v}|_{\R^n \times \{0\}} \in H^{\alpha+\nu}(\R^{n})$ for some $\nu \in [0,1/2)$ be a weak solution to \eqref{eq:mixed_full} with data $f_D = 0$, $f_N\in H^{-\alpha + \delta}(\R^n)$ for some $\delta \in [0,1/2)$. Let $\beta = \frac{\delta + \nu}{2}$.
Then, for any $\mu>0$ there exists a constant $C= C(\alpha, g,n)>0$ such that
\begin{align*}
\sup\limits_{h \in (0,1)} h^{-\beta}\|\tau_{j,h} \tilde{v}|_{\R^n \times \{0\}} \|_{\dot{H}^{\alpha}(\R^n)} 
& \leq C(\mu^{-1}\|\sqrt{|g|} f_N\|_{H^{-\alpha+\delta}(\R^n)} + \|x_{n+1}^{\frac{1-2\alpha}{2}} \nabla \tilde{v}\|_{L^2(\R^{n+1}_+)})\\
& \quad  + C\mu \|\tilde{v}|_{\R^n \times \{0\}} \|_{H^{\alpha+\nu}(\R^n)}.
\end{align*}
\end{lem}

We note that the lemma is only of interest if $\delta>\nu$. In this case $\beta > \nu$ and the lemma indeed encodes a tangential regularity gain when combined with Lemma \ref{lem:Sob_charact}. 

\begin{proof}
For $h \in (0,1)$ we consider the finite difference quotient $\tau_{j,h}\tilde{v}(x,x_{n+1}):= \tilde{v}(x+ h e_j,x_{n+1}) - \tilde{v}(x,x_{n+1})$, where $e_j$ denotes the $j$-th canonical unit vector and $j \in \{1,\dots,n\}$ and insert $\tilde{\varphi}(x,x_{n+1}):=- \tau_{j,h}\tilde{v}(x,x_{n+1})$ as a test function into the weak formulation \eqref{eq:mixed_weak} of \eqref{eq:mixed_full}. We note that $\tau_{j,h} \tilde{v}(x,0) = 0$ for $x_n\geq 0$ is an admissible test function; in particular it satisfies the necessary boundary conditions ensuring that $\tau_{j,h}\tilde{v}(x,x_{n+1}) \in X_{\alpha,0,\Omega_e}$. Hence, we obtain
\begin{align}
\label{eq:reg1}
\begin{split}
&\left|\int\limits_{\R^{n+1}_+} x_{n+1}^{1-2\alpha} \tilde{g} \nabla \tilde{v} \cdot \tau_{j,h} \nabla \tilde{v} d (x,x_{n+1})  \right|
 = \left| \langle f_N,\sqrt{|g|}\tau_{j,h} \tilde{v}|_{\R^n \times \{0\}}\rangle_{{H}^{-\alpha}(\R^n), {H}^{\alpha}(\R^n)} \right|\\
& \leq C h^{2\beta} \|\sqrt{|g|} f_N\|_{H^{-\alpha+\delta}(\R^n)}\sup\limits_{h \in (0,1)} h^{-2\beta}\| \tau_{j,h} \tilde{v}|_{\R^n \times \{0\}}\|_{H^{\alpha-\delta}(\R^n)}.
\end{split}
\end{align}
The last estimate follows from Lemma \ref{lem:Sob_charact}(iii) (which is an adaptation of \cite[Lemma 2.2]{KM07}). It remains to bound the left hand side and to use the trace inequality. To this end, writing out the finite difference operators and invoking the symmetry of $\tilde{g}$ implies (as in the proof of \cite[Theorem 2.3]{KM07}) that 
\begin{align}
\label{eq:reg2}
\begin{split}
-\int\limits_{\R^{n+1}_+} x_{n+1}^{1-2\alpha} \tilde{g} \nabla \tilde{v} \cdot \tau_{j,h} \nabla \tilde{v} d (x,x_{n+1})
&= \frac{1}{2}\int\limits_{\R^{n+1}_+} x_{n+1}^{1-2\alpha} \tilde{g} \nabla \tau_{j,h}\tilde{v} \cdot \nabla  \tau_{j,h}\tilde{v} d(x,x_{n+1}) \\
& \quad - \frac{1}{2}\int\limits_{\R^{n+1}_+} x_{n+1}^{1-2\alpha} \sum\limits_{k,\ell = 1}^{n+1} \tilde{g}_{k \ell} \tau_{j,h}\left( \p_{k} \tilde{v}    \p_{\ell} \tilde{v}\right) d (x,x_{n+1}).
\end{split}
\end{align}
As in \cite{KM07} this is a consequence of the following computation: For $j \in \{1,\dots,n\}$,
\begin{align*}
& \sum\limits_{k,\ell=1}^{n+1} \tilde{g}_{k \ell}(x) \p_k \tilde{v}(x,x_{n+1}) (\p_{\ell} \tilde{v}(x,x_{n+1})- \p_{\ell} \tilde{v}(x + h e_j, x_{n+1}))\\
& = \sum\limits_{k,\ell=1}^{n+1} \tilde{g}_{k \ell}(x)( \p_k \tilde{v}(x,x_{n+1}) - \p_k \tilde{v}(x + h e_j, x_{n+1}) ) (\p_{\ell} \tilde{v}(x,x_{n+1})- \p_{\ell} \tilde{v}(x + h e_j, x_{n+1}))\\
& \quad + \sum\limits_{k,\ell=1}^{n+1} \tilde{g}_{k \ell}(x)  \p_k \tilde{v}(x + h e_j, x_{n+1})  (\p_{\ell} \tilde{v}(x,x_{n+1})- \p_{\ell} \tilde{v}(x + h e_j, x_{n+1}))\\
& = \frac{1}{2}\sum\limits_{k,\ell=1}^{n+1} \tilde{g}_{k \ell}(x) (\tau_{j,h}\p_k \tilde{v})(x,x_{n+1})  (\tau_{j,h}\p_{\ell} \tilde{v})(x,x_{n+1})\\
& \quad - \frac{1}{2}\sum\limits_{k,\ell=1}^{n+1} \tilde{g}_{k \ell}(x) \left(   \p_k \tilde{v}(x,  x_{n+1})  (\tau_{j,h}\p_{\ell} \tilde{v})(x,x_{n+1}) + \p_k \tilde{v}(x + h e_j, x_{n+1}) (\tau_{j,h} \p_{\ell} \tilde{v})(x , x_{n+1}) \right),
\end{align*}
together with the symmetry of $\tilde{g}_{k \ell}$.
Using the regularity of $\tilde{g}$ and combining \eqref{eq:reg1} and \eqref{eq:reg2}, we thus arrive at
\begin{align*}
\frac{\lambda}{2} \| x_{n+1}^{\frac{1-2\alpha}{2}} \nabla \tau_{j,h} \tilde{v} \|^2_{L^2(\R^{n+1}_+)}
& \leq \frac{1}{2}\int\limits_{\R^{n+1}_+} x_{n+1}^{1-2\alpha} \tilde{g} \nabla \tau_{j,h}\tilde{v} \cdot \nabla  \tau_{j,h}\tilde{v} d(x,x_{n+1}) \\
&\leq \left| \frac{1}{2}\int\limits_{\R^{n+1}_+} x_{n+1}^{1-2\alpha} \sum\limits_{k,\ell =1}^{n+1} (\tau_{j,h}\tilde{g}_{k \ell}) \left( \p_{k} \tilde{v}    \p_{\ell} \tilde{v}\right) d (x,x_{n+1}) \right|\\
& \qquad + \left| \langle \sqrt{|g|} f_N,\tau_{j,h} \tilde{v}|_{\R^n \times \{0\}}\rangle_{{H}^{-\alpha}(\R^n), {H}^{\alpha}(\R^n)}\right| \\
& \leq C h^{2\beta} \|\sqrt{|g|} f_N\|_{H^{-\alpha+\delta}(\R^n)}\sup\limits_{h \in (0,1)} h^{-2\beta} \| \tau_{j,h} \tilde{v}|_{\R^n \times \{0\}}\|_{H^{\alpha-\delta}(\R^n)} \\
& \quad + C h  \| x_{n+1}^{\frac{1-2\alpha}{2}} \nabla \tilde{v} \|^2_{L^2(\R^{n+1}_+)}.
\end{align*}
Here we have estimated 
\begin{align*}
|\tau_{j,h}\tilde{g}_{k \ell}| \leq C h,
\end{align*}
and used $\lambda>0$ to denote the ellipticity constant of $\tilde{g}$.
Therefore, using that $\beta \in (0,1/2)$,
\begin{align*}
\lambda \| x_{n+1}^{\frac{1-2\alpha}{2}} \nabla \tau_{j,h} \tilde{v} \|^2_{L^2(\R^{n+1}_+)}
&\leq Ch^{2\beta} \|\sqrt{|g|} f_N\|_{H^{-\alpha+\delta}(\R^n)}\sup\limits_{h \in (0,1)} h^{-2\beta} \| \tau_{j,h} \tilde{v}|_{\R^n \times \{0\}}\|_{H^{\alpha-\delta}(\R^n)} \\
& \quad + C h  \| x_{n+1}^{\frac{1-2\alpha}{2}} \nabla \tilde{v} \|^2_{L^2(\R^{n+1}_+)}\\
&\leq C  h^{2\beta} \left(\|\sqrt{|g|} f_N\|_{H^{-\alpha+\delta}(\R^n)}\sup\limits_{h \in (0,1)} h^{-2\beta} \| \tau_{j,h} \tilde{v}|_{\R^n \times \{0\}}\|_{H^{\alpha-\delta}(\R^n)} \right. \\
& \quad \left. +   \| x_{n+1}^{\frac{1-2\alpha}{2}} \nabla \tilde{v} \|^2_{L^2(\R^{n+1}_+)} \right).
\end{align*}
By the trace estimate \eqref{eq:trace}, we hence infer that for all $j \in \{1,\dots,n\}$ and $h>0$
\begin{align*}
\lambda \| \tau_{j,h} \tilde{v}|_{\R^n \times \{0\}} \|^2_{\dot{H}^{\alpha}(\R^{n})}
&\leq C  h^{2\beta} \left(\|\sqrt{|g|} f_N\|_{H^{-\alpha+\delta}(\R^n)}\sup\limits_{h \in (0,1)} h^{-2\beta} \| \tau_{j,h} \tilde{v}|_{\R^n \times \{0\}}\|_{H^{\alpha-\delta}(\R^n)} \right.\\
& \quad \left. +   \| x_{n+1}^{\frac{1-2\alpha}{2}} \nabla \tilde{v} \|^2_{L^2(\R^{n+1}_+)} \right).
\end{align*}
By Young's inequality, we then obtain that for any $\mu>0$
\begin{align}
\label{eq:almost_final}
\begin{split}
\lambda \| \tau_{j,h} \tilde{v}|_{\R^n \times \{0\}} \|^2_{\dot{H}(\R^{n})}
&\leq C  h^{2\beta} \left(C \mu^{-1}\|\sqrt{|g|} f_N\|_{H^{-\alpha+\delta}(\R^n)}^2 \right. \\
& \left. \quad + C^{-1} \mu \sup\limits_{h \in (0,1)} h^{-4\beta} \| \tau_{j,h} \tilde{v}|_{\R^n \times \{0\}}\|_{H^{\alpha-\delta}(\R^n)}^2 \right.\\
& \quad \left. +   \| x_{n+1}^{\frac{1-2\alpha}{2}} \nabla \tilde{v} \|^2_{L^2(\R^{n+1}_+)} \right).
\end{split}
\end{align}
We note that by the assumptions and Lemma \ref{lem:Sob_charact}(iii),
\begin{align*}
\sup\limits_{h \in (0,1)} h^{-2\beta}\| \tau_{j,h} \tilde{v}|_{\R^n \times \{0\}}\|_{H^{\alpha-\delta}(\R^n)} \leq C \|\tilde{v}|_{\R^n \times \{0\}} \|_{H^{\alpha+\nu}(\R^n)} < \infty.
\end{align*}
Dividing \eqref{eq:almost_final} by $h^{2\beta}$ and taking the sup in $h \in (0,1)$ in this estimate results in 
 \begin{align*}
\frac{\lambda}{2} \sup\limits_{ h \in (0,1)}  h^{-2 \beta }\| \tau_{j,h} \tilde{v}|_{\R^n \times \{0\}} \|^2_{\dot{H}^{\alpha}(\R^{n})}
&\leq C \left(\mu^{-1}\|\sqrt{|g|} f_N\|_{H^{-\alpha+\delta}(\R^n)}^2  +   \| x_{n+1}^{\frac{1-2\alpha}{2}} \nabla \tilde{v} \|^2_{L^2(\R^{n+1}_+)} \right)\\
& \quad + C \mu   \|\tilde{v}|_{\R^n \times \{0\}} \|_{H^{\alpha+\nu}(\R^n)}^2,
\end{align*}
for any $\mu>0$,
which is the desired estimate.
\end{proof}

As a consequence of Lemma \ref{lem:improve}, we obtain the desired higher regularity result by an iteration argument.

\begin{prop}
\label{prop:higher_reg}
Let $\alpha \in (0,1)$ and let $\tilde{v} \in \dot{H}^{1}(\R^{n+1}_+, x_{n+1}^{1-2\alpha})$ with $\tilde{v}|_{\R^n \times \{0\}}  \in H^{\alpha}(\R^{n})$ being a weak solution to \eqref{eq:mixed_full} with data $f_D \in H^{\alpha+\delta}(\R^n)$, $f_N\in H^{-\alpha + \delta}(\R^n)$ for some $\delta \in [0,1/2)$. Then, we have that for any $\epsilon \in (0,\delta)$ there is a constant $C=C(\epsilon, \delta,n)>0$ such that
\begin{align*}
\|\tilde{v}|_{\R^n \times \{0\}} \|_{H^{\alpha+\delta-\epsilon}(\R^n)} \leq C(\|f_D\|_{H^{\alpha+\delta}(\R^n)} + \|\sqrt{|g|} f_N\|_{H^{-\alpha+\delta}(\R^n)}  + \|\tilde{v}|_{\R^n \times \{0\}} \|_{H^{\alpha}(\R^n)}).
\end{align*}
\end{prop}

\begin{proof}
Let us first assume that $f_D = 0$.
Using Lemma \ref{lem:improve} together with Lemma \ref{lem:Sob_charact}, we will iteratively improve the regularity of $\tilde{v}|_{\R^n \times \{0\}}$. As an auxiliary result, we recall that by energy estimates (see Proposition \ref{prop:Neumann} for the first bound)
\begin{align}
\label{eq:aux_Neumann}
\|x_{n+1}^{\frac{1-2\alpha}{2}} \nabla \tilde{v}\|_{L^2(\R^{n+1}_+)} \leq C \|\sqrt{|g|} f_N\|_{H^{-\alpha}(\R^n)} \leq C \|\sqrt{|g|}f_N\|_{H^{-\alpha+\delta}(\R^n)}.
\end{align}
We next choose $\tilde{\delta} \in (\delta, 1/2)$.
In order to infer the desired result, for $k \in \N$, $k\geq 1$, we iteratively apply Lemma \ref{lem:improve}. We begin by considering $\nu=\nu_1 := 0$, $\beta= \beta_1 := \frac{\tilde{\delta}}{2}$. In this case, Lemma \ref{lem:improve} yields that for any $\mu>0$
\begin{align*}
\sup\limits_{h \in (0,1)} h^{-\beta_1}\|\tau_{j,h} \tilde{v}|_{\R^n \times \{0\}} \|_{\dot{H}^{\alpha}(\R^n)} 
& \leq C(\mu^{-1}\|\sqrt{|g|} f_N\|_{H^{-\alpha+\delta}(\R^n)} + \|x_{n+1}^{\frac{1-2\alpha}{2}} \nabla \tilde{v}\|_{L^2(\R^{n+1}_+)})\\
& \quad  + C\mu \|\tilde{v}|_{\R^n \times \{0\}} \|_{H^{\alpha}(\R^n)}.
\end{align*}
Let $\epsilon \in (0,\delta)$. We next apply Lemma \ref{lem:Sob_charact} with $\epsilon_1 := \epsilon/2$ and add a corresponding $L^2$ contribution to both sides of the inequality, thereby producing the inhomogeneous norm on the left-hand side, to obtain that
\begin{align*}
\| \tilde{v}|_{\R^n \times \{0\}} \|_{H^{\alpha+\beta_1 - \epsilon_1}(\R^n)} 
& \leq C(\mu^{-1}\|\sqrt{|g|} f_N\|_{H^{-\alpha+\delta}(\R^n)} + \|x_{n+1}^{\frac{1-2\alpha}{2}} \nabla \tilde{v}\|_{L^2(\R^{n+1}_+)})\\
& \quad  + C(\mu + 1) \|\tilde{v}|_{\R^n \times \{0\}} \|_{H^{\alpha}(\R^n)}.
\end{align*}
Since all contributions on the right hand side of this estimate are finite, we may iterate this argument. Setting for $k\in \N$, $\nu:=\nu_k := \sum\limits_{j=1}^{k-1} \tilde{\delta} 2^{-j} $, $\beta:=\beta_k := \tilde{\delta}/2\left(1+ \sum\limits_{j=1}^{k-1}2^{-j} \right) - \frac{\epsilon}{2}\left( \sum\limits_{j=1}^{k-1}2^{-j} \right)$ and $\epsilon_k := \epsilon 2^{-k}$, we iteratively obtain that for $\mu>0$ arbitrary 
\begin{align*}
\|  \tilde{v}|_{\R^n \times \{0\}} \|_{H^{\alpha + \beta_k - \epsilon_k}(\R^{n})}
&\leq C_k \left(\mu^{-1}\|\sqrt{|g|} f_N\|_{H^{-\alpha+\delta}(\R^n)} +   \| x_{n+1}^{\frac{1-2\alpha}{2}} \nabla \tilde{v} \|_{L^2(\R^{n+1}_+)} \right)\\
& \quad  + C_{k}(\mu + k) \|\tilde{v}|_{\R^n \times \{0\}} \|_{H^{\alpha}(\R^n)}.
\end{align*}

Since $\tilde{\delta}>\delta$, and since $\beta_k \rightarrow \tilde{\delta}-\epsilon/2> \delta - \epsilon$, there exists $k_0 \in \N$ such that $\alpha + \beta_{k_0} - \epsilon_{k_0} > \alpha + \delta - \epsilon$ and, thus, for $\mu>0$,
\begin{align*}
\|  \tilde{v}|_{\R^n \times \{0\}} \|_{H^{\alpha + \delta - \epsilon}(\R^n)}
&\leq  \|  \tilde{v}|_{\R^n \times \{0\}} \|_{H^{\alpha + \beta_{k_0} - \epsilon_{k_0}}(\R^{n})}\\
&\leq C_{k_0} \left(\mu^{-1}\|\sqrt{|g|} f_N\|_{H^{-\alpha+\delta}(\R^n)} +   \| x_{n+1}^{\frac{1-2\alpha}{2}} \nabla \tilde{v} \|_{L^2(\R^{n+1}_+)} \right)\\
& \quad  + C_{k_0}(\mu + k_0) \|\tilde{v}|_{\R^n \times \{0\}} \|_{H^{\alpha}(\R^n)}.
\end{align*}
Fixing $\mu>0$ and recalling \eqref{eq:aux_Neumann} then concludes the proof in the case of $f_D = 0$.\\

For $f_D \neq 0$, we split $\tilde{v} = \tilde{v}_1 + \tilde{v}_2$ with $\tilde{v}_1 \in \dot{H}^{1}(\R^{n+1}_+, x_{n+1}^{1-2\alpha})$ being a solution to 
\begin{align}
\label{eq:Dirichlet_CS}
\begin{split}
\nabla \cdot x_{n+1}^{1-2\alpha} \tilde{g} \nabla \tilde{v}_1 & = 0 \mbox{ in } \R^{n+1}_+,\\
\tilde{v}_1 & = f_D \mbox{ in } \R^n.
\end{split}
\end{align}
This problem can be treated with standard difference quotient methods in the tangential direction (similarly as in Lemma \ref{lem:tangential_cont} but with the final estimates in $L^2$- instead of $L^{\infty}$-based spaces) relying on the finite difference characterization of Sobolev spaces. Such an argument leads to the conclusion that $v_1 \in H^{\alpha+\delta-\epsilon}(\R^n)$ for any $\epsilon \in (0,\delta)$. Inserting the function $v_1$ with its regularity properties into the resulting equation for $v_2$ gives rise to Neumann data $ f_N\in H^{-\alpha+\delta-\epsilon}(\R^n)$ for any $\epsilon \in (0,\delta)$. 
The function $\tilde{v}_2$ hence solves a problem in which $f_D =0$ on $\R^n_+ \times \{0\}$ but with non-trivial Neumann data on $\R^n_{-} \times \{0\}$. This concludes the proof.
\end{proof}

\end{appendix}

\section*{Acknowledgements}

K.K. is sincerely grateful to Andr\'as Vasy for many helpful discussions. K.K. would also like to thank Teemu Saksala for helpful discussions related to the paper \cite{HLOS_2018}. The research of K.K. is partially supported by the National Science Foundation (DMS 2408793). A.R. gratefully acknowledges support by the Deutsche Forschungsgemeinschaft (DFG, German Research Foundation) under Germany’s Excellence Strategy -- EXC-2047/1 -- 390685813. J.S. acknowledges that the IMB receives support from the EIPHI Graduate School (contract ANR-17-EURE-0002). The research of G.U. is partially supported by NSF, a Robert R. and Elaine F. Phelps Endowed Professorship at the University of Washington.


\begin{thebibliography}{99}

\bibitem{Agranovich_book_2015}
Agranovich, M., \emph{Sobolev spaces, their generalizations and elliptic problems in smooth and Lipschitz domains}, Springer Monographs in Mathematics. Springer, Cham, 2015.

\bibitem{Ammann_Lauter_Nistor_Vasy}
Ammann, B., Lauter, R., Nistor, V., Vasy, A., \emph{Complex powers and non-compact manifolds}, 
Comm. Partial Differential Equations \textbf{29} (2004), no.5--6, 671--705. 

\bibitem{ATW18}
Arendt, W., Ter Elst, A. F., Warma, M., \emph{Fractional powers of sectorial operators via the Dirichlet-to-Neumann operator}, Comm. Partial Differential Equations \textbf{43} (2018), no. 1, 1--24.


\bibitem{Aronson_1967}
Aronson, D., \emph{Bounds for the fundamental solution of a parabolic equation}, 
Bull. Amer. Math. Soc. \textbf{73} (1967), 890--896. 


\bibitem{BGS15}
Banica, V., del Mar González, M.,  Sáez, M., \emph{Some constructions for the fractional Laplacian on noncompact manifolds}, Rev. Mat. Iberoam. \textbf{31} (2015), no. 2, 681--712.

\bibitem{Baers_Covi_Ruland_2024}
 Baers, H., Covi, G., R\"uland, A., \emph{On instability properties of the fractional Calder\'on problem}, preprint 2024, \textsf{https://arxiv.org/abs/2405.08381}.

\bibitem{Belishev_1987}
Belishev, M., \emph{An approach to multidimensional inverse problems for the wave equation}, Dokl. AN SSSR \textbf{297} (1987) 524--527 (in Russian).


\bibitem{Belishev_Kurylev_1992}
Belishev, M., Kurylev Y., \emph{To the reconstruction of a Riemannian manifold via its spectral
data (BC-method)}, Comm. Partial Differential Equations, \textbf{17} (1992), 767--804.


\bibitem{Bhattacharyya_Ghosh_Uhlmann_2021}
Bhattacharyya, S.,  Ghosh, T.,  Uhlmann, G., \emph{Inverse problems for the fractional-Laplacian with lower order non-local perturbations}, Trans. Amer. Math. Soc. \textbf{374} (2021), no. 5, 3053--3075.



\bibitem{BWZ_2017}
Biccari, U., Warma, M.,  Zuazua, E., \emph{Local elliptic regularity for the Dirichlet fractional Laplacian},  Adv. Nonlinear Stud. \textbf{17} (2017), no. 2, 387--409.

\bibitem{Brasco_Gomez-Castro_Vazquez_2021}
Brasco, L.,  G\'{o}mez-Castro, D., V\'{a}zquez, J., \emph{Characterisation of homogeneous fractional Sobolev spaces}, 
Calc. Var. Partial Differential Equations \textbf{60} (2021), no. 2, Paper No. 60, 40 pp.


\bibitem{Brezis_2011}
Brezis, H., \emph{Functional analysis, Sobolev spaces and partial differential equations}, Springer, New York, 2011.

\bibitem{BCdPS13}
Br\"andle, C., Colorado, E.,  de Pablo, A.,  S\'{a}nchez, U., \emph{A concave-convex elliptic problem involving the fractional Laplacian}, 
Proc. Roy. Soc. Edinburgh Sect. A \textbf{143} (2013), no. 1, 39--71.


\bibitem{CS07}
Caffarelli, L.,  Silvestre, L., \emph{An extension problem related to the fractional Laplacian}, 
Comm. Partial Differential Equations \textbf{32} (2007), no. 7-9, 1245--1260.


\bibitem{Caffarelli_Stinga_2016}
Caffarelli, L.,  Stinga, P., \emph{Fractional elliptic equations, Caccioppoli estimates and regularity}, Ann. Inst. H. Poincar\'{e} Anal. Non Lin\'{e}aire \textbf{33} (2016), no. 3, 767--807. 

\bibitem{Calderon_1980}
Calder\'on, A., \emph{On an inverse boundary value problem}, Seminar on Numerical Analysis and its Applications to Continuum Physics (Rio de Janeiro, 1980), pp. 65--73, Soc. Brasil. Mat., Rio de Janeiro, 1980.

\bibitem{Cekic_Lin_Ruland_2020}
Ceki\'{c}, M.,  Lin, Y.-H., R\"uland, A., \emph{The Calder\'on problem for the fractional Schr\"odinger equation with drift}, Calc. Var. Partial Differential Equations \textbf{59} (2020), no. 3, Paper No. 91.

\bibitem{CG11}
Chang, S.-Y.~A., del Mar González, M., \emph{Fractional Laplacian in conformal geometry}, Advances in Mathematics,
Volume 226, Issue 2, 2011, pp. 1410-1432.

\bibitem{Chavel_book_1984}
 Chavel, I.,  \emph{Eigenvalues in Riemannian geometry}, Pure and Applied Mathematics, 115. Academic Press, Inc., Orlando, FL, 1984.
 
\bibitem{Chien_2023} 
 Chien, C. K., \emph{An inverse problem for fractional connection Laplacians}, J. Geom. Anal. \textbf{33} (2023), no. 12, Paper No. 375, 23 pp.
 
\bibitem{Choulli_Ouhabaz_2023} 
Choulli, M.,  Ouhabaz, E.,  \emph{Fractional anisotropic Calder\'on problem on complete Riemannian manifolds}, Communications in Contemporary Mathematics, to appear.

\bibitem{Covi_2020}
Covi, G., \emph{Inverse problems for a fractional conductivity equation}, 
Nonlinear Anal. \textbf{193} (2020), 111418.

\bibitem{Covi_2020_2}
Covi, G., \emph{An inverse problem for the fractional Schr\"odinger equation in a magnetic field}, Inverse Problems \textbf{36} (2020), no. 4, 045004.


\bibitem{Covi_quasilocal}
Covi, G., \emph{Uniqueness for the fractional Calder\'on problem with quasilocal perturbations}, SIAM J. Math. Anal. \textbf{54} (2022), no. 6, 6136--6163. 


\bibitem{Covi_de_Hoop_Salo_2024}
Covi, G.,  de Hoop, M.,  Salo, M., \emph{Geometrical optics for the fractional Helmholtz equation and applications to inverse problems}, preprint 2024, \textsf{https://arxiv.org/abs/2412.14698v1}.


\bibitem{Covi_Garcia-Ferrero_Ruland_preprint}
Covi, G.,  Garcia-Ferrero, M.,   R\"uland, A.,  \emph{On the Calder\'on problem for nonlocal Schr\"odinger equations with homogeneous, directionally antilocal principal symbols},  J. Differential Equations \textbf{341} (2022), 79--149.

\bibitem{Covi_Ghosh_Ruland_Uhlmann_2023}
Covi, G.,  Ghosh, T.,  R\"uland, A., Uhlmann, G., \emph{A reduction of the fractional Calder\'on problem to the local Calder\'on problem by means of the Caffarelli--Silvestre extension}, preprint 2023, \textsf{https://arxiv.org/abs/2305.04227}.


\bibitem{Covi_M_R_Uhlmann_preprint}
Covi, G., M\"onkk\"onen, K.,  Railo, J., Uhlmann, G., \emph{The higher order fractional Calder\'on problem for linear local operators: uniqueness}, Adv. Math. \textbf{399} (2022), Paper No. 108246, 29 pp.


\bibitem{Covi_Railo_Tyni_Zimmermann_2024}
Covi, G.,  Railo, J., Tyni, T., Zimmermann, P., \emph{Stability estimates for the inverse fractional conductivity problem}, 
SIAM J. Math. Anal. \textbf{56} (2024), no. 2, 2456--2487.

\bibitem{GPR_book}
Craioveanu, M., Puta, M.,  Rassias, T., \emph{Old and new aspects in spectral geometry}, 
Mathematics and its Applications, 534. Kluwer Academic Publishers, Dordrecht, 2001.

\bibitem{Davies_book}
Davies, E. B., \emph{Heat kernels and spectral theory}, 
Cambridge Tracts in Mathematics, 92. Cambridge University Press, Cambridge, 1989.



\bibitem{Nezza_Palatucci_Valdinoci_2012}
Di Nezza, E., Palatucci, G.,  Valdinoci, E., \emph{Hitchhiker's guide to the fractional Sobolev spaces}, 
Bull. Sci. Math. \textbf{136} (2012), no. 5, 521--573.

\bibitem{Dimassi_Sjostrand_book}
 Dimassi, M., Sj\"ostrand, J., \emph{Spectral asymptotics in the semi-classical limit}, London Mathematical Society Lecture Note Series, 268. Cambridge University Press, Cambridge, 1999.
 
\bibitem{DKSaloU_2009} 
Dos Santos Ferreira, D.,  Kenig, C., Salo, M., Uhlmann, G.,  \emph{Limiting Carleman weights and anisotropic inverse problems},  Invent. Math. \textbf{178} (2009), no. 1, 119--171.


\bibitem{DKurylevLS_2016}
Dos Santos Ferreira, D., Kurylev, Y., Lassas, M., Salo, M., \emph{The Calder\'on problem in transversally anisotropic geometries},  J. Eur. Math. Soc. (JEMS) \textbf{18} (2016), no. 11, 2579--2626. 

 

\bibitem{Eskin_book_1981}
Eskin, G., \emph{Boundary value problems for elliptic pseudodifferential equations}, Amer. Math. Soc., Providence, RI, 1981.

\bibitem{Eskin_book}
Eskin, G., \emph{Lectures on linear partial differential equations}, Graduate Studies in Mathematics, 123. American Mathematical Society, Providence, RI, 2011.

\bibitem{FF13}
Fall, M., Felli, V., \emph{Unique continuation properties for relativistic Schr\"odinger operators with a singular potential}, 
Discrete Contin. Dyn. Syst. \textbf{35} (2015), no. 12, 5827--5867.

\bibitem{FF14}
Fall, M., Felli, V., \emph{Unique continuation property and local asymptotics of solutions to fractional elliptic equations},
Comm. Partial Differential Equations \textbf{39} (2014), no. 2, 354--397.




\bibitem{Feizmohammadi_2021}
Feizmohammadi, A., \emph{Fractional Calder\'on problem on a closed Riemannian manifold}, Trans. Amer. Math. Soc. \textbf{377} (2024), no. 4, 2991--3013.

\bibitem{FGKU_2021}
Feizmohammadi, A., Ghosh, T., Krupchyk, K.,  Uhlmann, G., \emph{Fractional anisotropic Calder\'on problem on closed Riemannian manifolds}, J. Differential Geom., to appear. 


\bibitem{FKU_2024}
Feizmohammadi, A., Krupchyk, K.,  Uhlmann, G., \emph{Calder\'on problem for fractional Schr\"odinger operators on closed Riemannian manifolds}, preprint 2024, \textsf{https://arxiv.org/abs/2407.16866}. 


\bibitem{Feizmohammadi_Lin_2024}
Feizmohammadi, A.,  Lin, Y.-H., \emph{Entanglement principle for the fractional Laplacian with applications to inverse problems}, preprint 2024, \textsf{https://arxiv.org/abs/2412.13118v1}.

\bibitem{Felsinger_Kassmann_Voigt_2015}
Felsinger, M., Kassmann, M., Voigt, P., \emph{The Dirichlet problem for nonlocal operators},  Math. Z. \textbf{279} (2015), no. 3-4, 779--809.



\bibitem{Ghosh_Lin_Xiao_2017}
Ghosh, T.,  Lin, Y.-H.,  Xiao, J., \emph{The Calder\'on problem for variable coefficients nonlocal elliptic operators},  Comm. Partial Differential Equations \textbf{42} (2017), no. 12, 1923--1961.

\bibitem{Ghosh_Ruland_Salo_Uhlmann_2020}
Ghosh, T.,  R\"uland, A.,  Salo, M., Uhlmann, G., \emph{Uniqueness and reconstruction for the fractional Calder\'on problem with a single measurement},  J. Funct. Anal. \textbf{279} (2020), no. 1, 108505. 


\bibitem{Ghosh_Salo_Uhlmann_2020}
Ghosh, T.,  Salo, M.,  Uhlmann, G.,  \emph{The Calder\'on problem for the fractional Schr\"odinger equation},  Anal. PDE \textbf{13} (2020), no. 2, 455--475.

\bibitem{Ghosh_Uhlmann_2021}
Ghosh, T., Uhlmann, G., \emph{The Calder\'on problem for nonlocal operators}, preprint, \textsf{https://arxiv.org/abs/2110.09265}.



\bibitem{Grigis_Sjostrand_book}
Grigis, A.,  Sj\"ostrand, J., \emph{Microlocal analysis for differential operators.
An introduction.},  Cambridge University Press, Cambridge, 1994.

\bibitem{Grieser01}
Grieser, D. \emph{Uniform bounds for eigenfunctions of the Laplacian on manifolds with boundary},
Communications in Partial Differential Equations, \textbf{27} (7-8) (2002), 1283--1299.




\bibitem{Grigoryan_book_2009}
Grigor'yan, A., \emph{Heat kernel and analysis on manifolds}, American Mathematical Society, Providence, RI; International Press, Boston, MA, 2009.


\bibitem{Grubb_book}
Grubb, G., \emph{Distributions and operators}, Graduate Texts in Mathematics, Springer, New York, 2009. 


\bibitem{Grubb_2014}
Grubb, G., \emph{Local and nonlocal boundary conditions for $\mu$-transmission and fractional elliptic pseudodifferential operators}, Anal. PDE \textbf{7} (2014), no.7, 1649--1682.

\bibitem{Grubb_2015}
Grubb, G., \emph{Fractional Laplacians on domains, a development of H\"ormander's theory of $\mu$--transmission pseudodifferential operators}, 
Adv. Math. \textbf{268} (2015), 478--528. 

\bibitem{Grubb_2020}
Grubb, G., \emph{Exact Green's formula for the fractional Laplacian and perturbations}, 
Math. Scand. \textbf{126} (2020), no. 3, 568--592. 

\bibitem{Guillarmou_Hassell_Krup_2020}
Guillarmou, C., Hassell, A., Krupchyk, K., \emph{Eigenvalue bounds for non-self-adjoint Schr\"odinger operators with non-trapping metrics},  Anal. PDE \textbf{13} (2020), no. 6, 1633--1670. 



\bibitem{Harrach_Lin_Weth_2024}
Harrach, B.,  Lin, Y.-H.,   Weth, T., \emph{The Calder\'on problem for the logarithmic Schr\"odinger equation}, preprint 2024, \textsf{https://arxiv.org/abs/2412.17775}.


\bibitem{Hassell_Tao_Wunsch_2006}
Hassell, A., Tao, T.,  Wunsch, J., \emph{Sharp Strichartz estimates on nontrapping asymptotically conic manifolds}, Amer. J. Math. \textbf{128} (2006), no. 4, 963--1024. 




\bibitem{HLOS_2018} 
Helin, T., Lassas, M., Oksanen, L.,  Saksala, T., \emph{Correlation based passive imaging with a white noise source}, J. Math. Pures Appl. (9) \textbf{116} (2018), 132--160. 


\bibitem{Hormander_1965-66}
H\"ormander, L., \emph{Seminar notes on pseudo-differential operators and boundary problems}, lectures
at IAS Princeton 1965--66, Lund University Publications, \url{https://lup.lub.lu.se/search/ws/files/42232915/Lars_Hormander_Seminar_Notes_IAS_1965_66.pdf}, 1966.





\bibitem{Hormander_book_III}
H\"ormander, L., \emph{The analysis of linear partial differential operators III. Pseudo-differential operators}, Reprint of the 1994 edition. Classics in Mathematics. Springer, Berlin, 2007.

\bibitem{JLX14} 
Jin, T.,  Li, Y.,  Xiong, J., \emph{On a fractional Nirenberg problem, part I: blow up analysis and compactness of solutions}, J. Eur. Math. Soc. (JEMS) \textbf{16} (2014), no. 6, 1111--1171.

\bibitem{K14} Kienzler, C., \emph{Flat fronts and stability for the porous medium equation}, PhD thesis, University of Bonn, 2014. 

\bibitem{KM07}
Kassmann, M.,  Madych, W., \emph{Difference quotients and elliptic mixed boundary value problems of second order}, Indiana Univ. Math. J. \textbf{56} (2007), no. 3, 1047--1082.

\bibitem{KPPV_2024}
Kenig, C., Pilod, D., Ponce, G., Vega, L., \emph{On the fractional Sch\"odinger equation with variable coefficients}, preprint 2024, \textsf{https://arxiv.org/abs/2411.01300}.


\bibitem{Kohn_Vogelius_1983}
Kohn, R., Vogelius, M., \emph{Identification of an unknown conductivity by means of measurements at the boundary}, Proc. SIAM-AMS Symp. on Inverse Problems, New York, 1983.


\bibitem{Lai_Zhou_2023}
Lai, R.-Y., Zhou, T., \emph{An inverse problem for the non-linear fractional magnetic Schr\"odinger equation}, 
J. Differential Equations \textbf{343} (2023), 64--89.



\bibitem{Lassas_Taylor_Uhlmann_2003}
Lassas, M.,  Taylor, M., Uhlmann, G., \emph{The Dirichlet-to-Neumann map for complete Riemannian manifolds with boundary},  Comm. Anal. Geom. \textbf{11} (2003), no. 2, 207--221. 

\bibitem{Lassas_Uhlmann_2001}
Lassas, M.,  Uhlmann, G., \emph{On determining a Riemannian manifold from the Dirichlet-to-Neumann map}, 
Ann. Sci. \'Ecole Norm. Sup. (4) \textbf{34} (2001), no. 5, 771--787. 

\bibitem{Lee_Uhlmann_1989}
Lee, J., Uhlmann, G., \emph{Determining anisotropic real-analytic conductivities by boundary measurements}, Comm. Pure Appl. Math. \textbf{42} (1989), no. 8, 1097--1112.


\bibitem{LPPS_2015}
Leonori, T. Peral, I., Primo, A.,  Soria, F., \emph{Basic estimates for solutions of a class of nonlocal elliptic and parabolic equations},  Discrete Contin. Dyn. Syst. \textbf{35} (2015), no. 12, 6031--6068. 

\bibitem{Lerner_book}
 Lerner, N., \emph{A course on integration theory}, Birkh\"auser/Springer, Basel, 2014.

\bibitem{Li_Li_2020}
Li, L., \emph{The Calder\'on problem for the fractional magnetic operator}, Inverse Problems \textbf{36} (2020), no. 7, 075003.

\bibitem{Li_Li_2021}
Li, L., \emph{Determining the magnetic potential in the fractional magnetic Calder\'on problem}, Comm. Partial Differential Equations \textbf{46} (2021), no. 6, 1017--1026.


\bibitem{Lin_Lin_Uhlmann_2023}
Lin, C.-L., Lin, Y.-H., Uhlmann, G., \emph{The Calder\'on problem for nonlocal parabolic operators: A
new reduction from the nonlocal to the local}, preprint 2023, \textsf{https://arxiv.org/abs/2308.09654}.


\bibitem{Lin_1990}
Lin, F.-H., \emph{A uniqueness theorem for parabolic equations}, Comm. Pure Appl. Math. \textbf{43} (1990), no. 1, 127--136.



\bibitem{Lin_Zimmermann_2024}
Lin, Y.-H., Zimmermann, P., \emph{Approximation and uniqueness results for the nonlocal diffuse optical tomography problem}, preprint 2024, \textsf{https://arxiv.org/abs/2406.06226}.


\bibitem{Lionheart_1997}
Lionheart, W., \emph{Conformal uniqueness results in anisotropic electrical impedance imaging},  
Inverse Problems \textbf{13} (1997), no. 1, 125--134. 

\bibitem{Mazya_book_2012}
Maz'ya, V., \emph{Sobolev spaces with applications to elliptic partial differential equations}.
Second, revised and augmented edition. Grundlehren der mathematischen Wissenschaften, 342. Springer, Heidelberg, 2011.

\bibitem{McLean_book_2000}
McLean, W., \emph{Strongly elliptic systems and boundary integral equations}, Cambridge University Press, Cambridge, 2000.

\bibitem{Nachman_1996}
Nachman, A., \emph{Global uniqueness for a two-dimensional inverse boundary value problem}, Ann.
of Math., \textbf{143} (1996), 71--96. 


\bibitem{N93}
Nekvinda, A., \emph{Characterization of traces of the weighted Sobolev space $ W^{1, p}(\Omega, d_M^\epsilon) $ on $ M$}, Czechoslovak Math. J. \textbf{43} (118) (1993), no. 4, 695--711.

\bibitem{NIST}
Olver, F. W. J., Olde Daalhuis,  A. B., Lozier, D. W., Schneider, B. I., Boisvert, R. F., Clark, C. W., Miller, B. R., Saunders, B. V., Cohl, H. S., and McClain, M. A., eds. 
\emph{NIST Digital Library of Mathematical Functions.}
\url{ https://dlmf.nist.gov/}, Release 1.2.2 of 2024-09-15. 


\bibitem{Porper_Eidelman_1984}
Porper, F., Eidel'man, S., \emph{Two-sided estimates of the fundamental solutions of second-order parabolic equations and some applications of them}, (Russian) Uspekhi Mat. Nauk \textbf{39} (1984), no. 3 (237), 107--156

\bibitem{Quan_Uhlmann}
Quan, H., Uhlmann, G., \emph{The Calder\'on problem for the fractional Dirac operator}, Math. Res. Lett. \textbf{31} (2024), no. 1, 279--302. 


\bibitem{Railo_Zimmermann_2023}
Railo, J.,  Zimmermann, P., \emph{Fractional Calder\'on problems and Poincar\'e inequalities on unbounded domains}, J. Spectr. Theory \textbf{13} (2023), no. 1, 63--131.



\bibitem{Railo_Zimmermann_2024}
Railo, J., Zimmermann, P., \emph{Low regularity theory for the inverse fractional conductivity problem}, 
Nonlinear Anal. \textbf{239} (2024), Paper No. 113418, 27 pp.


\bibitem{Roidos_Shao}
Roidos, N., Shao, Y., \emph{Functional inequalities involving nonlocal operators on complete Riemannian manifolds and their applications to the fractional porous medium equation}, Evolution Equations \& Control Theory, Evol. Equ. Control Theory \textbf{11} (2022), no. 3, 793--825. 

\bibitem{RO16}
Ros-Oton, X., \emph{Nonlocal elliptic equations in bounded domains: a survey}, Publ. Mat. \textbf{60} (2016), no. 1, 3--26.

\bibitem{R15}
R\"uland, A., \emph{Unique continuation for fractional Schrödinger equations with rough potentials}, 
Comm. Partial Differential Equations \textbf{40} (2015), no. 1, 77--114.

\bibitem{Ruland_2021}
R\"uland, A., \emph{On single measurement stability for the fractional Calder\'on problem}, SIAM J. Math. Anal. \textbf{53} (2021), no. 5, 5094--5113.

\bibitem{Ruland_2023}
R\"uland, A., \emph{Revisiting the Anisotropic fractional Calder\'on problem using the Caffarelli-Silvestre extension}, preprint 2023, \textsf{https://arxiv.org/pdf/2309.00858}.

\bibitem{Ruland_Salo_2018}
R\"uland, A., Salo, M., \emph{Exponential instability in the fractional Calder\'on problem}, Inverse Problems \textbf{34} (2018), no. 4, 045003.

\bibitem{Ruland_Salo_2020}
R\"uland, A., Salo, M., \emph{The fractional Calder\'on problem: low regularity and stability},  Nonlinear Anal. \textbf{193} (2020), 111529.

\bibitem{Saloff-Coste_2009}
Saloff-Coste, L., \emph{Sobolev inequalities in familiar and unfamiliar settings},  Sobolev spaces in mathematics. I, 299--343,  Int. Math. Ser. (N. Y.), 8, Springer, New York, 2009. 

\bibitem{Savare}
Savar\'{e}, G., \emph{Regularity results for elliptic equations in Lipschitz domains}, 
J. Funct. Anal. \textbf{152} (1998), no. 1, 176--201.

\bibitem{Seeley_1966}
Seeley, R., \emph{Complex powers of an elliptic operator.}, Proc. Sympos. Pure Math., Amer. Math. Soc., 1997,  288--307.   


\bibitem{Shubin_book_2001}
Shubin, M., \emph{Pseudodifferential operators and spectral theory}, 
Translated from the 1978 Russian original by Stig I. Andersson. Second edition. Springer-Verlag, Berlin, 2001.




\bibitem{Sjostrand_1997}
Sj\"ostrand, J., \emph{A trace formula and review of some estimates for resonances}, Microlocal analysis and spectral theory, 
NATO Adv. Sci. Inst. Ser. C: Math. Phys. Sci., 490, Kluwer Acad. Publ., Dordrecht, 1997, 377--437. 


\bibitem{ST10}
Stinga, P., Torrea, J., \emph{Extension problem and Harnack's inequality for some fractional operators}, 
Comm. Partial Differential Equations \textbf{35} (2010), no. 11, 2092--2122.


\bibitem{Strichartz_1983}
Strichartz, R., \emph{Analysis of the Laplacian on the complete Riemannian manifold}, J. Functional Analysis \textbf{52} (1983), no. 1, 48--79. 

\bibitem{Taylor_book_pseudo}
Taylor, M., \emph{Pseudodifferential operators}, Princeton Mathematical Series, 34. Princeton University Press, Princeton, N.J., 1981.


\bibitem{Uhlmann_2014}
Uhlmann, G., \emph{Inverse problems: seeing the unseen}, Bull. Math. Sci. \textbf{4} (2014), no. 2, 209--279.



\bibitem{Vishik_Eskin_1965}
Vishik, M.,  Eskin, G., \emph{Convolution equations in a bounded region}, Uspekhi Mat. Nauk \textbf{20} (1965), 89--152; English translation in: Russian Math. Surveys \textbf{20} (1965) 86--151.

\bibitem{Yu16}
Yu, H., \emph{Unique continuation for fractional orders of elliptic equations}, Ann. PDE \textbf{3} (2017), no. 2, Paper No. 16, 21 pp.




\end{thebibliography}
\end{document}